\documentclass[twoside]{article}

\usepackage[accepted]{aistats2024}
\usepackage[round]{natbib}

\usepackage{amsmath}
\usepackage{amssymb}
\usepackage{mathtools}
\usepackage{amsthm}
\usepackage{dsfont}
\usepackage{xcolor}         % colors
\usepackage{hyperref}
\usepackage{algorithm}
\usepackage{algorithmic}
\usepackage{pgfplots}
\usepackage{subcaption}
\usepackage{float}
\theoremstyle{plain}
\newtheorem{theorem}{Theorem}[section]
\newtheorem*{theorem*}{Theorem}
\newtheorem{proposition}[theorem]{Proposition}
\newtheorem*{proposition*}{Proposition}
\newtheorem{lemma}[theorem]{Lemma}
\newtheorem*{lemma*}{Lemma}

\newtheorem{remark}[theorem]{Remark}
\theoremstyle{definition}

\definecolor{forestgreen}{rgb}{0.13, 0.55, 0.13}
\DeclareMathOperator*{\argmax}{arg\,max}
\DeclareMathOperator*{\argmin}{arg\,min}

\renewcommand{\epsilon}{\varepsilon}

\hypersetup{
  linkcolor  = blue!70!black,
  citecolor  = blue!70!black,
  urlcolor   = blue!70!black,
  colorlinks = true,
}

\pgfplotsset{compat=1.17}
\begin{document}


%\runningauthor{Surname 1, Surname 2, Surname 3, \ldots., Surname n}

\twocolumn[

\aistatstitle{Efficient Model-Based Concave Utility Reinforcement Learning through Greedy Mirror Descent}

\aistatsauthor{ Bianca Marin Moreno \And Margaux Brégère \And Pierre Gaillard \And Nadia Oudjane }

\aistatsaddress{ Inria THOTH \\ EDF R$\&$D \And Sorbonne Université \\ EDF R$\&$D\And Inria THOTH \And EDF R$\&$D} ]

\begin{abstract}

Many machine learning tasks can be solved by minimizing a convex function of an occupancy measure over the policies that generate them. These include reinforcement learning, imitation learning, among others. This more general paradigm is called the Concave Utility Reinforcement Learning problem (CURL). Since CURL invalidates classical Bellman equations, it requires new algorithms. We introduce MD-CURL, a new algorithm for CURL in a finite horizon Markov decision process. MD-CURL is inspired by mirror descent and uses a non-standard regularization to achieve convergence guarantees and a simple closed-form solution, eliminating the need for computationally expensive projection steps typically found in mirror descent approaches. We then extend CURL to an online learning scenario and present Greedy MD-CURL, a new method adapting MD-CURL to an online, episode-based setting with partially unknown dynamics. Like MD-CURL, the online version Greedy MD-CURL benefits from low computational complexity, while guaranteeing sub-linear or even logarithmic regret, depending on the level of information available on the underlying dynamics.
  
\end{abstract}

\section{Introduction}\label{introduction}
We consider the concave utility reinforcement learning (CURL) problem, which consists on minimizing a convex function (or maximising a concave one) over state-action distributions induced by an agent's policy:
\begin{equation}\label{main_problem}
    \min_{\pi \in (\Delta_\mathcal{A})^{ \mathcal{X} \times N}} \bigg\{F(\mu^{\pi,p}) := \sum_{n=1}^N f_n(\mu_n^{\pi,p})\bigg\}.
\end{equation}
Here, we consider an episodic Markov decision process (MDP) with finite state space $\mathcal{X}$, finite action space $\mathcal{A}$, episodes of length $N$, and probability transition kernel $p := (p_n)_{n \in [N]}$ such that $p_n : \mathcal{X} \times \mathcal{A} \times \mathcal{X} \rightarrow [0,1]$. For all $s \in \mathbb{N}$ we denote $[s] := \{1,\ldots,s\}$.  Letting $\Delta_\mathcal{S}$ be the simplex over a finite set $\mathcal{S}$, we denote $\mu^{\pi,p} := (\mu_n)_{0 \leq n \leq N} \in (\Delta_{\mathcal{X} \times \mathcal{A}})^N$ the state-action distributions over an episode induced by the policy $\pi$ in the MDP with dynamics $p$. 

Many machine learning tasks are special cases of Problem~\eqref{main_problem}. For instance, for the \textbf{reinforcement learning} (RL) task \citep{RL}, $F(\mu^{\pi,p}) := -\langle \mu^{\pi,p}, r \rangle$, i.e. the inner product between the state-action distribution induced by $\pi$ and a reward $r$. For the \textbf{imitation learning} problem \citep{imitation_learning}, $F(\mu^{\pi,p}) := D_f(\mu^{\pi,p}, \mu^*)$, where $D_f$ represents a Bregman divergence induced by a function $f$. For some instances of the \textbf{mean field control} (MFC) problem \citep{MFC}, $F(\mu^{\pi,p}) := -\langle \mu^{\pi,p}, r(\mu^{\pi,p}) \rangle$, where the reward function also depends on the agents' state-action distribution. For \textbf{mean field game} (MFG) problems having the gradient of $F$ as reward, finding a Nash Equilibrium amounts to solving Problem~\eqref{main_problem} \citep{perrolat_curl, pfeiffer}.

\paragraph{Contribution $1$:} We present a new iterative algorithm focusing on solving Problem~\eqref{main_problem} called MD-CURL. It is inspired by the mirror descent algorithm~\citep{MD}, and we prove a convergence rate of order $1/\sqrt{K}$ where $K$ is the number of iterations. Our main new ingredient is the use of a non-standard regularization, which enables us to find both a simple closed-form solution - meaning we avoid the generally costly projection step that mirror descent algorithms undergo - and a convergence proof.  

Until now, there have been few algorithms for solving the general framework of Problem~\eqref{main_problem}. The first two approaches were proposed, by \cite{Hazan_curl} on the basis of the Frank-Wolfe algorithm \citep{Frank-Wolfe}, and \cite{csaba_curl} on the basis of policy gradient methods, both of which have theoretical guarantees. In the mean field community, \cite{perrolat_curl} prove that all algorithms for solving MFGs in discrete-time RL can be applied for solving CURL. \cite{MFG_survey} survey existing algorithms and perform numerical experiments, showing that the adaptation of online mirror descent (OMD) for MFGs presented by \cite{OMD_Perrolat} has the best performance, outperforming the previously mentioned approaches. However, this method has no proof of convergence for discrete iterations. Our proposed algorithm, as we demonstrate with showcase experiments, has the same performance as OMD for MFGs while having theoretical guarantees of convergence.
% Until now, there have been few algorithms for solving the general framework of Problem~\eqref{main_problem}. The first two approaches were proposed, by \cite{Hazan_curl} on the basis of the Frank-Wolfe algorithm \citep{Frank-Wolfe}, and \cite{csaba_curl} on the basis of policy gradient methods. In the mean field community, \cite{perrolat_curl} prove that all algorithms for solving MFGs in discrete-time RL can be applied for solving CURL. \cite{MFG_survey} survey existing algorithms and perform numerical experiments, showing that the adaptation of online mirror descent (OMD) for MFGs presented by \cite{OMD_Perrolat} has the best performance. However, this method has no proof of convergence for discrete iterations. A theoretical understanding of faster algorithms for CURL's general scenario is therefore missing.

\paragraph{Online extension of CURL:}
An interesting extension of Problem~\eqref{main_problem} is the online learning scenario, in which we consider computing a sequence of policies $(\pi^t)_{t \in [T]}$ for $T$ episodes with the objective of minimizing a total loss 
\begin{equation}\label{total_loss}
    \smash{L_T := \sum_{t=1}^T F^t(\mu^{\pi^t,p})},
\end{equation}

where we allow the objective function $F^t$ to change arbitrarily over time (and only be revealed at the end of each episode $t$).

Here, we consider dynamics such that $(x_0, a_0) \sim \mu_0(\cdot)$, and for all steps $n \in [N]$,
\begin{equation}\label{dynamics}
    x_{n+1} := g_n(x_n, a_n, \epsilon_n),
\end{equation}
where $(\epsilon_n)_{n \in [N]}$ is an independent sequence of external noises with $\epsilon_n \sim h_n(\cdot)$ for $h_n$ a distribution. 
% The only property we require of $g_n$ is that it be invertible with respect to $\epsilon_n$, i.e. $\epsilon_n = g_n^{-1}(x_n, a_n, x_{n+1})$. This allows us to recover external noise values by observing an agent's state-action trajectory.

Different variants of this problem can be considered, depending on the prior information available on the dynamics. Here, we consider the case where the agent has prior knowledge of the dynamics ($g_n$ is known), but may be subject to unknown external interference ($h_n$ is unknown). This includes scenarios such as: An energy central controlling the average consumption of electrical appliances. The temperature evolution equation is known, but consumer behavior is unknown and can interfere with the dynamics \citep{ana_busic}. Controlling a fleet of drones in a known environment, subject to external influences due to weather conditions or human intervention. Controlling the state of charge of electric vehicles so that their average consumption follows an energy production target that changes every day and is not known in advance. The dynamics of loading are known, but the arrival and departure of users are not \citep{adrien}.

\paragraph{Contribution $2$:}  We propose Greedy MD-CURL, an online learning algorithm for CURL with dynamics as in Equation~\eqref{dynamics} when $g_n$ is known but the noise distribution $h_n$ is unknown. At each episode, we play a policy $\pi^t$, observe the agent's behavior, update an estimate of the external noise, and use the estimated dynamics to compute the next policy using MD-CURL. Greedy MD-CURL achieves state-of-the-art sub-linear regrets with low complexity and simple closed-form solutions. We further avoid the $\sqrt{|\mathcal{X}|}$ term paid in UCRL approaches (see Section~\ref{related_work}) by showing a weaker control on the difference between the true and estimated probability kernels, being an advantage for models with large state spaces.

Balancing exploration and exploitation is challenging when both $g_n$ and $h_n$ are unknown, which can be computationally expensive. Greedy MD-CURL offers a low-complexity algorithm that achieves sub-linear regret, even without explicit exploration (see Remark~\ref{why_greedy}). Although its regret bound is not the state of the art, Greedy MD-CURL is a good option for scenarios where exploration is already induced by the objective function or by noisy models.

\section{Related Work}\label{related_work}

Online MDPs have mostly been studied in specific cases of CURL rather than in its general form, and draw inspiration from online learning problems \citep{cesa-bianchi_lugosi_2006}. In model-based RL \cite{even_dar_2009} were the first to propose a method dealing with adversarial functions, supposing the transition kernel is fully known in advance. \cite{neu12} extended this work to the case of unknown transition kernels with adversarial rewards, using techniques inspired by UCRL-2 (Upper Confidence Reinforcement Learning) \citep{UCRL-2}. Recently, UC-O-REPS was proposed by \cite{UC-O-REPS}, which extends \cite{zimin} O-REPS algorithm to the case of unknown dynamics and improves upon the regret bound of \cite{neu12}.

In the mean-field community, most approaches for unknown dynamics consider model-free scenarios, such as \cite{Angiuli2020UnifiedRQ, capponi_lehalle_2023, carmona}. M3-UCRL, proposed by \cite{M3-UCRL}, is the only model-based algorithm for mean-field control problems with unknown dynamics. It uses the principle of optimism under uncertainty with UCRL-2 techniques, but only provides regret bounds for Gaussian process dynamics and does not consider online adversarial objective functions.

We introduce the first algorithm for the online CURL problem with theoretical guarantees. Unlike UCRL and PSRL approaches \citep{PS-Osband-2013}, which are generally computationally expensive, our algorithm is nearly greedy and still achieves the same regret bounds with lower computational complexity, depending on the dynamics information available.

\section{General Problem Formulation}\label{problem_formulation}
Consider an episodic Markov decision process (MDP) with finite state space $\mathcal{X}$, finite action space $\mathcal{A}$, episodes of length $N$, and a sequence of transition probabilities ${p:= (p_n)_{n \in [N]}}$ where ${p_n : \mathcal{X} \times \mathcal{A} \times \mathcal{X} \rightarrow [0,1]}$. At time step $n$, an agent in state $x_n$ choosing action $a_n$ transitions to state $x_{n+1}$ with probability $p_{n+1}(x_{n+1}|x_n,a_n)$. At the start of an episode, the agent's first state-action couple follows a fixed distribution $\mu_0 \in \Delta_{ \mathcal{X} \times \mathcal{A}}$. Actions are chosen by means of a policy $\pi_n : \mathcal{X} \rightarrow \Delta_{\mathcal{A}}$ at each time step. In an episode, when an agent follows a sequence of strategies $\pi := (\pi_n)_{n \in [N]}$, we define $\mu^{\pi,p} := (\mu_n^{\pi,p})_{0 \leq n \leq N}$ the state-action distribution sequence induced by the policy $\pi$ in the MDP with probability kernel $p$ recursively for all $(x',a') \in \mathcal{X} \times \mathcal{A}$ and all $n \in [N]$:
\begin{align}\label{mu_induced_pi}
         \mu_0^{\pi,p}(x',a') &:= \mu_0(x',a') \\
         \mu_{n}^{\pi,p} (x',a') &:=\sum_{x\in\mathcal{X}} \sum_{a \in\mathcal{A}} \mu_{n-1}^{\pi,p}(x,a) p_{n}(x'| x, a) \pi_{n}(a' | x'). \nonumber
\end{align}
We let $\|\cdot\|_1$ be the $L_1$ norm, and for all ${v := (v_n)_{n \in [N]}}$, such that ${v_n \in \mathbb{R}^{\mathcal{X} \times \mathcal{A}}}$ we define ${\|v\|_{\infty, 1} := \sup_{0 \leq n \leq N} \|v_n\|_1}$. We define the objective function ${F(\mu) := \sum_{n=1}^N f_n(\mu_n)}$ where ${f_n: \Delta_{\mathcal{X} \times \mathcal{A}} \rightarrow \mathbb{R}}$ are convex and $\ell$-Lipschitz functions with respect to the norm $\|\cdot\|_1$.
% In this paper we study the CURL Problem~\eqref{main_problem}, which is an optimization problem, and an online learning extension of this problem, where the dynamics are given by Equation~\eqref{dynamics} and partially unknown, and the objective function can change arbitrarily over episodes.

\paragraph{Offline optimization setting (Section~\ref{algorithm_knownp})}
To solve the CURL problem, we propose a learning protocol that consists in following an iterative method. At each iteration $k \in [K]$, the learner computes a new policy by solving an auxiliary optimization problem. This auxiliary optimization problem, that we denote by $\mathfrak{F}$, depends on the previous policy $\pi^{k-1}$, the model dynamics $p$, and the objective function $F$, i.e. $\smash{\pi^{k} := \mathfrak{F}(\pi^{k-1}, p, F)}$. In Section~\ref{algorithm_knownp}, we show how to construct $\mathfrak{F}$ such that $\min_{k \in [K]} F(\mu^{\pi^k,p}) - F(\mu^{\pi^*,p})$ is bounded by a term of order $1/\sqrt{K}$, with $K$ the number of iterations, where $\pi^*$ is an optimal policy.

\paragraph{Online learning setting (Section~\ref{algorithm_unknownp})}
In the online extension of CURL, the objective function at episode $t \in [T]$ is denoted as $F^t := \sum_{n=1}^N f_n^t$, where $T$ is the total number of episodes, and $f_n^t: \Delta_{\mathcal{X} \times \mathcal{A}} \rightarrow \mathbb{R}$. We assume that $f_n^t$ is convex and $\ell$-Lipschitz with respect to the $\|\cdot\|_1$ norm. The functions $F^t$ are only revealed to the learner at the end of episode $t$. The learner's objective is to compute a sequence of strategies $(\pi^t)_{t \in[T]}$ minimizing their total loss defined in Equation~\eqref{total_loss}, and the learner's performance is measured by comparison to the best stationary policy, using the following regret:
\begin{equation}\label{regret}
    R_T := \sum_{t=1}^T F^t(\mu^{\pi^t, p}) - \min_{\pi \in (\Delta_{\mathcal{A}})^{ \mathcal{X} \times N}} \sum_{t=1}^T F^t(\mu^{\pi,p}).
\end{equation}

We consider the dynamics of Equation~\eqref{dynamics} when $g_n$ is known but $h_n$ is unknown. In order to choose a sequence of policies that minimize their total loss, the learner must then both optimize the objective function and learn the noise distribution through observations. The learner's online protocol is in Algorithm~\ref{alg:online_protocol}. 

At each episode $t$, the learner chooses a policy $\pi^t$, send it to $M$ independent agents, observes the external noise $(\epsilon_1^{j,t}, \ldots, \epsilon_N^{j,t})$ for each agent $j \in [M]$ over all $N$ steps (to retrieve the noises, it is enough to observe the agent's trajectory and for $g_n$ to be invertible), computes an estimate  $\hat{p}^{t+1}$ of the probability kernel using the observations, observes the objective function $F^t$, and calculates the policy for the next episode by applying the auxiliary problem $\mathfrak{F}$ on $\pi^{t}, F^t,$ and $\hat{p}^{t+1}$. 

To compute a strategy sequence with sub-linear regret the learner faces two challenges: how to estimate $\hat{p}^{t}$ from the data and how to define the auxiliary optimization problem $\mathfrak{F}$. In Section~\ref{algorithm_unknownp}, we show that by considering the same auxiliary optimization problem $\mathfrak{F}$ as in the optimization of the offline CURL problem, and by taking $\hat{p}^{t+1}$ as the empirical mean estimator, we can build an algorithm that achieves sub-linear regret. This result is studied in detail in Section~\ref{algorithm_unknownp}. Observing $M$ independent agents following the same policy is relevant to many applications, such as controlling the charging state of a set of electric vehicles. This allows us to explicitly express the dependence on $M$. Note that if we set $M=1$, we recover the standard case.

\begin{algorithm}[ht]
\caption{Learner's Online Protocol}\label{alg:online_protocol}
\begin{algorithmic}
\STATE {\bfseries Input:} initial state-action distribution $\mu_0$, initial strategy sequence $\pi^1$.
    \FOR{$t= 1,\ldots,T$}
    \FOR{$j =1, \ldots, M$}
    \STATE the $j$-th agent playing episode $t$ starts at $(x_0^{j,t}, a_0^{j,t}) \sim \mu_0(\cdot)$    
    \FOR{$n = 1, \ldots,N$}
        \STATE environment draws new state $x_{n}^{j,t} \sim p_{n}(\cdot|x_{n-1}^{j,t}, a_{n-1}^{j,t}) $
        \STATE learner observes agent's $j$ external noise $\epsilon_n^{j,t}$
        \STATE agent $j$ chooses an action $a_n^{j,t} \sim \pi^{t}_n(\cdot|x_n^{j,t})$
    \ENDFOR
    \ENDFOR
    \STATE learner computes, for all $n \in [N]$, new estimate $\hat{p}^{t+1}_n$ from data $(\epsilon_n^{j,s})_{s \in [t], j \in [M]}$
    \STATE objective function $F^t$ is exposed
    \STATE learner computes $\pi^{t+1} = \mathfrak{F}(\pi^{t}, \hat{p}^{t+1}, F^t)$ and send to all agents
    \ENDFOR
\STATE {\bfseries return} $\pi^T$
\end{algorithmic}
\end{algorithm}

% \begin{remark}
%     The online CURL problem is a generalization of the original CURL problem. If the probability transition were known, we wouldn't need to observe the state-action pairs at each episode, and we could just use the auxiliary optimization problem in $\pi^t$, $F^t$ and $p$. 
% \end{remark}
% \begin{remark}
%     In the case of a multi-agent problem with homogeneous agents, it can be assumed that more than one agent is observed at each episode, and that the true probability transition kernel is learned faster.
% \end{remark}

\section{CURL as an optimisation problem}\label{algorithm_knownp}
% We present a new approach and its theoretical guarantees for dealing with the optimization Problem~\eqref{main_problem}.
\subsection{Reformulation of learner's objective}\label{problem_reformulation}
The CURL problem as presented in Equation~\eqref{main_problem} is problematic in that it is not convex in $\pi$, and calculating the gradient of $F$ with respect to the strategy $\pi$ can be intractable. Therefore, we reformulate the learner's objective to obtain a convex problem. We define
 \begin{align}\label{set_convex_mu}
        &\mathcal{M}^p_{\mu_0} :=  \bigg\{ \mu \in 
        (\Delta_{\mathcal{X} \times \mathcal{A}})^N  \big| \; \sum_{a' \in \mathcal{A}} \mu_{n}(x',a') = \\
        &\quad \sum_{x \in \mathcal{X} , a \in \mathcal{A}} p_{n}(x'|x,a) \mu_{n-1}(x,a)\;, \forall x' \in\mathcal{X}, 
         \forall n \in [N] \bigg\}, \nonumber
 \end{align}
 as the set of state-action distribution sequences satisfying the Bellman-flow in the MDP with transition kernel $p$ and initial state-action distribution $\mu_0$. For now, we assume that the probability kernel $p$ is known and, to minimize notations, we let $\mu^{\pi} :=\mu^{\pi, p}$ and $\mathcal{M}_{\mu_0} := \mathcal{M}^p_{\mu_0}.$ We also assume $\mu_0$ is always known.

 For any $\mu \in \mathcal{M}_{\mu_0}^p$, there exists a strategy $\pi$ such that $\mu^{\pi} = \mu$. It suffices to take $\pi_n(a|x) \propto \mu_n(x,a)$
 %}{\sum_{a \in \mathcal{A}} \mu_n(x,a)}$ 
 when the normalization factor is non-zero, and arbitrarily defined otherwise. This result is formally enunciated and proved in Proposition~\ref{opt_mu_equal_pi} in Appendix~\ref{missing_results} (see also \cite{putterman}). We therefore have the equivalence 
 \begin{equation*}
     \min_{\pi \in (\Delta_\mathcal{A})^{\mathcal{X} \times N}} F(\mu^{\pi}) \equiv \min_{\mu \in \mathcal{M}_{\mu_0}} F(\mu).
 \end{equation*}
 Note that the optimization problem over $\mu$ is convex. 

 \subsection{The Algorithm}
To build an algorithm solving the CURL problem, we need to build the auxiliary optimization problem $\mathfrak{F}$ discussed in Section~\ref{problem_formulation}. Let $\smash{\mathcal{M}_{\mu_0}^*}$ denote the subset of $\smash{\mathcal{M}_{\mu_0}}$ where the corresponding policies $\pi$ are such that $\pi_n(a|x) \neq 0$ for all $(x,a) \in \mathcal{X} \times \mathcal{A}$, $n \in [N]$. We define a regularization function $\smash{\Gamma : \mathcal{M}_{\mu_0} \times \mathcal{M}_{\mu_0}^* \to \mathbb{R}}$ as
\begin{equation}\label{gamma}
       \Gamma(\mu^\pi, \mu^{\pi'}) := \sum_{n=1}^{N} \mathbb{E}_{(x,a) \sim \mu^{\pi}_n(\cdot)}\bigg[\log\bigg(\frac{\pi_{n}(a|x)}{\pi'_{n}(a|x)}\bigg)\bigg],
\end{equation}
that is well defined thanks to the bijection between strategies and state-action distributions satisfying the Bellman flow (see Proposition~\ref{opt_mu_equal_pi}). We define the following iterative scheme with $\tau_k > 0$ and $\smash{\langle \nabla F(\mu^k),\mu^\pi\rangle := \sum_{n=1}^N \langle \nabla f_n(\mu_n^k), \mu_n^\pi \rangle}$:
\begin{equation}\label{MD_iterative_scheme}
\begin{split}
\mu^{k+1} \in \argmin_{\mu^\pi\in\mathcal{M}_{\mu_0}}
\bigg\{
\langle \nabla F(\mu^k),\mu^\pi\rangle 
+\frac{1}{\tau_k}  \Gamma(\mu^\pi, \mu^k)
\bigg\},
\end{split}
\end{equation}
where the idea is, at iteration $k+1$, to choose $\mu^\pi$ minimizing a linearization of the objective function around $\mu^k$, the distribution sequence found at the previous iteration, and at the same time penalizing the distance between policy $\pi$ inducing $\mu^\pi$ and $\pi^k$ inducing $\mu^k$. Our first main result of this section is in Theorem~\ref{MD_explicit_result}. It shows that, due to the choice of penalizing strategies, the iterative scheme in Equation~\eqref{MD_iterative_scheme} can be solved through dynamic programming \citep{DP_principles} by building a Bellman recursion:

\begin{theorem}\label{MD_explicit_result}
Let $k \geq 0$. The solution of Problem~\eqref{MD_iterative_scheme} is $\mu^{k+1} = \mu^{\pi^{k+1}}$, where for all $n \in [N]$, and $(x,a) \in \mathcal{X} \times \mathcal{A}$, 
\begin{equation}\label{policy_update}
\textstyle{
    \pi_{n}^{k+1}(a|x) := \frac{\pi_{n}^{k}(a|x) \exp\left(\tau_k \tilde{Q}_{n}^{k}(x,a) \right)}{\sum_{a' \in \mathcal{A}}\pi_{n}^{k}(a'|x) \exp\left(\tau_k \tilde{Q}_{n}^{k}(x,a') \right)},}
\end{equation}
where $\tilde{Q}$ is a regularized $Q$-function satisfying the following recursion
\begin{align}\label{Bellman_Q_tilde}
    % \begin{cases}
        &\Tilde{Q}^k_N(x,a) = -\nabla f_N(\mu_N^k)(x,a)  \nonumber \\
        % \!\begin{aligned}
            &\Tilde{Q}^k_n(x,a) = \max_{\pi_{n+1} \in (\Delta_\mathcal{A})^{\mathcal{X}}} \Bigg\{ -\nabla f_n(\mu_n^k)(x,a) + \nonumber\\
            & \qquad \sum_{x'} p_{n+1}(x'|x,a) \sum_{a'} \pi_{n+ 1}(a'|x')  \\
            & \qquad \bigg[  - \frac{1}{\tau_k}\log\left( \frac{\pi_{n+1}(a'|x')}{\pi_{n+1}^k(a'|x')}\right) 
             + \Tilde{Q}^k_{n+1}(x',a') \bigg] \Bigg\}. \nonumber
        % \end{aligned}
    % \end{cases}
\end{align}
\end{theorem}
\begin{proof}
See Appendix~\ref{thm_1_proof}.
\end{proof}

\begin{algorithm}[ht]
\caption{MD-CURL}\label{alg:MD}
\begin{algorithmic}[1]
\STATE {\bfseries Input:} number of iterations $K$, initial sequence of policies $\smash{\pi^0 \in (\Delta_\mathcal{A})^{\mathcal{X} \times N}}$ such that $\smash{\mu^0 := \mu^{\pi_0} \in \mathcal{M}_{\mu_0}^*}$, objective function $F:= \sum_{n=1}^N f_n$, probability kernel $p = (p_n)_{n \in [N]}$, initial state-action  distribution $\mu_0$, sequence of non-negative learning rates $(\tau_k)_{k \leq K}$.
    \FOR{$k= 0,\ldots,K-1$} 
    \STATE $\mu^k = \mu^{\pi^k}$ as in Equation~\eqref{mu_induced_pi} \label{line:mu_induced_pi}
    \STATE $\tilde{Q}_N^{k}(x,a) = -\nabla f_N(\mu_N^k)(x,a)$, $\forall (x,a) \in \mathcal{X} \times \mathcal{A}$ 
    \FOR{$n = N,\ldots,1$}
    \STATE $\forall (x,a) \in \mathcal{X} \times \mathcal{A}:$
    \STATE $\pi_n^{k+1}(a|x)=  \frac{\pi_{n}^{k}(a|x) \exp\left(\tau_k \tilde{Q}_{n}^{k}(x,a) \right)}{\sum_{a'}\pi_{n}^{k}(a'|x) \exp\left(\tau_k \tilde{Q}_{n}^{k}(x,a') \right)}$ \label{line:exp_twist}
    \STATE $\tilde{Q}_{n-1}^{k}(x,a)$ using the recursion in Equation~\eqref{Bellman_Q_tilde} \label{line:compute_Q}
    \ENDFOR
    \ENDFOR
\STATE {\bfseries return} $\pi^K$
\end{algorithmic}
\end{algorithm}

It is not obvious at first sight, but we can show that $\Gamma$ is a Bregman divergence, making the iterative scheme an instance of mirror descent \citep{MD}. Therefore, we can state the convergence result of Algorithm~\ref{alg:MD}, MD-CURL, in Theorem~\ref{thm:convergence_rate}, the second main result of this section. Solving a mirror descent (MD) instance usually includes a projection step that is generally computationally expensive. The low-complexity methods existing in the literature can only offer approximate solutions \citep{MD_changing_costs, UC-O-REPS}. We show that with the judicious choice of divergence as in Equation~\eqref{MD_iterative_scheme}, MD can be solved accurately avoiding all costly projection steps. 

\begin{theorem}\label{thm:convergence_rate}
Let $\pi^*$ be a minimizer of Problem~\eqref{main_problem}. Define $L := \ell N$ where $\ell$ is the Lipschitz constant of $f_n$ with respect to $\|\cdot\|_1$ for all $n \in [N]$. Applying $K$ iterations of MD-CURL to this problem, with, for each $1 \leq k \leq K$,
$\tau_k := L^{-1} \sqrt{2 \Gamma(\mu^{\pi^*},\mu^0)/K},$
gives the following convergence rate
$$ \smash{\min_{0 \leq k \leq K} F(\mu^{\pi^k}) - F(\mu^{\pi^*}) \leq L \frac{\sqrt{2 \Gamma(\mu^{\pi^*},\mu^0)} }{\sqrt{K}}}.$$
\end{theorem}
\begin{proof} 
For ease of notation, for any probability measure $\eta \in \Delta_E$, whatever the (finite) space $E$, we introduce the neg-entropy function, with the convention that $0 \log(0) = 0$, $\phi (\eta):=\sum_{x\in E} \eta (x)\log \eta (x)$.
 
\begin{proposition}\label{penalization_is_bregman}
    Let $\mu \in \mathcal{M}_{\mu_0}$ with marginal given by $\rho \in (\Delta_\mathcal{X})^N$. The divergence $\Gamma$ is a Bregman divergence induced by 
    \begin{equation*}
     \smash{\psi(\mu) := \sum_{n=1}^N \phi(\mu_n) - \sum_{n=1}^N \phi(\rho_n)}.
    \end{equation*}
    Also, $\psi$ is $1$-strongly convex with respect to $\|\cdot\|_{\infty,1}$.
\end{proposition}
 % \vspace*{-10pt}
The proof is in Appendix~\ref{penalization_is_bregman_proof} and consists in showing that the $\Gamma$ divergence taking values on the sequence of state-action distributions is in fact the KL divergence on the joint distribution. Next, if $f_n$ is convex and $\ell$-Lipschitz with respect to the norm $\|\cdot\|_1$ for all $n \in [N]$, then $F$ is also convex and Lipschitz with constant $\smash{L := \ell N}$ with respect to the norm $\|\cdot\|_{\infty,1}$ (see Appendix~\ref{proof_conv_MD_MFC}). Since the set $\mathcal{M}_{\mu_0}$ is convex, all convergence assumptions of MD \citep{MD} are satisfied, and the rate of convergence follows.
\end{proof}

\section{Online learning extension of CURL}\label{algorithm_unknownp}
We consider here the online variant of Problem~\eqref{main_problem}, where the learner must compute a sequence of strategies while facing unknown external noise and arbitrarily changing objective functions. We introduce Greedy MD-CURL, a new algorithm achieving sub-linear regret with a simple closed-form solution. At episode $t$, Greedy MD-CURL solves an optimization problem in the MDP induced by the estimated probability kernel $\hat{p}^t$ using one iteration of MD-CURL. We refer to $p$ as the true probability kernel and $\hat{p}^t$ as the estimated one.

\subsection{Learning the model}\label{learning_the_model}

Since the learner does not know the noise dynamics, it has to estimate it from its experience. To obtain a sub-linear regret, the learner must learn $\hat{p}^t$ in such a way that its distance to the real probability kernel decreases with $t$ with high probability. Let us denote $M_n^t$ the number of times the learner observes step $n$ until the start of episode $t$, and $\epsilon_n^{s}$ the $s$-th noise observed at step $n$. Recall that the dynamics follow Equation~\eqref{dynamics}, and that the learner observes the noise values from the agent's trajectory. Let $\delta_x$ be the Dirac distribution centered in $x$. We define
\begin{equation}\label{p_estimation}
    \smash{\hat{p}_n^t(\cdot| x,a) := \frac{1}{M_n^t} \sum_{s=1}^{M^t_n} \delta_{g_n(x,a, \epsilon_n^{s})}(\cdot)}.
\end{equation}

For any function $\Lambda: \mathcal{X} \rightarrow \mathbb{R}$, for all $n \in [N]$ and $(x,a) \in \mathcal{X} \times \mathcal{A}$ we introduce the notation
\begin{equation*}
    \big(p_n - \hat{p}_n^t \big) \big(\Lambda \big) (x,a) := \sum_{x' \in \mathcal{X}} \big( p_n(x'|x,a) - \hat{p}_n^t(x'|x,a) \big) \Lambda(x').
\end{equation*}
We have the following concentration result.
\begin{lemma}\label{concentration_inner_mu}
    Let $\gamma > 0$. For any $0 < \delta < 1$ and any function $\Lambda: \mathcal{X} \rightarrow \mathbb{R}$ such that $|\Lambda(x')| \leq \sqrt{\gamma}/2$ for all $x' \in \mathcal{X}$,
    \begin{equation*}
         \smash{\big(p_n - \hat{p}_n^t \big) \big(\Lambda \big) (x,a) \leq \sqrt{\frac{\gamma}{2 M_n^t} \log\bigg(\frac{N |\mathcal{X}| |\mathcal{A}| T}{\delta} \bigg)}}
    \end{equation*}
    holds with probability $1-\delta$ simultaneously for all $(x,a) \in \mathcal{X} \times \mathcal{A}$, steps $n \in [N]$, and episodes $t \in [T]$.
\end{lemma}
\begin{proof}
    See Appendix~\ref{proof:concentration_inner_mu}.
\end{proof}

In the literature \citep{UCRL-2, UC-O-REPS}, it is common to bound the $L_1$ deviation between $p$ and $\hat{p}^t$ instead. However, this deteriorates the bound by an additional factor of $\sqrt{|\mathcal{X}|}$, which we can avoid here because of our dynamics hypothesis. This means that in the final regret analysis, we only pay the number of states in a term proportional to $\sqrt{\log(|\mathcal{X}|)}$, which is an advantage for problems with large state spaces or even to discretize continuous state space problems.

We further state Lemma~\ref{lemma:weak_bound_mu}, which is proven in Appendix~\ref{proof:weak_bound_mu} and used later to prove the regret bound of Greedy MD-CURL.

\begin{lemma}\label{lemma:weak_bound_mu}
For any vector $v \in \mathbb{R}^{|\mathcal{X}| \times |\mathcal{A}|}$, for any strategy $\pi$ and for all $0 \leq n \leq N$,
    \begin{equation*}
        \begin{split}
            &\langle v,  \mu_n^{\pi, p} - \mu_n^{\pi, \hat{p}^t} \rangle = \sum_{i=1}^{n} \sum_{y \in \mathcal{X} \times \mathcal{A}} \mu_{i-1}^{\pi, \hat{p}^t}(y) \big(p_i- \hat{p}_i^t \big) (\Lambda^{i,n, \pi}_v)(y),
        \end{split}
    \end{equation*}
    where $\Lambda^{i,n, \pi}_v :\mathcal{X} \rightarrow \mathbb{R}$ is a function depending on $v, i, n$ and $\pi$ defined in Equation~\eqref{def_Lambda}. Also, if $\|v\|_\infty := \sup_{(x,a) \in \mathcal{X} \times \mathcal{A}} |v(x,a)| \leq V$, then $\|\Lambda_v^{i,n, \pi}\|_\infty \leq V$.
\end{lemma}

% In some specific cases, it is possible to improve this bound. For example, if the noise follows a categorical distribution with $|I|$ possible values, instead of paying $\sqrt{|\mathcal{X}| }$, we would pay $\sqrt{|I|}$, for more details see Appendix~\ref{categorical}.  Real-world examples of external noises following categorical distributions include the electric vehicle charging control problem presented in Section~\ref{introduction}. In this problem, the external noise has only two possible values: the arrival/departure of an user. In these cases, $|I|<<|X|$ as $\mathcal{X}$ contains all the possible charging states of the vehicle.

\subsection{Optimization problem}
Recall that the learner follows the online protocol in Algorithm~\ref{alg:online_protocol}. At each episode, the learner estimates $\hat{p}^t$ from the noise observations using Equation~\eqref{p_estimation}, We denote by $\mathcal{M}_{\mu_0}^{t} := \mathcal{M}_{\mu_0}^{\hat{p}^t}$ the set induced by this estimate (as in Equation~\eqref{set_convex_mu}). At every episode the learner solves
\begin{equation}\label{MD_iteration_unkonwnp}
    \mu^{t+1} \in \argmin_{\mu \in \mathcal{M}_{\mu_0}^{t+1}} \{ \tau \langle \nabla F^t(\mu^t), \mu \rangle + \Gamma( \mu, \tilde{\mu}^t) \}, \vspace*{-6pt}
\end{equation}
where, $\mu^t := \mu^{\pi^t, \hat{p}^t}$ and $\tilde{\mu}^t := \mu^{\tilde{\pi}^t,\hat{p}^t}$ with \vspace*{-5pt}
\begin{equation}\label{tilde_pi}
    \smash{\tilde{\pi}^t := (1-\alpha_t) \pi^t +  \textstyle{\frac{\alpha_t}{|\mathcal{A}|}},} \vspace{-2pt}
\end{equation}
and $\alpha_t \in (0,1/2)$ is an exploration parameter. 
% Let us denote by $\pi^t$ the policy inducing $\mu^t$ in the MDP with transition kernel $\hat{p}^t$ for all $t \in [T]$ (as in Equation~\eqref{mu_induced_pi}). 
% Note that $\tilde{\pi}^t_n(a|x) \geq \frac{\alpha_t}{|\mathcal{A}|}$. Therefore, as $|\mathcal{A}| \geq 1$, $|\log(\tilde{\pi}^t_n(a|x))| \leq \log(|\mathcal{A}|/\alpha_t)$. This will be useful later in the regret analysis.

In Theorem~\ref{MD_explicit_result}, we have already shown that the optimization problem of Equation~\eqref{MD_iteration_unkonwnp} with Bregman divergence $\Gamma$ has the format of an exponential twist as in Equation~\eqref{policy_update}. Consequently, we can build Greedy MD-CURL in Algorithm~\ref{alg:greedy_MD}. Note that to compute the policy for episode $t+1$, we perform one iteration of MD-CURL using $\pi^t$ to compute $\mu^t$ as in line~\ref{line:mu_induced_pi} of Algorithm~\ref{alg:MD}, $\tilde{\pi}^t$ to compute the exponential twist in line~\ref{line:exp_twist} and to compute $\tilde{Q}$ recursively in line~\ref{line:compute_Q}, the objective function $F^t$ and the estimated probability kernel $\hat{p}^{{t+1}}$.

\begin{remark}\label{why_greedy}
    We call our algorithm Greedy because it solves the optimization problem~\eqref{MD_iteration_unkonwnp} at each episode using the empirically estimated dynamics~\eqref{p_estimation} as if they were the true ones, without confidence intervals or exploration bonuses related to visit counts as usually is the case \citep{UCRL-2, UC-O-REPS, ucbvi}.
\end{remark}

\begin{algorithm}
\caption{Greedy MD-CURL}\label{alg:greedy_MD}
\begin{algorithmic}
\STATE {\bfseries Input:} number of episodes $T$, initial sequence of policies $\smash{\pi^1 \in (\Delta_\mathcal{A})^{\mathcal{X} \times N}}$, number of observations per episode $M$, initial state-action  distribution $\mu_0$, learning rate $\tau > 0$, sequence of parameters $(\alpha_t)_{t \in [T]}.$
\STATE {\bfseries Initialization:} \quad $\forall (x,a)$, $p^1(\cdot|x,a) = \frac{1}{|\mathcal{X}|}$  
    \FOR{$t= 1,\ldots,T$} 
    \FOR{$j=1, \ldots, M$}
    \STATE $j$-th agent starts at $(x_0^{j,t}, a_0^{j,t}) \sim \mu_0(\cdot)$
    \FOR{$n = 1, \ldots, N$}
    \STATE environment draws new state $x_{n}^{j,t} \sim p_n(\cdot|x_{n-1}^{j,t}, a_{n-1}^{j,t})$
    \STATE learner observes agent $j$'s external noise $\epsilon^{j,t}_n$ 
    \STATE  agent $j$ chooses an action $a_n^{j,t} \sim \pi_n^t(\cdot|x_{n}^{j,t})$
    \ENDFOR
    \ENDFOR
    \STATE update probability kernel estimate for all $(x,a)$:
    \STATE $\displaystyle{\hat{p}^{t+1}_n(\cdot|x,a) := \frac{1}{Mt} \sum_{j=1}^M \delta_{g_n(x,a,\epsilon_n^{j,t})} + \frac{t-1}{t} \hat{p}^{t}_n(\cdot|x,a)}$ 
    \STATE compute policy for the next episode:
    \STATE $\quad \pi^{t+1} := \text{MD-CURL}(1, \pi^t \backslash \tilde{\pi}^t, F^t, \hat{p}^{t+1}, \mu_0, \tau)$
    \STATE compute $\tilde{\pi}^{t+1}$ as in Equation~\eqref{tilde_pi}
    \ENDFOR
\STATE {\bfseries return} $(\pi^t)_{t \in [T]}$
\end{algorithmic}
\end{algorithm}

\subsection{Regret analysis}\label{analysis}
In this section, we prove the regret bound of Greedy MD-CURL. For that, we use the results from Subsection~\ref{learning_the_model} and some results of OMD \citep{OMD}, while also having to handle an online optimization problem with varying constraint sets. We decompose the regret~\eqref{regret} into three terms, 
\begin{equation*}
    \begin{split}
R_T &= \sum_{t=1}^T  F^t(\mu^{\pi^t,p}) - F^t(\mu^{\pi^t, \hat{p}^t}) \\ &+\sum_{t=1}^T F^t(\mu^{\pi^t, \hat{p}^t}) - F^t(\mu^{\pi^*, \hat{p}^{t+1}}) \\
&+ \sum_{t=1}^T F^t(\mu^{\pi^*, \hat{p}^{t+1}}) - F^t(\mu^{\pi^*,p}) \\
&:= R_T^{MDP}((\pi^t)_{t \in [T]}) + R_T^{policy} + R_T^{MDP}(\pi^*)\,, 
    \end{split}
\end{equation*}
where ${\pi^* := \argmin_{\pi \in (\Delta_{\mathcal{A}})^{\mathcal{X} \times N}} \sum_{t=1}^T F^t(\mu^{\pi,p})}$. The terms $R_T^{MDP}((\pi^t)_{t \in [T]})$ and $R_T^{MDP}(\pi^*)$ pay for the error due to not knowing the true probability kernel, and the term $R_T^{policy}$ pays for calculating sub-optimal policies using MD-CURL with constraint sets varying with each episode. Propositions~\ref{prop:bound_regret_proba} and \ref{thm:bound_regret_md} bound each of these terms, yielding our main result:

\begin{theorem}\label{main_result}
Consider an episodic MDP with finite state space $\mathcal{X}$, finite action space $\mathcal{A}$, episodes of length $N$, and probability kernel $p:=(p_n)_{n \in [N]}$. Let $F^t := \sum_{n=1}^N f_n^t$ convex with $f_n^t$ $\ell$-Lipschitz with respect to the norm $\|\cdot\|_1$ for all $n \in [N], t \in [T]$. Let
\begin{equation}\label{auxiliary_b}
\begin{split}
         b &:= \smash{\textstyle{ \big( \sum_{t=1}^T 2 \big[ N \alpha_t
        + \frac{N^2}{t} \log\big(\frac{|\mathcal{A}|}{\alpha_t}\big) + N^2\big(\frac{1}{t} + \alpha_t \big)^2\big] }} \\
        & \qquad\qquad   \smash{\textstyle{ + \big(N \log(|\mathcal{A}|)\big) \big)^{\frac{1}{2}}}}
\end{split}
\end{equation}
Then, with probability $1-\delta$, Greedy MD-CURL obtains, for $\tau = \frac{b}{L \sqrt{T}}$,
    \begin{equation*}
                \smash{R_T \leq  2 \ell N b \sqrt{T} + 2 \ell N^2 \sqrt{\frac{2 T}{M} \log\bigg(\frac{N |\mathcal{X}| |\mathcal{A}| T}{\delta}\bigg)}}.
    \end{equation*}
\end{theorem}

In particular, choosing $\alpha_t =T^{-1}$ for all $t \in [T]$, yields $R_T = O(\sqrt{T} \log(T))$.

%$= O\big(L N^2 \big(\frac{T}{M}\big)^{1/2} \log^{1/2}\big({N |\mathcal{X}| |\mathcal{A}| T}/\delta\big) \big)}$.

% One advantage of our approach is that, when analyzing $R_T^{MDP}$ terms, we do not bound the $L_1$ norm of the difference between $p$ and $\hat{p}^t$ but use the result of Lemma~\ref{concentration_inner_mu} instead. This means we do not pay the usual $\sqrt{|\mathcal{X}|}$ term.

\subsubsection{Bounding $R_T^{MDP}$}
Here we show the bounds on $R_T^{MDP}\big((\pi^t)_{t \in [T]}\big)$ and $R_T^{MDP}(\pi^*)$. Both indicate the difference between the loss of playing a sequence of policies over $T$ episodes in the actual MDP and the loss of playing the same sequence of policies but in the estimated MDP. For the first term, the sequence is that produced by Greedy MD-CURL, i.e. $(\pi^t)_{t \in [T]}$, and for the second term, it is the best stationary policy over the horizon $T$, i.e. $\pi^*$. The results are presented in Proposition~\ref{prop:bound_regret_proba} and use the lemmas from Subsection~\ref{learning_the_model}.

\begin{proposition}\label{prop:bound_regret_proba}
Under the same hypothesis as in Theorem~\ref{main_result}, with probability $1-\delta$, Greedy MD-CURL obtains,
    \begin{equation*}
    R_T^{MDP}\big((\pi^t)_{t \in [T]}\big) \leq \ell N^2 \sqrt{\frac{2 T}{M} \log\bigg(\frac{N |\mathcal{X}| |\mathcal{A}| T}{\delta}\bigg)}. \\
    \end{equation*}
    The exact same result being also valid for $R_T^{MDP}\big(\pi^*\big)$.
\end{proposition}
\begin{proof}
    % The proof uses the results from Subsection~\ref{learning_the_model} and can be found in Appendix~\ref{proof:bound_regret_proba}.
    See Appendix~\ref{proof:bound_regret_proba}.
\end{proof}

% \begin{proof}
% The proof steps are the same for both terms, hence we show only the steps for $ R_T^{MDP}\big((\pi^t)_{t \in [T]}\big) $. Using the convexity of $F^t$ we obtain
% \begin{equation*}
% \begin{split}
%     R_T^{MDP}\big((\pi^t)_{t \in [T]}\big) &\leq \sum_{t=1}^T  \sum_{n=1}^N \langle \nabla f_n^t(\mu^{\pi^t,p}), \mu^{\pi^t,p}_n - \mu^{\pi^t,\hat{p}^t}_n\rangle.
% \end{split}
% \end{equation*}

% To bound the inner product for each $n$, we first use the result of Lemma~\ref{lemma:weak_bound_mu}. Then, we show that the function $\Lambda_{\nabla f_n^t(\mu^{\pi^t,p})}$ found is bounded by $l$, the Lipschitz constant of $f^t_n$. Therefore, it satisfies Lemma~\ref{concentration_inner_mu} with $\gamma = 4 l^2$, which allows us to conclude. The detailed proof is in Appendix~\ref{proof:bound_regret_proba}. 
% \end{proof}

% Note that in the example of categorical distributions detailed in Appendix~\ref{categorical}, $ R_T^{MDP}\big((\pi^t)_{t \in [T]}\big) \leq O(\sqrt{|I| T})$ (and the same bound for $R_T^{MDP}\big(\pi^*\big)$), depending on the number of possible noises but independent of the size of the state space and action space. 

\subsubsection{Bounding $R_T^{policy}$}
The term $R_T^{policy}$ pays for the loss associated with the convergence of MD-CURL. Our main challenge is to deal with the terms concerning variable constraint sets $\mathcal{M}_{\mu_0}^t$. They depend on a bound on the difference between the state-action distributions induced by two consecutive probability kernel estimates, i.e. $\|\mu^{\pi, \hat{p}^t} - \mu^{\pi, \hat{p}^{t+1}}\|_{\infty,1}$ stated in Lemma~\ref{lemma:bound_consecutive_p}. We also need a bound on $\|\nabla \psi(\mu^{\pi,p})\|_{\infty,1}$, the function inducing the Bregman divergence, justifying our construction in Equation~\eqref{tilde_pi}. The result is stated in Proposition~\ref{thm:bound_regret_md}.

\begin{lemma}\label{lemma:bound_consecutive_p}
    For any policy sequence $\pi \in (\Delta_{\mathcal{A}})^{\mathcal{X} \times N}$, the estimation of the probability kernel for two consecutive episodes done by Greedy MD-CURL satisfies, for all episodes $t \in [T-1]$, the following inequality
    \begin{equation*}
        \smash{\| \mu^{\pi, \hat{p}^{t+1}} - \mu^{\pi, \hat{p}^t} \|_{\infty,1} \leq \frac{2 N}{t}.}
    \end{equation*}
\end{lemma}

\begin{proposition}\label{thm:bound_regret_md}
Under the same hypothesis as in Theorem~\ref{main_result}, let $b$ be defined as in Equation~\eqref{auxiliary_b}. Then, Greedy MD-CURL obtains, for $\tau = \frac{b}{L \sqrt{T}}$,
    \begin{equation*}
                \smash{R_T^{policy} \leq  2 \ell N b \sqrt{T}.}
    \end{equation*}
\end{proposition}
\begin{proof}
See Appendix~\ref{proof:bound_OMD_term}.
\end{proof}

\begin{remark}
Appendix~\ref{bounds_unknown_g} shows that Greedy MD-CURL also has sub-linear regret in T when both $g_n$ and $h_n$ are unknown and the learner observes the trajectory of state-action pairs that each agent follows. Although its regret is not the best in the state of the art, Greedy MD-CURL is a good option when exploration is already induced by the environment.
\end{remark}

\section{Showcase experiments}\label{experiments}
In this section, we evaluate the performance of MD-CURL and Greedy MD-CURL on the \textit{entropy maximisation} and \textit{multi-objectives} problems, both introduced by \cite{perrolat_curl}. To test Greedy MD-CURL's ability to learn the unknown dynamics, we consider a version with fixed, non-adversarial objective functions and the same probability kernel for all $ n \in [N]$. Appendix~\ref{more_experiments} provides further experimental results.

\subsection{Environments}
We consider a model where the state space is a $11 \times 11$ four-room dimensional grid world with a single door connecting adjacent rooms. At each step, the agent can choose to stay still, go right, left, up or down, provided that there are no walls in the way:
\begin{equation}\label{four_walls_dynamics}
    x_{n+1} = x_n + a_n + \epsilon_n,
\end{equation}
with $a_n \in \{(0,0), (0,1), (1,0), (-1,0), (0,-1)\}$. The external noise $\epsilon_n$ represents a perturbation that pushes the agent to a neighbor state with a certain probability. We suppose the initial distribution is a Dirac at the upper left corner of the grid as in Figure~\ref{fig:auxiliary_images} [left]. 
\paragraph{Entropy maximisation} At each step, $f_n(\mu_n^{\pi,p}) := \langle \rho_n^{\pi,p}, \log(\rho_n^{\pi,p}) \rangle$, where $\rho_n^{\pi,p}(x) := \sum_{a \in \mathcal{A}} \mu_n^{\pi,p}(x,a)$. Thus minimizing $F := \sum_{n =1}^N f_n$ means maximizing the entropy, so the optimal value is when the distribution is uniform over the state space (obs.: contrary to intuition, the uniform policy does not provide an optimal solution, as can be seen in Figure~\ref{fig:auxiliary_images}, [middle]). 
\paragraph{Multi-objectives} The goal is for the distribution to be concentrated on the three targets in Figure~\ref{fig:auxiliary_images}, [right], by the final step $N$. We let $f_n(\mu_n^{\pi,p}) := -\sum_{k=1}^3 (1 - \langle \rho_n^{\pi,p}, e^k \rangle)^2$, where $e^k \in \mathbb{R}^{|\mathcal{X}|}$ is a vector with zero everywhere and $1$ in the element corresponding to a target state. Note that the target may not be reachable by any policy.
 
\begin{figure}
     \centering
     \begin{subfigure}[h]{0.15\textwidth}
         \centering
         \includegraphics[scale=0.215]{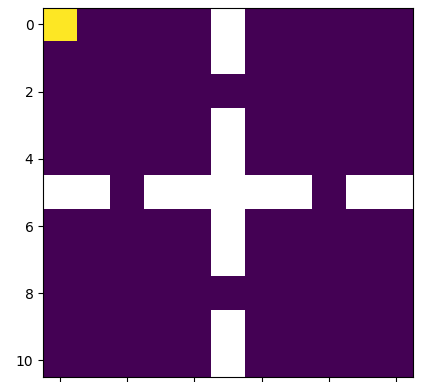}
     \end{subfigure}
     \begin{subfigure}[h]{0.15\textwidth}
         \centering
         \includegraphics[scale=0.215]{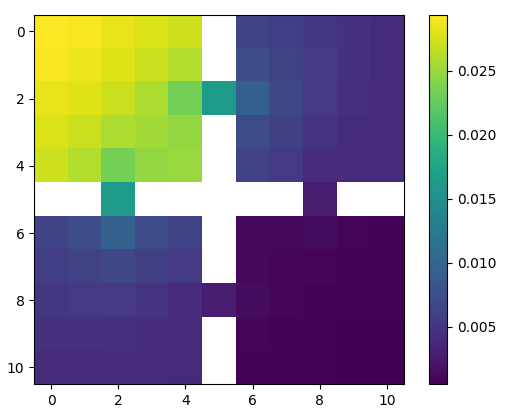}
     \end{subfigure}
     \begin{subfigure}[h]{0.15\textwidth}
         \centering
         \includegraphics[scale=0.215]{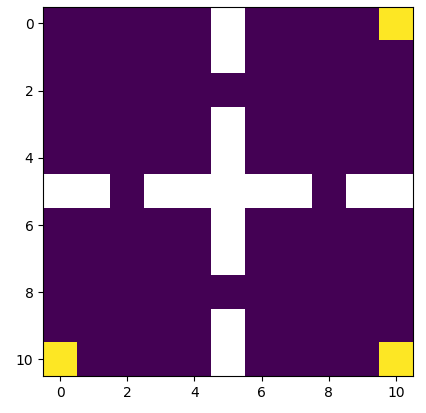}
     \end{subfigure}
        \caption{[left] Initial agent distribution; [middle] Distribution induced by the uniform policy; [right] The three targets. }
        \label{fig:auxiliary_images}
\end{figure}

\subsection{Numerical experiments}
For all experiments we consider $N = 40$. Figures~\ref{fig:exp1_max_entropy} and \ref{fig:exp1_multi_obj} show at left the state distribution at $N=40$ computed after $500$ iterations of MD-CURL for each setting, and at right its log regret per iteration compared to that of OMD for MFG. The OMD algorithm for MFGs is the state-of-the-art method for the problems addressed in this paper, as shown by \cite{MFG_survey}, but have no convergence results for discrete iterations. Therefore, MD-CURL is a good alternative for achieving state-of-the-art performance, with the advantage of having theoretical results. Note that both algorithms converge similarly, we leave the analysis of their differences for future work.
\begin{figure}
     \centering
     \begin{subfigure}[h]{0.23\textwidth}
         \centering
         \includegraphics[scale=0.25]{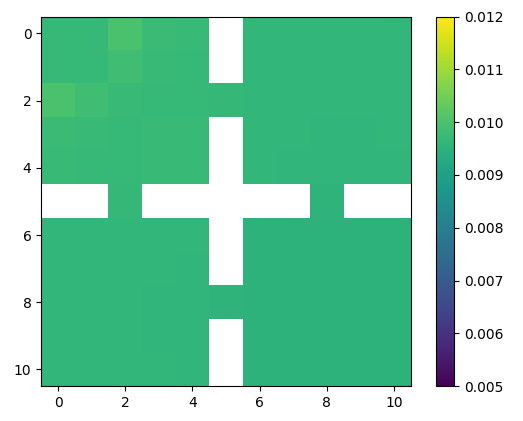}
     \end{subfigure}
     \begin{subfigure}[h]{0.23\textwidth}
         \centering
         \input{experiment1/max_entropy/objective_function_step40_noise_0.05_model_True_noise_True.txt}
     \end{subfigure}
        \caption{Entropy maximisation: [left] MD-CURL distribution at $N = 40$; [right] Log regret.}
        \label{fig:exp1_max_entropy}
\end{figure}
\begin{figure}
     \centering
     \begin{subfigure}[h]{0.23\textwidth}
         \centering
         \includegraphics[scale=0.25]{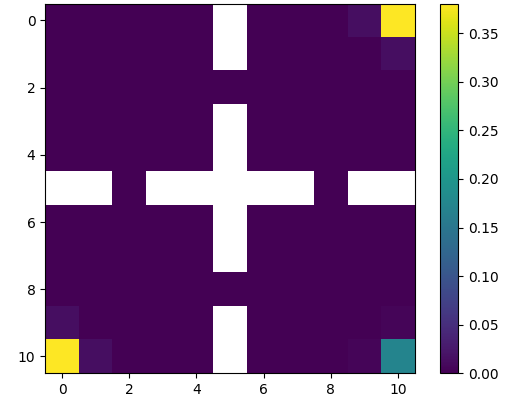}
     \end{subfigure}
     \begin{subfigure}[h]{0.23\textwidth}
         \centering
         \input{experiment1/multi_obj/objective_function_step40_noise_0.05_model_True_noise_True.txt}
     \end{subfigure}
        \caption{Multi-objectives: [left] MD-CURL distribution at $N=40$; [right] Log regret.}
        \label{fig:exp1_multi_obj}
\end{figure}

We now examine Greedy MD-CURL for online CURL. We add a noise $\epsilon_n$ that follows a categorical distribution $h_n$, with a $0.2$ probability of going up and $0$ for other directions. We suppose $g_n$ is known but $h_n$ is unknown, and we take $M=10$. Figure~\ref{fig:exp2_learn_noise} compares the log regret per iteration for Greedy MD-CURL (blue), MD-CURL with known noise dynamics (green), and MD-CURL with unknown noise dynamics, where the learner never learns the noise distribution, i.e. $\hat{p}^t_n(\cdot|x,a) = \delta_{g_n(x,a,0)}$ for all $(x,a)$ (red). We see that Greedy MD-CURL quickly matches MD-CURL with known noise dynamics, and that never learning the noise is sub-optimal. We do not compare Greedy MD-CURL to other algorithms in the literature as none is well-suited to our scenario, and the ones that could be adapted use UCRL techniques making them computationally expensive or intractable.
\begin{figure}
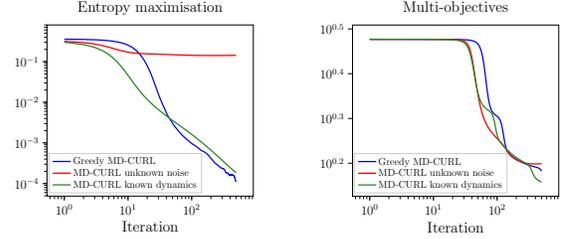

     \centering
     \begin{subfigure}[h]{0.23\textwidth}
         \centering
\input{experiment2/objective_function_step40_noise_0.2_step_40_maxentropy_algo_MD-MFC.txt}
     \end{subfigure}
     \begin{subfigure}[h]{0.23\textwidth}
     \centering
\input{experiment2/objective_function_step40_noise_0.2_step_40_multiobj_algo_MD-MFC.txt}
     \end{subfigure}
        \caption{Log regret per iteration, $N = 40$: [left] Entropy maximisation; [right] Multi-objectives.}
        \label{fig:exp2_learn_noise}
\end{figure}

Finally, Greedy MD-CURL achieves sub-linear regret even with unknown dynamics (see Appendix~\ref{bounds_unknown_g}). Figure~\ref{fig:exp3_learn_p} shows how it learns the full dynamics for both the entropy maximization problem (right) and the multi-objective problem with $0.2$ probability of being perturbed in any reachable neighboring state (left). It exploits the fact that maximizing entropy and perturbations with high probability favors exploration.
\begin{figure}[H]
     \centering
     \begin{subfigure}[h]{0.23\textwidth}
         \centering
% This file was created with tikzplotlib v0.10.1.
\begin{tikzpicture}[scale=0.4]

\definecolor{darkgray176}{RGB}{176,176,176}

\begin{loglogaxis}[
legend style={nodes={scale=0.8, transform shape}},
legend cell align={left},
legend style={fill opacity=0.8, draw opacity=1, text opacity=1, draw=white!80!black, at={(1.15,0.99)},
anchor=north east},
tick align=outside,
tick pos=left,
title={\Large{Entropy maximisation}},
x grid style={darkgray176},
xlabel={\Large{Iteration}},
xtick style={color=black},
tick label style={font=\large}, 
y grid style={darkgray176},
ytick style={color=black}
]
\addplot [very thick, red]
table[y expr=1-\thisrow{values}] {%
t values
0 0.775753891021572
1 0.789377087487765
2 0.804582067614656
3 0.824757445976491
4 0.845199699511137
5 0.865241020377067
6 0.885254104594146
7 0.903948232640846
8 0.919600398259425
9 0.932069563643089
10 0.941661833859077
11 0.948944548484887
12 0.954304076947016
13 0.958254883628021
14 0.961315831714412
15 0.963795888520919
16 0.965635966384273
17 0.967017349150771
18 0.968556519248033
19 0.96992136396739
20 0.971151675955051
21 0.972945167696412
22 0.974921002127651
23 0.977068585869415
24 0.979131661208312
25 0.981445286252038
26 0.983760291665862
27 0.986075497956768
28 0.988278411021784
29 0.990261774543928
30 0.991973931134441
31 0.993427105140056
32 0.99463673916496
33 0.995664487393412
34 0.996469289383898
35 0.9971717223661
36 0.997754394053355
37 0.998219720050772
38 0.998597092164227
39 0.998875112957719
40 0.999088204307579
41 0.999243982858267
42 0.99934851592503
43 0.999411159498303
44 0.999437860172663
45 0.999439608195844
46 0.999418768724889
47 0.999378966102125
48 0.999323425981553
49 0.9992549621318
50 0.999175988121352
51 0.999147177976097
52 0.999140873564967
53 0.999129553190016
54 0.999114803030021
55 0.999135276219068
56 0.999148905348827
57 0.999157671378535
58 0.999162852606972
59 0.999165296244119
60 0.999165576268622
61 0.999164088609293
62 0.999161110554416
63 0.999156838989829
64 0.999151415649919
65 0.999144944131233
66 0.999137501525997
67 0.999129146455754
68 0.999119924650699
69 0.999109872832783
70 0.99909902141588
71 0.999087396376707
72 0.999075020543373
73 0.999061914475356
74 0.999048097057871
75 0.999033585897841
76 0.999018397583253
77 0.999002547849617
78 0.99898605168427
79 0.999061003737676
80 0.999118500154464
81 0.999163651077543
82 0.999199862394163
83 0.999229461945099
84 0.999254075857909
85 0.999274862076781
86 0.999292659484579
87 0.999308085532843
88 0.999321601544074
89 0.999333557195046
90 0.99934422129382
91 0.999353803361655
92 0.999362468946881
93 0.999370350607608
94 0.999377555866561
95 0.99938417302746
96 0.999390275467276
97 0.999395924833112
98 0.99940117344556
99 0.999406066122852
100 0.999410641578965
};
\addlegendentry{\large{Greedy MD-CURL unknown $g$}}
\addplot [very thick, forestgreen]
table[y expr=1-\thisrow{values}] {%
t values
0 0.858023096406266
1 0.876463754204909
2 0.896003620903909
3 0.915613736157719
4 0.934111831459075
5 0.95047301791243
6 0.964075366933002
7 0.974766569358555
8 0.98276775189413
9 0.988508817698306
10 0.992480318432942
11 0.9951385390804
12 0.996862015367511
13 0.997942529114798
14 0.998593815496927
15 0.998966661815079
16 0.999164335983321
17 0.999255785510735
18 0.999285948678397
19 0.999283379898387
20 0.999265686496046
21 0.999243304311133
22 0.999222066487778
23 0.99920492113253
24 0.999193061504991
25 0.999186657525824
26 0.999185320568524
27 0.999188392159726
28 0.99919511795342
29 0.999204748038566
30 0.999216590754239
31 0.99923003780326
32 0.999244572178121
33 0.999259766252026
34 0.999275274654205
35 0.999290824769333
36 0.999306206554812
37 0.999321262640537
38 0.999335879219208
39 0.999349977955287
40 0.999363508974117
41 0.999376444898103
42 0.999388775847168
43 0.999400505298531
44 0.999411646695058
45 0.999422220694889
46 0.999432252963419
47 0.999441772419206
48 0.999450809856466
49 0.999459396877597
50 0.999467565079034
51 0.999475345442717
52 0.999482767893166
53 0.99948986098698
54 0.999496651707292
55 0.9995031653406
56 0.999509425417539
57 0.999515453702544
58 0.999521270220272
59 0.999526893308977
60 0.999532339693038
61 0.999537624568419
62 0.999542761696182
63 0.999547763500242
64 0.999552641166457
65 0.999557404740813
66 0.999562063225088
67 0.999566624668796
68 0.999571096256606
69 0.999575484390699
70 0.999579794767773
71 0.999584032450562
72 0.999588201933868
73 0.999592307205239
74 0.99959635180044
75 0.999600338853969
76 0.999604271144882
77 0.999608151138207
78 0.999611981022247
79 0.999615762742079
80 0.999619498029531
81 0.999623188429944
82 0.999626835325981
83 0.999630439958753
84 0.999634003446518
85 0.99963752680117
86 0.999641010942751
87 0.999644456712178
88 0.999647864882375
89 0.999651236167974
90 0.99965457123375
91 0.999657870701919
92 0.999661135158442
93 0.999664365158439
94 0.999667561230822
95 0.999670723882245
96 0.999673853600454
97 0.999676950857105
98 0.999680016110134
99 0.999683049805722
100 0.999686052379923
};
\addlegendentry{\large{MD-CURL known dynamics}}
\end{loglogaxis}

\end{tikzpicture}
     \end{subfigure}
     \begin{subfigure}[h]{0.23\textwidth}
     \centering
% This file was created with tikzplotlib v0.10.1.
\begin{tikzpicture}[scale=0.4]

\definecolor{darkgray176}{RGB}{176,176,176}

\begin{loglogaxis}[
legend style={nodes={scale=0.8, transform shape}},
legend cell align={left},
legend style={fill opacity=0.8, draw opacity=1, text opacity=1, draw=white!80!black, at={(0.80,0.23)},
anchor=north east},
tick align=outside,
tick pos=left,
title={\Large{Multi-objectives with central noise}},
x grid style={darkgray176},
xlabel={\Large{Iteration}},
xtick style={color=black},
tick label style={font=\large}, 
y grid style={darkgray176},
ytick style={color=black}
]
\addplot [very thick, red]
table[y expr=-\thisrow{values}] {%
t values
0 -2.97409912129946
1 -2.97392996758267
2 -2.97364744043188
3 -2.97327500752979
4 -2.97283187246578
5 -2.97232343447139
6 -2.97175058238509
7 -2.97111869116665
8 -2.9704272174765
9 -2.96967703375588
10 -2.96886723922463
11 -2.96799486277751
12 -2.96706174612443
13 -2.96606064896959
14 -2.96499225694162
15 -2.9638530830779
16 -2.96263783187259
17 -2.9613426290389
18 -2.95996234815014
19 -2.95849019129022
20 -2.95692150306371
21 -2.95524891454499
22 -2.95346447747784
23 -2.95156356598625
24 -2.94953774088549
25 -2.94737493677429
26 -2.94506314308201
27 -2.94259389686057
28 -2.9399533025359
29 -2.9371260727497
30 -2.93410068697582
31 -2.93086865933176
32 -2.92741662210436
33 -2.92372342856374
34 -2.919771867494
35 -2.91554636632033
36 -2.91103203063061
37 -2.90619448853045
38 -2.90101557572776
39 -2.8954726129857
40 -2.88953906837096
41 -2.88320246859551
42 -2.87643289566903
43 -2.86921532214557
44 -2.86152445539862
45 -2.85333394845822
46 -2.84461791427316
47 -2.83535850398926
48 -2.8255322309433
49 -2.81511703175196
50 -2.80410423449484
51 -2.79246975633045
52 -2.78021228913809
53 -2.76731896402804
54 -2.75381296061599
55 -2.73969238728067
56 -2.72498034624805
57 -2.70969800992351
58 -2.69384668502845
59 -2.67748879414855
60 -2.66069000381792
61 -2.64345078211892
62 -2.62579259957491
63 -2.60774239097912
64 -2.58935961192459
65 -2.57076549179317
66 -2.55204933656469
67 -2.53323790209704
68 -2.51434835215273
69 -2.49545018348548
70 -2.4766019528186
71 -2.45790727785349
72 -2.43938347844198
73 -2.4210944779372
74 -2.40308096954102
75 -2.3853608007524
76 -2.36796802860243
77 -2.35096107482604
78 -2.33433908982286
79 -2.31812142305071
80 -2.30233120378154
81 -2.28697539538095
82 -2.27205194215169
83 -2.25755771874489
84 -2.24352415639219
85 -2.22994054068618
86 -2.21680043503122
87 -2.20407287550204
88 -2.19177209955945
89 -2.17991328576507
90 -2.16848492047251
91 -2.15745883815697
92 -2.14682166715545
93 -2.13657068110845
94 -2.12667646532098
95 -2.1171527246139
96 -2.10799044896304
97 -2.09915777031379
98 -2.09063974903179
99 -2.08243305726996
100 -2.07452194727046
101 -2.06691096503771
102 -2.05957156086473
103 -2.0524922671798
104 -2.04567075507786
105 -2.03911066881234
106 -2.03279071409964
107 -2.02669147625119
108 -2.02080715690442
109 -2.01512906248613
110 -2.00964996357047
111 -2.00436344546523
112 -1.99926278887518
113 -1.99433293311303
114 -1.98957113649103
115 -1.98496320799093
116 -1.98050526230053
117 -1.97618741318166
118 -1.97201426315078
119 -1.967971347783
120 -1.96405153734461
121 -1.96024738362874
122 -1.95655652996266
123 -1.9529708101578
124 -1.94948262425922
125 -1.94608734539894
126 -1.94277999275598
127 -1.9395540783523
128 -1.93641061394146
129 -1.93334199409922
130 -1.93034383972498
131 -1.92741110195914
132 -1.92453983480998
133 -1.92172630440132
134 -1.91897316770243
135 -1.91627376257471
136 -1.91362631732611
137 -1.91102541112052
138 -1.90847153303644
139 -1.90596148737758
140 -1.90348932051652
141 -1.90105770269872
142 -1.89866142148817
143 -1.89629689920951
144 -1.89396697805224
145 -1.89167000063553
146 -1.88940050300053
147 -1.88715917756732
148 -1.88494227304013
149 -1.8827496021494
150 -1.88057930873397
151 -1.87843560253169
152 -1.87631478239694
153 -1.8742158345594
154 -1.87213796863602
155 -1.87008222680018
156 -1.86804461824032
157 -1.86602518609999
158 -1.8640264639072
159 -1.86204472667942
160 -1.86007827193675
161 -1.85813355678945
162 -1.85620798565844
163 -1.85429949925909
164 -1.85240859794989
165 -1.85052928047969
166 -1.84866295592775
167 -1.84681081924525
168 -1.84497634966874
169 -1.8431539668967
170 -1.84134599621646
171 -1.8395537457639
172 -1.8377784204221
173 -1.83602088229933
174 -1.8342799479464
175 -1.83256049503594
176 -1.83086113118988
177 -1.82918367334676
178 -1.82752443456791
179 -1.82588360932223
180 -1.82426397027907
181 -1.8226633596758
182 -1.82108426094111
183 -1.81952511531445
184 -1.81798396635249
185 -1.81646135118286
186 -1.81495935593782
187 -1.81347948290329
188 -1.81201952579748
189 -1.81058272904101
190 -1.80916941055117
191 -1.80777972758717
192 -1.80641029336389
193 -1.80506517719129
194 -1.80373886705127
195 -1.80243577427793
196 -1.80115582928227
197 -1.79989937618044
198 -1.79866492002099
199 -1.79745361049639
200 -1.79626286040264
201 -1.79509285007306
202 -1.79394495266927
203 -1.79281915762134
204 -1.79171482435124
205 -1.79062808290221
206 -1.78956125366574
207 -1.78851453615893
208 -1.78748797684648
209 -1.78648166086507
};
\addlegendentry{\large{Greedy MD-CURL unknown $g$}}
\addplot [very thick, forestgreen]
table[y expr=-\thisrow{values}] {%
t values
0 -2.97409912129946
1 -2.97282481606283
2 -2.97146302618436
3 -2.97000646693291
4 -2.96844718313238
5 -2.96677648797607
6 -2.96498489759644
7 -2.96306206155802
8 -2.96099668961431
9 -2.95877647529441
10 -2.95638801717052
11 -2.95381673901552
12 -2.95104681050025
13 -2.94806107061473
14 -2.94484095663299
15 -2.94136644218458
16 -2.93761598884396
17 -2.93356651659304
18 -2.92919339952795
19 -2.92447049422786
20 -2.91937020921806
21 -2.91386362484947
22 -2.90792067356068
23 -2.90151039073003
24 -2.89460124597689
25 -2.88716156362375
26 -2.87916003886598
27 -2.87056635281597
28 -2.86135188485006
29 -2.85149051454817
30 -2.84095949808966
31 -2.82974039556858
32 -2.81782001686931
33 -2.80519134529888
34 -2.79185439111178
35 -2.77781692251838
36 -2.76309502084341
37 -2.74771341009224
38 -2.73170551976627
39 -2.71511325323056
40 -2.69798645146338
41 -2.68038206211491
42 -2.6623630444461
43 -2.64399705962937
44 -2.62535501089491
45 -2.60650950738826
46 -2.58753332840346
47 -2.56849796080554
48 -2.54947227275791
49 -2.53052137281513
50 -2.5117056869142
51 -2.4930802687548
52 -2.47469434323327
53 -2.45659106929141
54 -2.43880749853532
55 -2.42137469952724
56 -2.40431801457192
57 -2.38765741561923
58 -2.37140792792385
59 -2.35558009363693
60 -2.3401804518919
61 -2.3252120166365
62 -2.3106747380333
63 -2.29656593741788
64 -2.28288070940014
65 -2.26961228765133
66 -2.25675237323795
67 -2.24429142609153
68 -2.23221892141412
69 -2.22052357359873
70 -2.20919353067708
71 -2.19821654247465
72 -2.18758010562469
73 -2.17727158842824
74 -2.1672783382951
75 -2.15758777419878
76 -2.14818746625621
77 -2.13906520422214
78 -2.13020905638336
79 -2.12160742005919
80 -2.11324906466875
81 -2.10512316811404
82 -2.09721934705256
83 -2.08952768149294
84 -2.08203873403935
85 -2.07474356403356
86 -2.06763373679326
87 -2.06070132811908
88 -2.05393892423607
89 -2.04733961734536
90 -2.04089699698354
91 -2.03460513741858
92 -2.02845858134708
93 -2.02245232019686
94 -2.0165817713775
95 -2.01084275285838
96 -2.00523145548691
97 -1.99974441348808
98 -1.99437847360875
99 -1.98913076338612
100 -1.9839986590276
101 -1.97897975338903
102 -1.974071824528
103 -1.96927280528878
104 -1.96458075434417
105 -1.95999382907767
106 -1.95551026063686
107 -1.95112833142778
108 -1.9468463552515
109 -1.94266266021126
110 -1.93857557444433
111 -1.9345834146599
112 -1.93068447739666
113 -1.9268770328533
114 -1.92315932109458
115 -1.91952955039623
116 -1.9159858974642
117 -1.91252650924839
118 -1.90914950606594
119 -1.90585298575496
120 -1.90263502859309
121 -1.89949370273547
122 -1.89642706995191
123 -1.89343319147078
124 -1.8905101337664
125 -1.88765597415624
126 -1.88486880610188
127 -1.88214674413427
128 -1.87948792834711
129 -1.87689052842316
130 -1.87435274717584
131 -1.87187282360332
132 -1.86944903546407
133 -1.86707970139231
134 -1.86476318257831
135 -1.86249788404377
136 -1.86028225554512
137 -1.85811479213953
138 -1.85599403444873
139 -1.85391856865514
140 -1.85188702626386
141 -1.84989808366213
142 -1.84795046150595
143 -1.84604292396125
144 -1.84417427782462
145 -1.84234337154622
146 -1.84054909417511
147 -1.83879037424494
148 -1.83706617861599
149 -1.83537551128717
150 -1.83371741219009
151 -1.83209095597547
152 -1.83049525080058
153 -1.82892943712511
154 -1.82739268652163
155 -1.8258842005057
156 -1.82440320938971
157 -1.82294897116376
158 -1.82152077040605
159 -1.82011791722492
160 -1.81873974623359
161 -1.81738561555883
162 -1.81605490588399
163 -1.81474701952655
164 -1.81346137955022
165 -1.81219742891128
166 -1.81095462963868
167 -1.80973246204731
168 -1.80853042398366
169 -1.80734803010307
170 -1.80618481117772
171 -1.80504031343434
172 -1.80391409792062
173 -1.80280573989946
174 -1.80171482826989
175 -1.80064096501369
176 -1.79958376466677
177 -1.79854285381413
178 -1.79751787060764
179 -1.79650846430543
180 -1.79551429483217
181 -1.79453503235918
182 -1.79357035690351
183 -1.79261995794522
184 -1.79168353406187
185 -1.79076079257957
186 -1.78985144923977
187 -1.78895522788101
188 -1.78807186013502
189 -1.78720108513639
190 -1.78634264924523
191 -1.78549630578222
192 -1.78466181477534
193 -1.78383894271792
194 -1.78302746233724
195 -1.78222715237335
196 -1.78143779736754
197 -1.78065918745997
198 -1.77989111819613
199 -1.77913339034159
200 -1.77838580970465
201 -1.77764818696662
202 -1.77692033751923
203 -1.77620208130891
204 -1.77549324268757
205 -1.77479365026965
206 -1.77410313679497
207 -1.77342153899729
208 -1.77274869747823
209 -1.77208445658626
};
\addlegendentry{\large{MD-CURL known dynamics}}
\end{loglogaxis}

\end{tikzpicture}
     \end{subfigure}
        \caption{Log regret for Greedy MD-CURL with unknown $g_n$ and $h_n$: [left] Entropy maximisation; [right] Multi-objectives.}
        \label{fig:exp3_learn_p}
\end{figure}
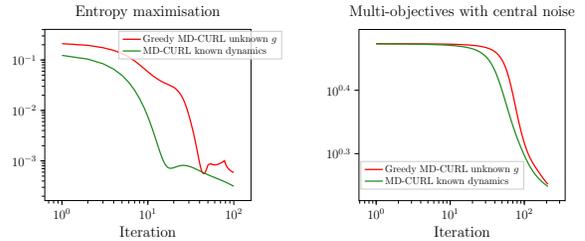
\section{Conclusion and future works}
In this paper we analyzed two versions of the CURL problem in episodic MDPs with finite state and action spaces. For the offline optimization problem, where the dynamics $g_n$ and $h_n$ are known, we proposed an algorithm based on mirror descent converging with a rate of $O\big(1/\sqrt{K}\big)$ for K iterations. For the online learning extension with adversarial costs, we proposed an algorithm with a simple closed-form solution, and regret of $O\big(\sqrt{T} \log(T)\big)$ when $g_n$ is known but $h_n$ is unknown. Also, we showed that for this specific dynamic, we can improve the bounds of existing work and pay the number of states only in a term proportional to $\sqrt{\log(|\mathcal{X}|})$.

A future direction is to investigate if we can achieve optimal regrets for variants of Greedy MD-CURL under more general assumptions about the available dynamics information. For example, by considering the case where $g_n$ is a parametric function with unknown parameters, rather than being completely known.

\bibliography{bib}

\onecolumn
% \aistatstitle{Instructions for Paper Submissions to AISTATS 2023: \\
% Supplementary Materials}
% \raggedbottom
\appendix
\section{Equivalence between policies and distributions in $\mathcal{M}_{\mu_0}$}\label{missing_results}

\begin{proposition}\label{opt_mu_equal_pi}
Let $\mu_0 \in \Delta_{\mathcal{X} \times \mathcal{A}}$. The application ${\pi \mapsto \mu^{\pi}}$ is a bijection from $(\Delta_\mathcal{A})^{\mathcal{X} \times N}$ to $\mathcal{M}_{\mu_0}$.
\end{proposition}

\begin{proof}
Consider a fixed initial state-action distribution $\mu_0 \in \Delta_{\mathcal{X} \times \mathcal{A}}$. Let $\mu \in \mathcal{M}_{\mu_0}$ and define $\rho = (\rho_n)_{0 \leq n \leq N}$ such that for all $x \in \mathcal{X}$, $\rho_n(x) = \sum_{a} \mu_n(x,a)$  (the associated state distribution). First, let us deal with the case where $\rho_n(x) \neq 0$. Define a policy sequence $\pi \in (\Delta_\mathcal{A})^{\mathcal{X} \times N}$ such that $\smash{\pi_{n}(a|x) = \frac{\mu_{n}(x,a)}{\rho_{n}(x)}}$ for all $(x,a) \in \mathcal{X} \times \mathcal{A}$. We want to show that $\mu^\pi = \mu$ for this policy $\pi$. We reason by induction. For $n=0$, $\mu_0^\pi = \mu_0$ by definition. Suppose $\mu^\pi_n = \mu_n$, thus for $n+1$ and for all $(x',a') \in \mathcal{X} \times \mathcal{A}$ 
    \begin{align*}
        \mu^\pi_{n+1}(x',a') &= \sum_{x,a} p_{n+1}(x'|x,a) \mu_n^\pi(x,a) \pi_{n+1}(a'|x') \\
        &= \sum_{x,a} p_{n+1}(x'|x,a) \mu_n(x,a) \frac{\mu_{n+1}(x',a')}{\rho_{n+1}(x')} \\
        &= \sum_{a} \mu_{n+1}(x',a) \frac{\mu_{n+1}(x',a')}{\rho_{n+1}(x')} \\
        &= \rho_{n+1}(x') \frac{\mu_{n+1}(x',a')}{\rho_{n+1}(x')} \\
        &= \mu_{n+1}(x',a'),
    \end{align*}
where the first equality comes from Equation~\eqref{mu_induced_pi}, the second equality comes from the induction assumption and the way we defined the strategy $\pi$, and the third comes from the assumption that $\mu \in \mathcal{M}_{\mu_0}$.

In the case $\rho_n(x) = 0$, we therefore have $\mu_n(x,a) = 0$ for all $a \in \mathcal{A}$, so any choice of $\pi_n(a|x)$ would work. Because we want to make sure that there is a unique mapping from each $\pi$ to $\mu^\pi$ we agree to always set $\pi_n(a|x) = \frac{1}{|\mathcal{A}|}$ in this case, where $|\mathcal{A}|$ is the number of possible actions.

\end{proof}

%%%%%%%%%%%%%%%%%%%%%%%%%%%%%%%%%%%%%%%%%%%%%%%%

\section{Proofs of Section~\ref{algorithm_knownp}: algorithm~\ref{alg:MD} scheme and convergence rate}\label{missing_proofs_2}

By abuse of notation, for any probability measure $\eta \in \Delta_E$ whatever the finite space $E$ on which it is defined we introduce the neg-entropy function, with the convention $0 \log(0) = 0$,
\begin{equation}\label{neg-entropy}
\phi (\eta):=\sum_{x\in E} \eta (x)\log \eta (x),    
\end{equation}

 to which we associate the Bregman divergence $D$, also known as the KL divergence, such that for any pair $(\eta,\nu)\in\Delta_E\times \Delta_E$, 
$$
D(\eta,\nu):=\phi(\eta)-\phi(\nu)-\langle \nabla \phi(\nu),\eta-\nu\rangle.
$$

Let $\rho_n$ denote the marginal probability distribution on $\mathcal{X}$ associated with $\mu_n$ i.e., for all $x \in \mathcal{X}$
$$
\rho_n(x):= \sum_{a \in \mathcal{A}} \mu_n(x,a)\ .
$$
Observe that to any $\mu=(\mu_n)_{1 \leq n \leq N}\in \mathcal{M}_{\mu_0}$ one can associate a unique probability mass function on $\Delta_{(\mathcal{X} \times \mathcal{A})^N}$ denoted by $\mu_{1:N}$ such that $\mu_{1:N}$ is \textit{generated} by the strategy $\pi=(\pi_n)_{ n \in [N]}$ associated with $\mu$ which is determined by 
$$
\pi_{n}(a\vert x)=\frac{\mu_n(x,a)}{\rho_n(x)}\ ,
$$
when $\rho_n(x) \neq 0$, otherwise we fix an arbitrary strategy $\pi_n(a \vert x) = \frac{1}{|\mathcal{A}|}$.

Before proving Theorems~\ref{MD_explicit_result} and \ref{thm:convergence_rate} we state and prove a lemma which is key to proving both theorems.

\begin{lemma}\label{prop:D_decomp}
For any $\mu \in \mathcal{M}_{\mu_0}$ and $\mu' \in \mathcal{M}_{\mu_0}^*$, with associated probability mass functions $\mu_{1:N},\mu'_{1:N}\in \Delta_{(\mathcal{X} \times \mathcal{A})^N} $ \textit{generated} by $\pi,\pi'$ respectively with the same initial state-action distribution, i.e. $\mu_0 = \mu'_0$, we have
\begin{equation}\label{eq:youpi}
\begin{split}
D(\mu_{1:N},\mu'_{1:N}) =\Gamma(\mu, \mu') ={\displaystyle \sum_{n=1}^{N} D(\mu_n,\mu'_n)- \sum_{n=1}^{N} D(\rho_n,\rho'_n)},
\end{split}
\end{equation}
where $\Gamma(\mu, \mu') := \sum_{n=1}^{N} \mathbb{E}_{(x,a) \sim \mu_n(\cdot)}\left[\log\bigg(\frac{\pi_{n}(a|x)}{\pi'_{n}(a|x)}\bigg)\right].$
\end{lemma}

\begin{proof}
For each $n \in [N]$, let us define a transition matrix $P^{\pi_n}$ for all $x, x' \in \mathcal{X}$ and $a, a' \in \mathcal{A}$,
\begin{equation*}
    P^{\pi_n}(x',a'|x,a) := p_n(x'|x,a) \pi_n(a'|x').
\end{equation*}
Given Definition~\ref{mu_induced_pi}, for any randomized policy the state-action distributions evolve according to linear dynamics
\begin{equation*}
\mu_n(x',a') = \langle \mu_{n-1}(\cdot), P^{\pi_n}(x',a'| \cdot) \rangle.
\end{equation*}
Any randomized policy $\pi$ induces a probability mass function $\mu_{1:N}$  that is Markovian:
\begin{equation}\label{mu_matrix_decomp}
    \mu_{1:N}(\vec{y}) = \mu_0(y_0) P^{\pi_1}(y_1|y_0) \ldots P^{\pi_N}(y_N|y_{N-1}),
\end{equation}
where $\vec{y}$ represents the elements of $(\mathcal{X} \times \mathcal{A})^{N+1}$ such that $y_i = (x_i,a_i)$ for all $0\leq i \leq N$. Note that $\mu_n(y_n)$ is the marginal probability mass function. 

Consider $\mu, \mu' \in \mathcal{M}_{\mu_0}$ the state-action distribution sequences induced by $\pi, \pi'$ respectively (i.e, $\mu = \mu^\pi$ and $\mu' = \mu^{\pi'})$. Thus, computing the relative entropy between the probability mass functions $\mu_{1:N},\mu'_{1:N}$ gives
\begin{align*}
    D(\mu_{1:N}, \mu'_{1:N}) &= \sum_{\vec{y}} \mu_{1:N}(\vec{y}) \log\left(\frac{\mu_{1:N}(\vec{y})}{\mu'_{1:N}(\vec{y})}\right) \\
    &= \sum_{y_0,\ldots,y_N} \mu_{1:N}(\vec{y}) \log\left(\frac{\mu_0(y_0) P^{\pi_1}(y_1|y_0) \ldots P^{\pi_N}(y_N|y_{N-1})}{\mu'_0(y_0) P^{\pi'_1}(y_1|y_0) \ldots P^{\pi'_N}(y_N|y_{N-1})}\right) \\
    &= \sum_{y_0,\ldots,y_N} \mu_{1:N}(\vec{y}) \sum_{i=1}^N\log\left(\frac{P^{\pi_i}(y_i|y_{i-1})}{P^{\pi'_i}(y_i|y_{i-1})}\right).
\end{align*}

Where
\begin{align*}
    \sum_{i=1}^N\log\left(\frac{P^{\pi_i}(y_i|y_{i-1})}{P^{\pi'_i}(y_i|y_{i-1})}\right) &= \sum_{i=1}^N\log\left( \frac{p_i(x_i|x_{i-1},a_{i-1}) \pi_i(a_i|x_i)}{p_i(x_i|x_{i-1},a_{i-1}) \pi_i'(a_{i}|x_i)} \right) \\
    &= \sum_{i=1}^N \log\left( \frac{ \pi_i(a_i|x_i)}{\pi_i'(a_{i}|x_i)} \right).
\end{align*}

Thus,
\begin{align*}
    D(\mu_{1:N}, \mu'_{1:N}) &= \sum_{\vec{y}} \mu_{1:N}(\vec{y}) \sum_{i=1}^N \log\left( \frac{ \pi_i(a_i|x_i)}{\pi_i'(a_{i}|x_i)} \right) \\
    &= \sum_{\vec{y}} \mu_0(y_0) P^{\pi_1}(y_1|y_0) \ldots P^{\pi_N}(y_N|y_{N-1}) \sum_{i=1}^N \log\left( \frac{ \pi_i(a_i|x_i)}{\pi_i'(a_{i}|x_i)} \right) \\
    &= \sum_{i=1}^N \sum_{x \in \mathcal{X}}\sum_{a \in \mathcal{A}} \mu_i(x,a) \log\left(\frac{\pi_i(a|x)}{\pi'_i(a|x)}\right).
\end{align*}

Where for the last equality we used that
\[
    \sum_{y_0,\ldots,y_{i-1}} \mu_0(y_0) P^{\pi_1}(y_1|y_0)\ldots P^{\pi_{i}}(y_{i}| y_{i-1}) = \sum_{y_{i}} \mu_{i}(y_{i}) 
\]
and for a fixed $y_i$,
\[
    \sum_{y_{i+1},\ldots,y_N} P^{\pi_{i+1}}(y_{i+1}|y_{i})\ldots
    P^{\pi_{N}}(y_{N}| y_{N-1}) = 1.
\]

This proves the first equality of the lemma. We now prove the second. For this, we recall that Proposition~\ref{opt_mu_equal_pi} gives a unique relation between a state-action distribution sequence $\mu \in \mathcal{M}_{\mu_0}$ and the policy sequence $\pi \in (\Delta_\mathcal{A})^{\mathcal{X} \times N}$ inducing it by taking for all $1 \leq i \leq N$, $(x,a) \in \mathcal{X} \times \mathcal{A}$, 
$$
\pi_{i}(a\vert x)=\frac{\mu_i(x,a)}{\rho_i(x)},\,
$$
where $\rho$ is the marginal on the states of $\mu$.
Using this relation, we have then that
\begin{equation*}
    \begin{split}
         D(\mu_{1:N},\mu'_{1:N}) &=  \sum_{i=1}^{N} \sum_{x \in \mathcal{X}} \sum_{a \in \mathcal{A}} \mu_{i}(x,a) \log\bigg(\frac{\pi_{i}(a|x)}{\pi'_{i}(a|x)}\bigg) \\
         &=  \sum_{i=1}^{N} \sum_{x \in \mathcal{X}} \sum_{a \in \mathcal{A}} \mu_{i}(x,a) \log\bigg(\frac{\mu_{i}(a|x)}{\rho_{i}(x)} \frac{\rho'_{i}(x)}{\mu'_{i}(a|x)}\bigg) \\
         &= \sum_{i=1}^{N} \sum_{x \in \mathcal{X}} \sum_{a \in \mathcal{A}} \mu_{i}(x,a) \log\bigg(\frac{\mu_{i}(a|x)}{\mu'_{i}(a|x)}\bigg) - \sum_{i=1}^{N} \sum_{x \in \mathcal{X}} \sum_{a \in \mathcal{A}} \mu_{i}(x,a) \log\bigg(\frac{\rho_{i}(x)}{\rho'_{i}(x)}\bigg) \\
         &=  \sum_{i=1}^{N} \sum_{x \in \mathcal{X}} \sum_{a \in \mathcal{A}} \mu_{i}(x,a) \log\bigg(\frac{\mu_{i}(a|x)}{\mu'_{i}(a|x)}\bigg) - \sum_{i=1}^{N} \sum_{x \in \mathcal{X}} \rho_{i}(x) \log\bigg(\frac{\rho_{i}(x)}{\rho'_{i}(x)}\bigg) \\
         &= \sum_{i=1}^{N} D(\mu_i, \mu_i') - \sum_{i=1}^{N} D(\rho_i, \rho_i') 
    \end{split}
\end{equation*}
which concludes the proof.
\end{proof}

\subsection{Proof of Theorem \ref{MD_explicit_result}: formulation of Algorithm~\ref{alg:MD}}\label{thm_1_proof}
\begin{theorem*}
Let $k \geq 0$. The solution of Problem~\eqref{MD_iterative_scheme} is $\mu^{k+1} = \mu^{\pi^{k+1}}$, where for all $n \in [N]$, and $(x,a) \in \mathcal{X} \times \mathcal{A}$, 
\begin{equation*}
    \pi_{n}^{k+1}(a|x) := \frac{\pi_{n}^{k}(a|x) \exp\left(\tau_k \tilde{Q}_{n}^{k}(x,a) \right)}{\sum_{a' \in \mathcal{A}}\pi_{n}^{k}(a'|x) \exp\left(\tau_k \tilde{Q}_{n}^{k}(x,a') \right)},
\end{equation*}
where $\tilde{Q}$ is a regularized $Q$-function satisfying the following recursion
\begin{equation*}
    \begin{cases}
        \Tilde{Q}^k_N(x,a) = -\nabla f_N(\mu_N^k)(x,a) \\
        \!\begin{aligned}
            &\Tilde{Q}^k_n(x,a) = \max_{\pi_{n+1} \in (\Delta_\mathcal{A})^{\mathcal{X}}} \Bigg\{ -\nabla f_n(\mu_n^k)(x,a) +\\
            & \qquad \sum_{x'} p_{n+1}(x'|x,a) \sum_{a'} \pi_{n+ 1}(a'|x') \bigg[  - \frac{1}{\tau_k}\log\left( \frac{\pi_{n+1}(a'|x')}{\pi_{n+1}^k(a'|x')}\right) 
             + \Tilde{Q}^k_{n+1}(x',a') \bigg] \Bigg\}.
        \end{aligned}
    \end{cases}
\end{equation*}
\end{theorem*}
\begin{proof}

At each iteration we seek to find $\mu^{k+1}$ a minimizer of
\begin{equation}
\label{MD_opt_problem_appendix}
\begin{split}
\min_{\mu^\pi\in\mathcal{M}_{\mu_0}}
\bigg\{
\langle \nabla F(\mu^k),\mu^\pi\rangle 
+\frac{1}{\tau_k} \sum_{n=1}^{N} \mathbb{E}_{(x,a) \sim \mu_n(\cdot)}\left[\log\bigg(\frac{\pi_{n}(a|x)}{\pi^k_{n}(a|x)}\bigg)\right]
\bigg\}\,
\end{split}
\end{equation}

where recall that $\langle \nabla F(\mu^k),\mu^\pi\rangle := \sum_{n=1}^N \langle \nabla f_n(\mu_n^k), \mu^\pi_n \rangle$. We further use that ${r_n(x_n, a_n, \mu_n) := - \nabla f_n(\mu_n)(x_n, a_n)}$.

Now, we use the optimality principle to solve this optimization problem with an algorithm backwards in time. Remember that the initial distribution $\mu_0$ is always fixed. The equivalence between solving a minimization problem on sequences of state-action distributions in $\mathcal{M}_{\mu_0}$ and on sequences of policies in $(\Delta_\mathcal{A})^{\mathcal{X} \times N}$ (see Proposition~\ref{opt_mu_equal_pi}), allows us to reformulate Problem~\eqref{MD_opt_problem_appendix} on $\mathcal{M}_{\mu_0}$ into a problem on $(\Delta_\mathcal{A})^{\mathcal{X} \times N}$, thus

\begin{equation*}
    \begin{split}
        \eqref{MD_opt_problem_appendix} &= \max_{\pi \in (\Delta_\mathcal{A})^{\mathcal{X} \times N}} \bigg\{\sum_{n=0}^N \sum_{x,a} \mu_n^\pi(x,a) r_n(x,a,\mu_n^k) \\
        &\qquad- \frac{1}{\tau_k} \sum_{n=1}^N \sum_{x,a} \mu_{n-1}^\pi(x,a) \sum_{x',a'} p_n(x'|x,a) \pi_n(a'|x') \log\left( \frac{\pi_n(a'|x')}{\pi_n^k(a'|x')}\right) \bigg\} \\
        &= \max_{\pi \in (\Delta_\mathcal{A})^{\mathcal{X} \times N}} \bigg\{\sum_{n=0}^N \sum_{x,a} \mu^\pi_n(x,a) \bigg[ r_n(x,a,\mu_n^k) \\
        &\qquad - \frac{1}{\tau_k} \sum_{x',a'} p_{n+1}(x'|x,a) \pi_{n+1}(a'|x') \log\left( \frac{\pi_{n+1}(a'|x')}{\pi_{n+1}^k(a'|x')}\right) \bigg] \bigg\}\\
        &= \max_{\pi \in (\Delta_\mathcal{A})^{\mathcal{X} \times N}} \bigg\{\mathbb{E}_\pi \bigg[ r_N(x_N,a_N,\mu_N^k) + \sum_{n=0}^{N-1}  r_n(x_n,a_n,\mu_n^k) \\
        &\qquad- \frac{1}{\tau_k} \sum_{x',a'} p_{n+1}(x'|x_n,a_n) \pi_{n+1}(a'|x') \log\left( \frac{\pi_{n+1}(a'|x')}{\pi_{n+1}^k(a'|x')}\right) \bigg] \bigg\}.
    \end{split}
\end{equation*}

Let us define a regularized version of the state-action value function that we denote by $\Tilde{Q}^k$, such that for all $i \in [N]$, $(x,a) \in \mathcal{X} \times \mathcal{A}$,
\begin{equation}
\begin{split}
    \Tilde{Q}^k_i(x,a) &=  \max_{\pi_{i+1:N} \in (\Delta_\mathcal{A})^{\mathcal{X} \times ({N-i})}} \mathbb{E}_\pi \bigg[ r_N(x_N,a_N,\mu_N^k) + \sum_{n=i}^{N-1} \bigg\{ r_n(x_n,a_n,\mu_n^k) \\
    &- \frac{1}{\tau_k} \sum_{x',a'} p_{n+1}(x'|x_n,a_n) \pi_{n+1}(a'|x') \log\left( \frac{\pi_{n+1}(a'|x')}{\pi_{n+1}^k(a'|x')}\right) \bigg\} \bigg| (x_i,a_i) = (x,a) \bigg],  
\end{split}
\end{equation}
where $\pi_{i+1:N} = \{\pi_{i+1},\ldots,\pi_N\}$.

First, note that $\mathbb{E}_{(x,a) \sim \mu_0(\cdot)}[\Tilde{Q}^k_0(x,a)] =$ \eqref{MD_opt_problem_appendix}. Moreover, the optimality principle states that this regularized state-action value function satisfies the following recursion
\begin{equation*}
    \begin{cases}
    \Tilde{Q}_N(x,a) = r_N(x,a, \mu_N^k) \\
    \!\begin{aligned}
            \Tilde{Q}_i(x,a) &= \max_{\pi_{i+1} \in (\Delta_\mathcal{A})^{\mathcal{X}}} \bigg\{ r_i(x,a,\mu_i^k) + \\
            &\qquad \sum_{x'} p_{i+1}(x'|x,a) \sum_{a'} \pi_{i+ 1}(a'|x') \left[  - \frac{1}{\tau_k}\log\left( \frac{\pi_{i+1}(a'|x')}{\pi_{i+1}^k(a'|x')}\right)  + \Tilde{Q}_{i+1}(x',a') \right] \bigg\}.
    \end{aligned}
    \end{cases}
\end{equation*}

Thus, to solve \eqref{MD_opt_problem_appendix} we compute backwards in time, i.e. for $i = N-1,\ldots,0$, for all $x \in \mathcal{X}$,
\begin{equation*}
    \pi_{i+1}^{k+1}(\cdot|x) \in \argmax_{\pi(\cdot|x) \in \Delta_\mathcal{A}} \left\{\big\langle \pi(\cdot|x), \Tilde{Q}_{i+1}^k(x,\cdot) \big\rangle - \frac{1}{\tau_k} D\big(\pi(\cdot|x), \pi^k_{i+1}(\cdot|x)\big) \right\},
\end{equation*}
where $D$ is the KL divergence.

The solution of this optimisation problem for each time step $i$ can be found by writing the associated Lagrangian function $\mathcal{L}$. Let $\lambda$ be the Lagrangian multiplier associated with the simplex constraint. For simplicity, let $\pi_x := \pi(\cdot|x)$, $\pi_x^k := \pi_{i+1}^k(\cdot|x)$ and $\tilde{Q}^k_x := \tilde{Q}_{i+1}^{k}(x,\cdot)$. Thus,
\begin{equation*}
    \mathcal{L}(\pi_x, \lambda) = \langle \pi_x, \tilde{Q}^k_x \rangle - \frac{1}{\tau_k} D(\pi_x, \pi_x^k) - \lambda \left(\sum_{a \in \mathcal{A}} \pi_x(a) - 1 \right).
\end{equation*}
Taking the gradient of the Lagrangian with respect to $\pi_x(a)$ for each $a \in \mathcal{A}$ gives
\begin{equation*}
    \frac{\partial \mathcal{L} (\pi_x, \lambda)}{\partial \pi_x(a)} = \tilde{Q}^k_x(a) - \frac{1}{\tau_k} \log\left(\frac{\pi_x(a)}{\pi_x^k(a)}\right) - \frac{1}{\tau_k} - \lambda,
\end{equation*}
and thus
\begin{equation*}
    \begin{split}
        \frac{\partial \mathcal{L}  (\pi_x, \lambda)}{\partial \pi_x(a)} = 0 \Longrightarrow \quad \pi_x(a) = \pi_x^k(a) \exp{\left(\tau_k \tilde{Q}_x^k(a) - 1 - \tau_k\lambda\right)}.
    \end{split}
\end{equation*}

Applying the simplex constraint, $\sum_{a\in \mathcal{A}} \pi_x(a) = 1$, we find the value of the Lagrangian multipler $\lambda$, and we get for all $a \in \mathcal{A}$

\begin{equation*}
   \pi_x(a) = \frac{\pi_x^k(a) \exp{\left(\tau_k \tilde{Q}_x^k(a)\right)}}{\sum_{a'\in \mathcal{A}} \pi_x^k(a') \exp{\left(\tau_k \tilde{Q}_x^k(a')\right)},}
\end{equation*}
which proves the theorem. 

\end{proof}

\subsection{Proof of Proposition~\ref{penalization_is_bregman}}\label{penalization_is_bregman_proof}

\begin{proposition*}
    Let $\mu \in \mathcal{M}_{\mu_0}$ with marginal given by $\rho \in (\Delta_\mathcal{X})^N$. The divergence $\Gamma$ is a Bregman divergence induced by the function
    \begin{equation*}
     \psi(\mu) := \sum_{n=1}^N \phi(\mu_n) - \sum_{n=1}^N \phi(\rho_n).
    \end{equation*}
    Also, $\psi$ is $1$-strongly convex with respect to the $\|\cdot\|_{\infty,1}$ norm.
\end{proposition*}

\begin{proof}
 Lemma~\ref{prop:D_decomp} states that for any $\mu \in \mathcal{M}_{\mu_0}$ and $\mu' \in \mathcal{M}_{\mu_0}^*$, induced by $\pi,\pi'$ respectively as in Equation~\ref{mu_induced_pi}, with the same initial state-action distribution, i.e. $\mu_0 = \mu'_0$, we have
 \begin{equation*}
     \Gamma(\mu, \mu') := \sum_{n=1}^{N} \mathbb{E}_{(x,a) \sim \mu_n(\cdot)}\left[\log\bigg(\frac{\pi_{n}(a'|x')}{\pi^k_{n}(a'|x')}\bigg)\right]
={\displaystyle \sum_{n=1}^{N} D(\mu'_n,\mu_n)- \sum_{n=1}^{N} D(\rho'_n,\rho_n)}.
 \end{equation*}
 
Recall that $\phi$ is the negentropy and that $D$ is the Bregman divergence induced by the negentropy. Define the function $\psi: (\Delta_{\mathcal{X} \times \mathcal{A}})^N \to \mathbb{R}$ such that for any $\mu \in (\Delta_{\mathcal{X} \times \mathcal{A}})^N$
\begin{equation*}
        \psi(\mu) := \sum_{n=1}^N \phi(\mu_n) - \sum_{n=1}^N \phi(\rho_n).
\end{equation*}
To show $\Gamma$ is a Bregman divergence induced by $\psi$ we need to show that for any $\mu, \mu' \in (\Delta_{\mathcal{X} \times \mathcal{A}})^N $,
    \begin{equation*}
        \psi(\mu) - \psi(\mu') - \langle \nabla \psi(\mu'), \mu - \mu' \rangle = \Gamma(\mu, \mu').
    \end{equation*}
    For that, first recall that the marginal $\rho$ is such that for each $1 \leq n \leq N$, and for all $x \in \mathcal{X}$, $\rho_n(x) = \sum_{a \in \mathcal{A}} \mu_n(x,a)$. Thus,
    \begin{equation}\label{decomposition_psi}
    \begin{split}
        \psi(\mu) &= \sum_n \left[\sum_{x,a} \mu_n(x,a) \log(\mu_n(x,a)) - \sum_{x} \rho_n(x) \log(\rho_n(x)) \right]  \\
        &= \sum_n \sum_{x,a} \mu_n(x,a) \log\left(\frac{\mu_n(x,a)}{\sum_{a'}\mu_n(x,a')}\right).
    \end{split}
    \end{equation}
    Computing the first order partial derivative of $\psi$ with respect to $\mu_n(x,a)$ for any  $(x,a) \in \mathcal{X} \times \mathcal{A}$ and $1 \leq n \leq N$, we get
        \begin{equation*}
    \begin{split}
        \frac{\partial \psi}{\partial \mu_n(x,a)}(\mu) &= \log\left(\frac{\mu_n(x,a)}{\sum_{a'}\mu_n(x,a')}\right) + \mu_n(x,a)\frac{1}{\mu_n(x,a)} - \sum_{a'}\mu_n(x,a') \frac{1}{\sum_{a'}\mu_n(x,a')} \\
        &= \log\left(\frac{\mu_n(x,a)}{\sum_{a'}\mu_n(x,a')}\right) =  \log\left(\frac{\mu_n(x,a)}{\rho_n(x)}\right).
    \end{split}
    \end{equation*}
    Hence, as $\phi(\mu_n) = \langle \mu_n, \log(\mu_n) \rangle$ and $\phi(\rho_n) = \langle \rho_n, \log(\rho_n) \rangle$, and $\pi_n = \mu_n/\rho_n$,
    \begin{equation*}
    \begin{split}
        \psi(\mu) - \psi(\mu') - \langle \nabla \psi(\mu'), \mu - \mu' \rangle &= \sum_{n=1}^N \big[\phi(\mu_n) - \phi(\rho_n) - \big( \phi(\mu_n') - \phi(\rho_n') - \big\langle \mu_n - \mu_n', \log(\mu_n) - \log(\rho_n)\big\rangle \big) \big] \\
        &= \sum_{n=1}^N \big[\phi(\mu_n) - \phi(\rho_n) - \mu_n \log(\pi_n') \big] \\
        &= \sum_{n=1}^N \big[\mu_n \big( \log(\pi_n) - \log(\pi_n') \big) \big] \\
        &= \Gamma(\mu, \mu').
    \end{split}
    \end{equation*}
    
    Now we just need to show that $\psi$ is strongly convex. For that, we apply the following convexity property \citep{boyd}: $\psi$ is $1$-strongly convex with respect to a norm $\|\cdot\|$ if and only if for all $\mu, \mu' \in (\Delta_{\mathcal{X} \times \mathcal{A}})^N $, $\langle \nabla \psi (\mu) - \nabla \psi (\mu'), \mu - \mu'\rangle \geq \|\mu - \mu'\|^2$.
    Indeed,
    \begin{align}\label{psi_gamma_convex}
        \langle \nabla \psi (\mu) - \nabla \psi (\mu'), \mu - \mu'\rangle &= \sum_n \sum_{x,a} \left[\frac{\partial \psi}{\partial \mu_n(x,a)}(\mu) - \frac{\partial \psi}{\partial \mu_n(x,a)}(\mu') \right] \big(\mu_n(x,a) - \mu '_n(x,a)\big) \nonumber \\
        &= \sum_n \sum_{x,a} \left[\log\left(\frac{\mu_n(x,a)}{\rho_n(x)}\right) - \log\left(\frac{\mu'_n(x,a)}{\rho'_n(x)}\right) \right] \big(\mu_n(x,a) - \mu '_n(x,a)\big) \nonumber \\
        % &= \sum_n \sum_{x,a} \log\bigg(\frac{\mu_n(x,a)}{\mu_n'(x,a)}\bigg)\big(\mu_n(x,a) - \mu_n'(x,a)\big) - \sum_n \sum_{x} \log\bigg(\frac{\rho_n(x)}{\rho_n'(x)}\bigg)\big(\rho_n(x) - \rho_n'(x)\big) \\
        &\overset{(a)}{=} \sum_n D(\mu_n, \mu_n') + D(\mu_n, \mu'_n) - D(\rho_n, \rho'_n) - D(\rho'_n, \rho_n) \nonumber \\
        &\overset{(b)}{=} \Gamma(\mu, \mu') + \Gamma(\mu', \mu),
    \end{align}
    where $(a)$ comes from the definition of the KL divergence $D$ and $(b)$ comes from the definition of $\Gamma$. 
    
It remains to find a norm that lower bound the right-hand side. By Lemma~\ref{prop:D_decomp},
    \begin{equation*}
        \begin{split}
            \Gamma(\mu, \mu') &= \sum_{n=1}^N D(\mu_n,\mu_n')-\sum_{n=1}^N D(\rho_n,\rho'_n) =D(\mu_{1:N},\mu'_{1:N}) \\
            &\geq 2\Vert \mu_{1:N}-\mu_{1:N}\Vert^2_{\textrm{TV}} =\frac{1}{2} \Vert \mu_{1:N}-\mu'_{1:N}\Vert_1^2,
        \end{split}
    \end{equation*}
the last inequality being a consequence of Pinsker's inequality. The norm $\|\cdot\|_{\text{TV}}$ stands for the total variation norm.
Let $y$ represent an element of $(\mathcal{X} \times \mathcal{A})^{N+1}$ such that $y_i \in \mathcal{X} \times \mathcal{A}$ for all $0 \leq i \leq N$. Observe that
\begin{eqnarray*}
\Vert \mu_{1:N}-\mu'_{1:N}\Vert_1
&=&\sum_{y\in (\mathcal{X} \times \mathcal{A})^{N+1}}  \vert \mu_{1:N}(y)-\mu'_{1:N}(y)\vert \\
&\geq&
\sum_{y_n\in \mathcal{X} \times \mathcal{A}} \bigg\vert \sum_{y_s \in \mathcal{X} \times \mathcal{A}\,,\,s\neq n}
\big( \mu_{1:N}(y)-\mu'_{1:N}(y)\big) \bigg\vert
\\
&=&\sum_{y_n\in \mathcal{X} \times \mathcal{A}} \vert \mu_n(y_n)-\mu'_n(y_n)\vert, \quad \textrm{for all} \ n\in \{0,\cdots, N\}.
\end{eqnarray*}
Thus,
\begin{eqnarray*}
        \Vert \mu_{1:N}-\mu'_{1:N}\Vert_1 &\geq& \Vert \mu-\mu'\Vert_{\infty,1}\,
\end{eqnarray*}
which implies that
\begin{equation}\label{gamma_strong_convex}
\Gamma(\mu,\mu') \geq \frac{1}{2} \| \mu-\mu'\|_{\infty,1}^2.
\end{equation}
Finally, plugging back into Equation~\eqref{psi_gamma_convex} shows that $\psi$ is $1$-strongly convex with respect to the norm $\|\cdot\|_{\infty,1}$.

\end{proof}

\subsection{Complements of the proof of Theorem~\ref{thm:convergence_rate}}
\begin{lemma}\label{proof_conv_MD_MFC} 
    Let $f_n: \Delta_{\mathcal{X} \times \mathcal{A}} \rightarrow \mathbb{R}$ be convex and $\ell$-Lipschitz with respect to the norm $\|\cdot\|_1$ for all $n \in [N]$. If $F: (\Delta_{\mathcal{X} \times \mathcal{A}})^N \rightarrow \mathbb{R} $ is defined for all $\mu := (\mu_n)_{n \in [N]} \in (\Delta_{\mathcal{X} \times \mathcal{A}})^N$ as $F(\mu) := \sum_{n=1}^N f_n(\mu_n)$, then $F$ is also convex and $L$-Lipschitz with respect to the norm $\|\cdot\|_{\infty,1}$ for $L = \ell N$.
\end{lemma}
\begin{proof}
\textbf{Convexity:} $F$ is convex as the sum of convex functions.

\textbf{Lipschitz:} Let $\mu, \mu' \in (\Delta_{\mathcal{X} \times \mathcal{A}})^N$. As $f_n$ is Lipschitz with respect to $\|\cdot\|_1$ with constant $\ell$, then $|f_n(\mu_n) - f_n(\mu'_n)| \leq \ell \|\mu_n - \mu_n\|_1$ for all $1 \leq n \leq N$. 
Therefore,
\begin{equation*}
    \begin{split}
        |F(\mu) - F(\mu')| &= \bigg| \sum_{n=1}^N f_n(\mu_n) - f_n(\mu'_n) \bigg| \leq \sum_{n=1}^N |f_n(\mu_n) - f_n(\mu'_n)| \\
        &\leq \ell \sum_{n=1}^N \|\mu_n - \mu_n'\|_1 \leq \ell N \|\mu - \mu'\|_{\infty,1}.
    \end{split}
\end{equation*}
\end{proof}

\section{Proofs of Subsection~\ref{learning_the_model}: concentration results}

\subsection{Proof of Lemma~\ref{concentration_inner_mu}}\label{proof:concentration_inner_mu}

\begin{lemma*}
    Let $\gamma > 0$. For any $0 < \delta < 1$, and for any function $\Lambda: \mathcal{X} \rightarrow \mathbb{R}$ such that $|\Lambda(x')| \leq \sqrt{\gamma}/2$ for all $x' \in \mathcal{X}$,
    \begin{equation*}
         \big(p_n - \hat{p}_n^t \big) \big(\Lambda \big) (x,a) \leq \sqrt{\frac{\gamma}{2 M_n^t} \log\bigg(\frac{N |\mathcal{X}| |\mathcal{A}| T}{\delta} \bigg)}
    \end{equation*}
    holds with probability $1-\delta$ simultaneously for all $(x,a) \in \mathcal{X} \times \mathcal{A}$, steps $n \in [N]$, and episodes $t \in [T]$.
\end{lemma*}

\begin{proof}
Let $\gamma > 0$. Recall that, for all $n \in [N]$, and for all $(x,a) \in \mathcal{X} \times \mathcal{A}$, $p_n(x|x,a) := \mathbb{P}\big(g_n(x,a, \epsilon_n) = x \big)$ and $\hat{p}_n^t(x|x,a) = \frac{1}{M_n^t} \sum_{s=1}^{M_n^t} \delta_{g_n(x,a, \epsilon^s_n)}$ where $M_n^t$ is the number of times we observe step $n$ until the start of episode $t$, and $\epsilon^s_n$ is the $s$-th noise observed at step $n$. Note that $M_n^t$ is not random. Therefore,
\begin{equation*}
     \big(p_n- \hat{p}_n^t \big) (\Lambda)(x,a) := \sum_{x' \in \mathcal{X}} \big(p_n(x'|x,a) - \hat{p}^t_n(x'|x,a) \big) \Lambda (x') = \mathbb{E}_{\epsilon_n \sim h_n(\cdot)} \big[ \Lambda \big( g_n(x,a, \epsilon_n) \big) \big] -  \frac{1}{M_n^t} \sum_{s=1}^{M_n^t} \Lambda \big( g_n(x,a, \epsilon^s_n) \big).
\end{equation*}

% Hoeffding inequality states that if $(X_s)_{s \in [M]}$ is a sequence of i.i.d. random variables such that $X_s \in [-\sqrt{\gamma}/2, \sqrt{\gamma}/2]$ for all $s \in [M]$, then
% \begin{equation*}
%     \mathbb{P} \bigg( \frac{1}{M}\sum_{s=1}^M X_s - \mathbb{E}[X_1] \geq \xi \bigg) \leq \exp\bigg(\frac{-2 \xi^2 M}{\gamma} \bigg).
% \end{equation*}

From the hypothesis on the bound of $\Lambda$, we have that almost surely $\Lambda \big( g_n(x,a, \epsilon^s_i) \big) \in [-\sqrt{\gamma}/2, \sqrt{\gamma}/2]$ for all $s \in M_n^t$, therefore, applying Hoeffding's inequality to the sequence of random variables $\big(\Lambda \big( g_n(x,a, \epsilon^s_i) \big) \big)_{s \in [M_n^t]}$ yields, for all $\xi > 0$,
\begin{equation}\label{hoeffding_fixed_y}
     \mathbb{P} \bigg( \big(p_n- \hat{p}_n^t \big) (\Lambda)(x,a)\geq \xi \bigg) \leq \exp\bigg(\frac{-2 \xi^2 M_n^t}{\gamma} \bigg).
\end{equation}

Applying the union bound we then get that simultaneously for all $(x,a) \in \mathcal{X} \times \mathcal{A}$, steps $n \in [N]$ and episodes $t \in [T]$, 
\begin{equation*}
      \big(p_n- \hat{p}_n^t \big) (\Lambda)(x,a)\leq \sqrt{\frac{\gamma}{2 M_n^t}\log\bigg(\frac{N |\mathcal{X}| |\mathcal{A}| T }{\delta}\bigg)} 
\end{equation*}
holds with probability $1 - \delta$ for any $0 < \delta < 1$.
\end{proof}
\subsection{Proof of Lemma~\ref{lemma:weak_bound_mu}}\label{proof:weak_bound_mu}

\begin{lemma*}
For any vector $v \in \mathbb{R}^{|\mathcal{X}| \times |\mathcal{A}|}$, for any strategy $\pi$ and for all $n \in [N]$,
    \begin{equation*}
        \begin{split}
            &\langle v,  \mu_n^{\pi, p} - \mu_n^{\pi, \hat{p}^t} \rangle = \sum_{i=1}^{n} \sum_{y \in \mathcal{X} \times \mathcal{A}} \mu_{i-1}^{\pi, \hat{p}^t}(y) \big(p_i- \hat{p}_i^t \big) (\Lambda^{i,n, \pi}_v)(y),
        \end{split}
    \end{equation*}
    where $\Lambda^{i,n, \pi}_v :\mathcal{X} \rightarrow \mathbb{R}$ is a function depending on $v, i, n$ and $\pi$ defined in Equation~\eqref{def_Lambda}. Also, if $\|v\|_\infty := \sup_{(x,a) \in \mathcal{X} \times \mathcal{A}} |v(x,a)| \leq V$, then $\|\Lambda_v^{i,n, \pi}\|_\infty \leq V$.
\end{lemma*}

\begin{proof}
For $y \in \mathcal{X} \times \mathcal{A}$, we denote by $v(y)$ the element $y$ of vector $v$.

For all $n \in [N]$, for all $y := (x,a) \in \mathcal{X} \times \mathcal{A}$ and $y' := (x',a') \in \mathcal{X} \times \mathcal{A}$, let
\begin{equation*}
    \begin{split}
        K_n(y, y') &:= p_n(x|x',a') \pi_n(a|x), \\
        \hat{K}_n^t(y, y') &:= \hat{p}^t_n(x|x',a') \pi_n(a|x).
    \end{split}
\end{equation*}
For $\eta$ a vector over $\mathcal{X} \times \mathcal{A}$, we define for all $y \in  \mathcal{X} \times \mathcal{A}$ and $y_0 \in \mathcal{X} \times \mathcal{A}$ the following notations
\begin{equation*}
\begin{split}
        \eta K_{1:n} (y) &:= \sum_{y_0 \in \mathcal{X} \times \mathcal{A}} \ldots \sum_{y_{n-1} \in \mathcal{X} \times \mathcal{A}} \eta(y_0) K_1(y_0,y_1) \ldots K_n(y,y_{n-1}) \\
         \eta(y_0) K_{1:n} (y) &:= \sum_{y_1 \in \mathcal{X} \times \mathcal{A}} \ldots \sum_{y_{n-1} \in \mathcal{X} \times \mathcal{A}} \eta(y_0) K_1(y_0,y_1) \ldots K_n(y,y_{n-1}).
\end{split}    
\end{equation*}
We can then rewrite the definition of a state-action distribution satisfying the Markovian dynamics and induced by a policy $\pi$ stated in Equation~\eqref{mu_induced_pi} as $\mu_n^{\pi,p} = \mu_0 K_{1:n}$, and $\mu_n^{\pi, \hat{p}^t} = \mu_0 \hat{K}^t_{1:n}$. With the convention that ${K_{n+1:n} := \mathrm{Id}}$ is the identity operator for all $n$, then
\begin{equation*}
\begin{split}
    \mu_n^{\pi, p} - \mu_n^{\pi, \hat{p}^t} &= \mu_0 K_{1:n} - \mu_0 \hat{K}^t_{1:n} \\
    &= (\mu_0 K_{1:n} - \mu_0 \hat{K}^t_1 K_{2:n}) + (\mu_0 \hat{K}^t_1 K_{2:n} - \mu_0 \hat{K}^t_{1:2}K_{3:n}) + \ldots + (\mu_0 \hat{K}^t_{1:n-1} K_{n} - \mu_0 \hat{K}^t_{1:n}) \\
    &= \sum_{i=1}^n \mu_{i-1}^{\pi, \hat{p}^t} \big( K_i - \hat{K}_i^t \big) K_{i+1:n}.
\end{split}
\end{equation*}

Note that, for all $y \in \mathcal{X} \times \mathcal{A}$, and all $i \in \{0,\ldots,n\}$,
\begin{equation*}
    \begin{split}
\mu_{i-1}^{\pi, \hat{p}^t} \big( K_i - \hat{K}_i^t \big) K_{i+1:n} (y)
     &= \sum_{y_{i-1}} \mu_{i-1}^{\pi, \hat{p}^t}(y_{i-1}) \sum_{y_i} \big( K_i(y_{i-1}, y_i) - \hat{K}_i^t(y_{i-1},y_i) \big) K_{i+1:n}(y) \\
     &=  \sum_{y_{i-1}} \mu_{i-1}^{\pi, \hat{p}^t}(y_{i-1}) \sum_{x_i} \big(p_i(x_i| y_{i-1}) - \hat{p}_i^t(x_i | y_{i-1}) \big) \sum_{a_i} \pi_i(a_i|x_i) K_{i+1:n}(y).
    \end{split}
\end{equation*}

Hence,
\begin{equation*}
\begin{split}
    \langle v,  \mu_n^{\pi, p} - \mu_n^{\pi, \hat{p}^t} \rangle &= \sum_{y} v(y) \sum_{i=1}^{n} \mu_{i-1}^{\pi, \hat{p}^t} \big( K_i - \hat{K}_i^t \big) K_{i+1:n} (y) \\
    &= \sum_{i=1}^{n} \sum_{y_{i-1} \in \mathcal{X} \times \mathcal{A}} \mu_{i-1}^{\pi, \hat{p}^t}(y_{i-1}) \sum_{x_i \in \mathcal{X}} \big(p_i(x_i| y_{i-1}) - \hat{p}_i^t(x_i | y_{i-1}) \big) \sum_{a_i \in \mathcal{A}} \pi_i(a_i|x_i) \sum_{y \in \mathcal{X} \times \mathcal{A}} K_{i+1:n}(y) v(y) \\
    &= \sum_{i=1}^{n} \sum_{y_{i-1} \in \mathcal{X} \times \mathcal{A}} \mu_{i-1}^{\pi, \hat{p}^t}(y_{i-1}) \sum_{x_i \in \mathcal{X}} \big(p_i(x_i| y_{i-1}) - \hat{p}_i^t(x_i | y_{i-1}) \big) \Lambda^{i,n, \pi}_v(x_i) \\
    &:= \sum_{i=1}^{n} \sum_{y_{i-1} \in \mathcal{X} \times \mathcal{A}} \mu_{i-1}^{\pi, \hat{p}^t}(y_{i-1}) \big(p_i- \hat{p}_i^t \big) (\Lambda^{i,n, \pi}_v)(y_{i-1}) \\
\end{split}
\end{equation*}

where we define the function $\Lambda^{i,n, \pi}_v : \mathcal{X} \rightarrow \mathbb{R}$ for any $v \in \mathbb{R}^{\mathcal{X} \times \mathcal{A}}$ as
\begin{equation}\label{def_Lambda}
    \Lambda^{i,n, \pi}_v(x) := \sum_{a \in \mathcal{A}} \pi_i(a|x) \sum_{y \in \mathcal{X} \times \mathcal{A}} K_{i+1:n}(y) v(y).
\end{equation}

If $\|v\|_\infty \leq V$, then for all $x \in \mathcal{X}$,
\begin{equation*}
\begin{split}
        | \Lambda^{i,n, \pi}_{ v} (x)| &\leq \sum_{a \in \mathcal{A}} \pi_i\big(a| x \big) \sum_{y' \in \mathcal{X} \times \mathcal{A}} K_{i+1:n}(y') | v (y')| \\
        &\leq V \sum_{a \in \mathcal{A}} \pi_i\big(a| x\big) \sum_{y' \in \mathcal{X} \times \mathcal{A}} K_{i+1:n}(y') \\
        &=  V.
\end{split}
\end{equation*}
Therefore, $\|\Lambda^{i,n, \pi}_{ v} \|_\infty \leq V$.
\end{proof}

\subsection{Proof of Proposition~\ref{prop:bound_regret_proba}: upper bound on $R_T^{MDP}$}\label{proof:bound_regret_proba}

\begin{proposition*}
We consider an episodic MDP with finite state space $\mathcal{X}$, finite action space $\mathcal{A}$, episodes of length $N$, and probability kernel $p:=(p_n)_{n \in [N]}$. We let $F^t := \sum_{n=1}^N f_n^t$ convex with $f_n^t$ $\ell$-Lipschitz with respect to the norm $\|\cdot\|_1$ for all $n \in [N], t \in [T]$. Then, with probability $1-\delta$, Greedy MD-CURL obtains,
    \begin{equation*}
    R_T^{MDP}\big((\pi^t)_{t \in [T]}\big) \leq \ell N^2 \sqrt{\frac{2 T}{M} \log\bigg(\frac{N |\mathcal{X}| |\mathcal{A}| T}{\delta}\bigg)}. \\
    \end{equation*}
    The exact same result being also valid for $R_T^{MDP}\big(\pi^*\big)$.
\end{proposition*}

\begin{proof}

The proof steps are the same for both terms, hence we show only the steps for $ R_T^{MDP}\big((\pi^t)_{t \in [T]}\big) $. Using the convexity of $F^t$ we obtain
\begin{equation*}
\begin{split}
    R_T^{MDP}\big((\pi^t)_{t \in [T]}\big) &\leq \sum_{t=1}^T \langle \nabla F^t(\mu^{\pi^t,p}), \mu^{\pi^t,p} - \mu^{\pi^t,\hat{p}^t} \rangle = \sum_{t=1}^T  \sum_{n=1}^N \langle \nabla f_n^t(\mu^{\pi^t,p}), \mu^{\pi^t,p}_n - \mu^{\pi^t,\hat{p}^t}_n\rangle.
\end{split}
\end{equation*}

To bound the inner product for each $n$, we first use the result of Lemma~\ref{lemma:weak_bound_mu} to obtain that
\begin{equation*}
    \langle \nabla f_n^t(\mu^{\pi^t,p}), \mu^{\pi^t,p}_n - \mu^{\pi^t,\hat{p}^t}_n\rangle = \sum_{i=1}^{n} \sum_{y \in \mathcal{X} \times \mathcal{A}} \mu_{i-1}^{\pi, \hat{p}^t}(y) \big(p_i- \hat{p}_i^t \big) \big(\Lambda^{i,n, \pi^t}_{\nabla f_n^t(\mu^{\pi^t,p})}\big)(y).
\end{equation*}
% where
% \begin{equation*}
%     \Lambda_{\nabla f_n^t(\mu^{\pi^t,p})}(x) := \sum_{a \in \mathcal{A}} \pi_i(a|x) \sum_{y \in \mathcal{X} \times \mathcal{A}} K_{i+1:n}(y) \nabla f_n^t(\mu^{\pi^t,p})(y).
% \end{equation*}

As $f_n^t$ is $\ell$-Lipschitz with respect to the norm $\|\cdot\|_1$ for all $n$ and $t$, then for all state-action distribution $\mu_n \in \Delta_{\mathcal{X} \times \mathcal{A}}$, $\|\nabla f_n^t(\mu_n)\|_\infty := \sup_{(x,a)} |\nabla f_n^t(\mu_n)(x,a)| \leq \ell$. Hence, from the second result of Lemma~\ref{lemma:weak_bound_mu} we have $\big\|\Lambda^{i,n, \pi^t}_{\nabla f_n^t(\mu^{\pi^t,p})}\big\|_\infty \leq \ell$. Therefore, all the conditions of Lemma~\ref{concentration_inner_mu} are satisfied with $\gamma = 4 \ell^2$, meaning that 
\begin{equation*}
\begin{split}
        \sum_{t=1}^T \sum_{n=1}^N  \langle \nabla f_n^t(\mu_n),  \mu_n^{\pi^t, p} - \mu_n^{\pi^t, \hat{p}^t} \rangle &=  \sum_{t=1}^T \sum_{n=1}^N  \sum_{i=1}^{n-1} \sum_{y \in \mathcal{X} \times \mathcal{A}} \mu_{i-1}^{\pi^t, \hat{p}^t}(y) \big(p_i- \hat{p}_i^t \big) \big(\Lambda^{i,n, \pi^t}_{\nabla f_n^t(\mu^{\pi^t,p})}\big)(y) \\
        &\leq \sum_{t=1}^T N^2 \ell \sqrt{\frac{2}{M t
        }\log\bigg(\frac{N |\mathcal{X}| |\mathcal{A}| T }{\delta}\bigg)} \\
    &\leq N^2 \ell \sqrt{\frac{2 T}{M} \log\bigg(\frac{N |\mathcal{X}| |\mathcal{A}| T}{\delta}\bigg)}
\end{split}
\end{equation*}
holds with probability $1-\delta$, where we use that in Greedy MD-CURL, $M_n^t = M (t-1)$ for all $n \in [N]$.
\end{proof}

\section{Proofs of Theorem~\ref{main_result}: upper bound on $R_T$}

% \subsection{Proof of Lemma~\ref{concentration_L1}}\label{proof:concentration_L1}
% \begin{proof}
%     We use the following result from \cite{Weissman2003} on the $L_1$-deviation of the true distribution $p$ and the empirical distribution $\hat{p}$ over $m$ distinct events from $s$ samples:
%     \begin{equation*}
%         \mathbb{P} \big( \|\hat{p}(\cdot) - p(\cdot) \| \geq \epsilon \big) \leq (2^m - 2) \exp\bigg(\frac{-s \epsilon^2}{2} \bigg).
%     \end{equation*}
%     In our case, for each episode $t \in [T]$, for each step $n \in [N]$, and for each couple $(x,a) \in \mathcal{X} \times \mathcal{A}$, by taking $s = M_n^t$, $m = |\mathcal{X}|$, $p(\cdot) = p_n(\cdot|x,a)$ and $\hat{p}(\cdot) = \hat{p}^t_n(\cdot|x,a)$, we get
%     \begin{equation*}
%          \|\hat{p}^t_n(\cdot|x,a) -  p_n(\cdot|x,a) \|  \leq \sqrt{\frac{2 |\mathcal{X}|}{M_n^t} \log\bigg(\frac{2}{\delta}\bigg)}
%     \end{equation*}
%     with probability $1-\delta$, simultaneously for all $(x,a)$.
% \end{proof}

\subsection{Auxiliary result: $L_1$ bound between distributions induced by the same policy but different probability kernels}

The bound of Theorem~\ref{main_result} depends on the auxiliary lemma bellow stating that the $L_1$ deviation of two state-action distributions induced by the same policies but different probability kernels is bounded by the $L_1$ difference between the probability kernels.

\begin{lemma}\label{lemma:bound_norm_mu}
        For any strategy $\pi \in (\Delta_\mathcal{A})^{\mathcal{X} \times N}$, for any two probability kernels $p = (p_n)_{n \in [N]}$ and $q = (q_n)_{n \in [N]}$ such that $p_n, q_n: \mathcal{X} \times \mathcal{A} \times \mathcal{X} \rightarrow [0,1]$, and for all $n \in [N]$,
    \begin{equation*}
    \begin{split}
        \|\mu_n^{\pi,p} - \mu_n^{\pi, q}\|_1 \leq \sum_{i=0}^{n-1} \sum_{x,a} \mu_i^{\pi, p}(x,a) \| p_{i+1}(\cdot| x,a) - q_{i+1} (\cdot| x,a) \|_1.
    \end{split}
\end{equation*}
\end{lemma}

    \begin{proof}
From the definition of a state-action distribution sequence induced by a policy $\pi$ in a probability kernel $p$ in Equation~\eqref{mu_induced_pi},  we have that for all $(x,a) \in \mathcal{X} \times \mathcal{A}$ and $n \in [N]$,
\begin{equation*}
    \mu_n^{\pi, p}(x,a) = \sum_{x',a'} \mu_{n-1}^{\pi, p}(x',a') p_n(x|x',a') \pi_n(a|x).
\end{equation*}

Thus,
\begin{equation*}
\begin{split}
        \|\mu_n^{\pi, p} - \mu_n^{\pi, q}\|_1 &= \sum_{x,a} \big|\mu_n^{\pi, p}(x,a) - \mu_n^{\pi, q}(x,a)\big| \\
        &= \sum_{x,a} \sum_{x',a'} \big| \mu_{n-1}^{\pi, p}(x',a') p_n(x|x',a') - \mu_{n-1}^{\pi, q}(x',a') q_n(x|x',a') \big| \pi_n(a|x) \\
        &= \sum_{x} \sum_{x',a'} \big| \mu_{n-1}^{\pi, p}(x',a') p_n(x|x',a') - \mu_{n-1}^{\pi,q}(x',a') q_n(x|x',a') \big| \\
        &= \sum_{x} \sum_{x',a'} \big| \mu_{n-1}^{\pi,p}(x',a') p_n(x|x',a') - \mu_{n-1}^{\pi,p}(x',a') q_n(x|x',a') \\
        & \hspace{1cm} + \mu_{n-1}^{\pi,p}(x',a') q_n(x|x',a') - \mu_{n-1}^{\pi,q}(x',a') q_n(x|x',a') \big| \\
        &\leq  \sum_{x',a'} \mu_{n-1}^{\pi,p}(x',a') \| p_n(\cdot|x',a') - q_n(\cdot|x',a') \|_1 + \sum_{x',a'} \big| \mu_{n-1}^{\pi,p}(x',a') - \mu_{n-1}^{\pi,q}(x',a') \big| \\
        &= \sum_{x',a'} \mu_{n-1}^{\pi,p}(x',a') \| p_n(\cdot|x',a') - q_n(\cdot|x',a') \|_1 +\| \mu_{n-1}^{\pi,p} - \mu_{n-1}^{\pi,q} \|_1.
\end{split}
\end{equation*}

Since for $n=0$, $\| \mu_{0}^{\pi,p} - \mu_{0}^{\pi,q} \|_1 = 0$, by induction we get that
\begin{equation*}
\begin{split}
    \|\mu_n^{\pi,p} - \mu_n^{\pi,q}\|_1  \leq \sum_{i=0}^{n-1} \sum_{x',a'} \mu_i^{\pi,p}(x',a') \| p_{i+1}(\cdot| x',a') - q_{i+1} (\cdot| x',a') \|_1.
    \end{split}
\end{equation*}
\end{proof}

\subsection{Auxiliary result: upper bound on $-\psi$}
Lemma~\ref{upper_bound_gamma} shows that the function $-\psi$, where $\psi$ is the function that induces the Bregman divergence $\Gamma$ according to Proposition~\ref{penalization_is_bregman}, is upper bounded. The definition of $\psi$ is recalled in the lemma.

\begin{lemma}\label{upper_bound_gamma}
    Let $\phi$ be the neg-entropy function defined in Equation~\eqref{neg-entropy}. Let $\psi: (\Delta_{\mathcal{X} \times \mathcal{A}})^N \rightarrow \mathbb{R}$ such that for all $\mu := (\mu_n)_{n \in [N]} \in (\Delta_{\mathcal{X} \times \mathcal{A}})^N $, where we let $\rho := (\rho_n)_{n \in [N]} \in (\Delta_\mathcal{X})^N$ be the associated sequence of marginals,
    \begin{equation*}
        \psi(\mu) := \sum_{n=1}^N \phi(\mu_n) - \sum_{n=1}^N \phi(\rho_n).
    \end{equation*}
    Then, $\sup_{\mu \in (\Delta_{\mathcal{X} \times \mathcal{A}})^N } - \psi(\mu) \leq N \log(|\mathcal{A}|)$.
\end{lemma}
\begin{proof}
    For $\mu \in  (\Delta_{\mathcal{X} \times \mathcal{A}})^N$ and for Lagrangian multipliers $\lambda_n \in \mathbb{R}$ associated to the constraints $\sum_{(x,a) \in \mathcal{X} \times \mathcal{A}} \mu_n(x,a) = 1$ for all $n \in [N]$, consider the Lagrangian given by
\begin{equation*}
    \mathcal{L}(\mu, \lambda) = \psi(\mu) + \sum_n \lambda_n \bigg(1-\sum_{x,a}\mu_n(x,a) \bigg).
\end{equation*}
For every $n \in [N]$, and $(x,a) \in \mathcal{X} \times \mathcal{A}$,
\begin{equation*}
    \frac{\partial \mathcal{L}(\mu, \lambda)}{\partial \mu_n(x,a)} = \log\bigg(\frac{\mu_n(x,a)}{\sum_{a'}\mu_n(x,a')}\bigg) - \lambda_n = 0,
\end{equation*}
thus
\begin{equation*}
     \frac{\mu_n(x,a)}{\sum_{a'}\mu_n(x,a')} = \exp({\lambda_n}).
\end{equation*}
To satisfy the constraint for each $n$, we need
\begin{equation*}
    \sum_{x,a} \mu_n(x,a) = \sum_{x,a} \exp{(\lambda_n)} \sum_{a'} \mu_n(x,a') = \exp{(\lambda_n)} |\mathcal{A}| =1.
\end{equation*}
Using the decomposition of $\psi$ proved in Equation~\eqref{decomposition_psi}, we get
\begin{equation*}
    -\psi(\mu) = \sum_n \sum_{x,a} \mu_n(x,a) \log\bigg(\frac{\sum_{a'} \mu_n(x,a')}{\mu_n(x,a)}\bigg) \leq \sum_n \sum_{x,a} \mu_n(x,a) \log(|\mathcal{A}|) = N \log(|\mathcal{A}|).
\end{equation*}
\end{proof}

\subsection{Proof of Lemma~\ref{lemma:bound_consecutive_p}}
\begin{lemma*}
    For any policy sequence $\pi$, the estimation of the probability kernel for two consecutive episodes done by Greedy MD-CURL satisfies, for all episodes $t \in [T-1]$, the following inequality
    \begin{equation*}
       \| \mu^{\pi, \hat{p}^{t+1}} - \mu^{\pi, \hat{p}^t} \|_{\infty,1} \leq \frac{2 N}{t}.
    \end{equation*}
\end{lemma*}

\begin{proof}
 For all $(x,a, x') \in \mathcal{X} \times \mathcal{A} \times \mathcal{X}$, and for all $i \in [N]$,
\begin{equation}\label{difference_consecutive_p}
\begin{split}
\| \hat{p}_i^{t+1}(\cdot|x,a) - \hat{p}_i^{t}(\cdot|x,a) \|_1 &= \sum_{x' \in \mathcal{X}} | \hat{p}_i^{t+1}(x'|x,a) - \hat{p}_i^{t}(x'|x,a) | \\
&= \sum_{x' \in \mathcal{X}} \bigg|\frac{1}{M t } \bigg(\sum_{j=1}^M \delta_{g_i(x,a,\epsilon_i^{j,t})}(x') + M(t-1) \hat{p}_i^t(x'|x,a) \bigg) - \hat{p}_i^t(x'|x,a) \bigg| \\
        &= \frac{1}{M t}  \sum_{x' \in \mathcal{X}} \bigg| \sum_{j=1}^M \delta_{g_i(x,a,\epsilon_i^{j,t})}(x') - \hat{p}_i^t(x'|x,a) \bigg| \leq \frac{2}{t}.
\end{split}
\end{equation}
Therefore, using the result of Lemma~\ref{lemma:bound_norm_mu} with $\hat{p}_i^{t}$ and $\hat{p}_i^{t+1}$,
\begin{equation*}
\begin{split}
    \| \mu^{\pi, \hat{p}^{t+1}} - \mu^{\pi, \hat{p}^t} \|_{\infty,1} &= \sup_{n \in [N]} \|\mu_n^{\pi, \hat{p}^{t+1}} - \mu_n^{\pi, \hat{p}^{t}} \|_1 \\
    &\leq  \sup_{n \in [N]} \sum_{i=0}^{n-1} \sum_{x,a} \mu_i^{\pi, \hat{p}^t}(x,a)  \| \hat{p}_i^{t+1}(\cdot|x,a) - \hat{p}_i^{t}(\cdot|x,a) \|_1 \\
    &\leq \frac{2 N}{t}.
\end{split}
\end{equation*}
\end{proof}

%%%%%%%%%%%%%%%%%%%%%%%%%%%%%%%%%%%%%%%%%%

\subsection{Auxiliary result: bound between distributions induced by $\pi^t$ and $\tilde{\pi}^t$}
\begin{lemma}\label{proof_mu_tilde}
    For all episodes $t \in [T]$, where $\pi^t$ is the strategy calculated using Greedy MD-CURL, $\tilde{\pi}^t = (1-\alpha_t) \pi^t + \frac{\alpha_t}{|\mathcal{A}|}$, and $\hat{p}^t, \hat{p}^{t+1}$ are two consecutive estimates of the probability kernel by Greedy MD-CURL, we have
    \begin{equation*}
         \|\mu^{\pi^t, \hat{p}^t} - \mu^{\tilde{\pi}^t,\hat{p}^{t+1}}\|_{\infty,1} \leq \sup_{n \in [N]} \bigg\{ \sum_{i=1}^{n-1} \sum_{x,a} \mu_{i}^{\pi^t, \hat{p}^t}(x,a) \| \hat{p}_{i+1}^t(\cdot|x,a) - \hat{p}_{i+1}^{t+1}(\cdot|x,a) \|_1 + 2 n \alpha_t\bigg\} .
    \end{equation*}
\end{lemma}
\begin{proof}
Using similar arguments as in Lemma~\ref{lemma:bound_norm_mu}, we get that for all $n \in [N]$,
\begin{equation*}
    \begin{split}
        \|\mu_n^{\pi^t, \hat{p}^t} - \mu_n^{\tilde{\pi}^t,\hat{p}^{t+1}}\|_1 &= \sum_{x,a} \big|\mu_n^{\pi^t, \hat{p}^t}(x,a) - \mu_n^{\tilde{\pi}^t,\hat{p}^{t+1}}(x,a)\big| \\
        &\leq \sum_{x,a} \sum_{x',a'} \bigg| \mu_{n-1}^{\pi^t, \hat{p}^t}(x',a') \hat{p}_n^t(x|x',a') \pi_n^t(a|x) - \mu_{n-1}^{\tilde{\pi}^t,\hat{p}^{t+1}} (x',a') \hat{p}_n^{t+1}(x|x',a') \bigg( (1-\alpha_t) \pi_n^t(a|x) + \alpha_t \frac{1}{|\mathcal{A}|} \bigg) \bigg|  \\
        &\leq \sum_{x, a} \sum_{x',a'} \big| \mu_{n-1}^{\pi^t, \hat{p}^t}(x',a') \hat{p}_n^t(x|x',a') - \mu_{n-1}^{\tilde{\pi}^t,\hat{p}^{t+1}}(x',a') \hat{p}_n^{t+1}(x|x',a') \big| \pi^t_n(a|x) \\
        &\hspace{1.5cm}+ \alpha_t \sum_{x,a} \sum_{x',a'}  \mu_{n-1}^{\tilde{\pi}^t,\hat{p}^{t+1}}(x',a') \hat{p}_n^{t+1}(x|x',a') \bigg| \pi_n^t(a|x) - \frac{1}{|\mathcal{A}|} \bigg| \\
        &\leq \sum_{x} \sum_{x',a'} \big| \mu_{n-1}^{\pi^t, \hat{p}^t}(x',a') \hat{p}_n^t(x|x',a') - \mu_{n-1}^{\pi^t, \hat{p}^t}(x',a') \hat{p}_n^{t+1}(x|x',a') \\
        &\hspace{1.5cm}+ \mu_{n-1}^{\pi^t, \hat{p}^t}(x',a') \hat{p}_n^{t+1}(x|x',a')
        - \mu_{n-1}^{\tilde{\pi}^t,\hat{p}^{t+1}}(x',a') \hat{p}_n^{t+1}(x|x',a') \big| + 2 \alpha_t \\
        &\leq  \sum_{x',a'} \mu_{n-1}^{\pi^t, \hat{p}^t}(x',a') \| \hat{p}_n^t(\cdot|x',a') - \hat{p}_n^{t+1}(\cdot|x',a') \|_1 + \sum_{x',a'} \big| \mu_{n-1}^{\pi^t, \hat{p}^t}(x',a') - \mu_{n-1}^{\tilde{\pi}^t,\hat{p}^{t+1}}(x',a') \big| + 2 \alpha_t\\
        &\leq  \sum_{x',a'} \mu_{n-1}^{\pi^t, \hat{p}^t}(x',a') \| \hat{p}_n^t(\cdot|x',a') - \hat{p}_n^{t+1}(\cdot|x',a') \|_1 + \| \mu_{n-1}^{\pi^t, \hat{p}^t} - \mu_{n-1}^{\tilde{\pi}^t,\hat{p}^{t+1}} \|_1  + 2 \alpha_t\\ 
        &\leq \sum_{i=0}^{n-1} \sum_{x,a} \mu_{i}^{\pi^t,\hat{p}^t}(x,a) \| \hat{p}_{i+1}^t(\cdot|x,a) -\hat{p}_{i+1}^{t+1}(\cdot|x,a) \|_1 + 2 n \alpha_t,
    \end{split}
\end{equation*}
where for the last inequality we use that $\mu_0^{\pi^t,\hat{p}^t} = \mu_0^{\tilde{\pi}^t, \hat{p}^{t+1}}$. To finish we just take the sup over $n \in [N]$.
\end{proof}

\subsection{Proof of Proposition~\ref{thm:bound_regret_md}: upper bound on $R_T^{policy}$}\label{proof:bound_OMD_term}
\begin{proposition*}
Consider an episodic MDP with finite state space $\mathcal{X}$, finite action space $\mathcal{A}$, episodes of length $N$, and probability kernel $p:=(p_n)_{n \in [N]}$. Let $F^t := \sum_{n=1}^N f_n^t$ convex with $f_n^t$ $\ell$-Lipschitz with respect to the norm $\|\cdot\|_1$ for all $n \in [N], t \in [T]$. Let $b$ be defined as in Equation~\eqref{auxiliary_b}. Then, Greedy MD-CURL obtains, for $\tau = \frac{b}{L \sqrt{T}}$,
    \begin{equation*}
                \smash{R_T^{policy} \leq  2 \ell N b \sqrt{T}.}
    \end{equation*}
\end{proposition*}
\begin{proof}
 Using the convexity of $F^t$,
    \begin{equation*}
    \begin{split}
        R_T^{policy} &= \sum_{t=1}^T F^t(\mu^{\pi^t,\hat{p}^t}) - F^t(\mu^{\pi^*, \hat{p}^{t+1}}) \leq \sum_{t=1}^T \langle l^t, \mu^{\pi^t,\hat{p}^t} - \mu^{\pi^*,\hat{p}^{t+1}} \rangle 
        \end{split}
    \end{equation*}
 where $l^t := \nabla F^t(\mu^{\pi^t,t})$  to shorten notation, and we also use the notation introduced in the main paper $\mu^t := \mu^{\pi^t, \hat{p}^t}$ and $\tilde{\mu}^t := \mu^{\tilde{\pi}^t, \hat{p}^t}$, for all $t \in [T]$.
    We begin by examining Problem~\eqref{MD_iteration_unkonwnp}:
    \begin{equation*}
    \mu^{t+1} \in \argmin_{\mu \in \mathcal{M}_{\mu_0}^{t+1}} \{ \tau \langle l^t, \mu \rangle + \Gamma( \mu, \tilde{\mu}^t) \}.
    \end{equation*}
    Since $F^t$ is a convex function and $\mathcal{M}_{\mu_0}^{t+1}$ is a convex set, the optimality conditions imply that for all $\nu^{t+1} \in \mathcal{M}_{\mu_0}^{t+1}$,
    \begin{equation*}
    \langle \tau l^t + \nabla \psi(\mu^{t+1}) - \nabla \psi(\tilde{\mu}^t), \nu^{t+1} - \mu^{t+1} \rangle \geq 0.
\end{equation*}
Recall that $\psi$ is defined in Proposition~\ref{penalization_is_bregman} as the function inducing the Bregaman divergence $\Gamma$. Re-arranging the terms and using the three points inequality for Bregman divergences \citep{bubeck2014convex} we get that,
\begin{equation*}
\begin{split}
    \tau \langle l^t, \mu^{t+1} - \nu^{t+1} &\rangle \leq \langle \nabla \psi(\mu^{t+1}) - \nabla \psi(\tilde{\mu}^t), \nu^{t+1} - \mu^{t+1} \rangle = \Gamma(\nu^{t+1}, \tilde{\mu}^t) - \Gamma(\nu^{t+1}, \mu^{t+1}) - \Gamma(\mu^{t+1}, \tilde{\mu}^t).
\end{split}
\end{equation*}
This is in particular valid for $\nu^{t+1} := \mu^{\pi^*, \hat{p}^{t+1}}$. Therefore, by adding and subtracting $\tau \langle l^t, \mu^{t} \rangle$ on the left-hand side,
\begin{equation*}
\begin{split}
    &\hspace{0.5cm}\tau \langle l^t, \mu^{t+1} - \nu^{t+1} \rangle + \tau \langle l^t, \mu^t \rangle - \tau \langle l^t, \mu^t \rangle \leq \Gamma(\nu^{t+1}, \tilde{\mu}^t) - \Gamma(\nu^{t+1}, \mu^{t+1}) - \Gamma(\mu^{t+1}, \tilde{\mu}^t) \\
&\Rightarrow \tau \langle l^t, \mu^t - \nu^{t+1} \rangle \leq \tau \langle l^t, \mu^t - \mu^{t+1} \rangle + \Gamma(\nu^{t+1}, \tilde{\mu}^t) - \Gamma(\nu^{t+1}, \mu^{t+1}) - \Gamma(\mu^{t+1}, \tilde{\mu}^t). \\
\end{split}
\end{equation*}
Then, by summing over $t \in [T]$, and by taking $\nu^{t+1} := \mu^{\pi^*, \hat{p}^{t+1}}$, we obtain
\begin{equation}\label{R_MD_partial_bound}
\begin{split}
  R_T^{policy} \leq \underbrace{\frac{1}{\tau}\sum_{t=1}^T \big[\tau \langle l^t, \mu^{t} - \mu^{t+1} \rangle - \Gamma(\mu^{t+1}, \tilde{\mu}^t) \big]}_{A} + \underbrace{\frac{1}{\tau}\sum_{t=1}^T \big[ \Gamma(\nu^{t+1}, \tilde{\mu}^t) - \Gamma(\nu^{t+1}, \mu^{t+1}) \big]}_{B}.
\end{split}
\end{equation}
The term $A$ appears due to our lack of knowledge of $F^t$ at the beginning of episode $t$ for all episodes. To remedy this, we use Young's inequality and the strong convexity of $\Gamma$. Note that if we were to consider the case where all $F^t$ are known in advance, we would not have to deal with the $A$ term. As for the term $B$, in the classic OMD proof \citep{OMD} where the set of constraints is fixed the sum of the difference between the Bregman divergences telescopes (as we would with a fixed $\nu$). However, since we are considering time-varying constraint sets, this does not happen in our case. We now proceed to find an upper bound for each term.

\paragraph{Step $1$: upper bound on $B$}
We begin by analyzing the second term of the sum in Equation~\eqref{R_MD_partial_bound}. Recall that $\nu^t:= \mu^{\pi^*, \hat{p}^t}$ for all $t \in [T]$. In order to make the Bregman divergence terms telescope we add and subtract $\Gamma(\nu^t, \mu^t) - \Gamma(\nu^t, \tilde{\mu}^t)$, obtaining 
\begin{equation*}
    \sum_{t=1}^T \Gamma(\nu^{t+1}, \tilde{\mu}^t) - \Gamma(\nu^{t+1}, \mu^{t+1}) = \underbrace{\sum_{t=1}^T \Gamma(\nu^{t+1}, \tilde{\mu}^t) - \Gamma(\nu^t, \tilde{\mu}^t)}_{(i)} + \underbrace{\sum_{t=1}^T \Gamma(\nu^t, \tilde{\mu}^t) - \Gamma(\nu^t, \mu^t)}_{(ii)} + \underbrace{\sum_{t=1}^T \Gamma(\nu^t, \mu^t) - \Gamma(\nu^{t+1}, \mu^{t+1})}_{(iii)}.
\end{equation*}
We now analyze each term. Using the definition of a Bregman divergence induced by $\psi$ we get that
\begin{equation*}
\begin{split}
    (i) &= \sum_{t=1}^T \psi(\nu^{t+1}) - \psi(\tilde{\mu}^t) - \langle \nabla \psi(\tilde{\mu}^t), \nu^{t+1} - \tilde{\mu}^t \rangle - \psi(\nu^{t}) + \psi(\tilde{\mu}^t) + \langle \nabla \psi(\tilde{\mu}^t), \nu^{t} - \tilde{\mu}^t \rangle \\
    &=\sum_{t=1}^T \psi(\nu^{t+1}) - \psi(\nu^t) + \sum_{t=1}^T \langle \nabla \psi(\tilde{\mu}^t), \nu^t - \nu^{t+1} \rangle \\
    &\leq - \psi(\nu^1) + \sum_{t=1}^T \|\nabla \psi(\tilde{\mu}^t) \|_{1, \infty} \| \nu^t - \nu^{t+1}\|_{\infty,1}, 
\end{split}
\end{equation*}
where in the last inequality we used that the first term telescopes and we apply Holder's inequality to the second term. Recall that for $v := (v_n)_{n \in [N]}$ such that $v_n \in \mathbb{R}^{\mathcal{X} \times \mathcal{A}}$, we defined $\|v\|_{\infty,1} := \sup_{ n \in [N]} \|v_n\|_1$. We now also define $\|\zeta\|_{1,\infty} := \sup_{v} \{ |\langle \zeta, v \rangle |, \|v\|_{\infty,1} \leq 1 \} = \sup_{n \in [N]} \|\zeta_n\|_1$ as the respective dual norm.

With our choice of Bregman divergence, and given the definition of $\tilde{\pi}$ in Equation~\eqref{tilde_pi}, for each $n \in [N], (x,a) \in \mathcal{X} \times \mathcal{A}$, $|\nabla \psi(\tilde{\mu}^t)(n,x,a)| = |\log(\tilde{\pi}^t_n(a|x))| \leq \log(|\mathcal{A}|/\alpha_t)$. Plugging this result with the result of Lemma~\ref{lemma:bound_consecutive_p} into the bound of $(i)$ we obtain that
\begin{equation*}
    (i) \leq - \psi(\nu^1) + \sum_{t=1}^T N \log\bigg(\frac{|\mathcal{A}|}{\alpha_t}\bigg) \|\mu^{\pi^*, \hat{p}^t} - \mu^{\pi^*, \hat{p}^{t+1}}\|_{\infty,1} \leq  - \psi(\nu^1) + 2 N^2 \sum_{t=1}^T \log\bigg(\frac{|\mathcal{A}|}{\alpha_t}\bigg) \frac{1}{t}.
\end{equation*}

As for the second term, using our definition of $\Gamma$, we obtain that 
\begin{equation*}
\begin{split}
    (ii) &= \sum_{t=1}^T \sum_{n, x,a} \mu_n^{\pi^*, \hat{p}^t}(x,a) \log \bigg(\frac{\pi_n^*(a|x)}{\tilde{\pi}_n^t(a|x)} \bigg) - \sum_{n, x,a} \mu_n^{\pi^*, \hat{p}^t}(x,a) \log \bigg(\frac{\pi_n^*(a|x)}{\pi_n^t(a|x)} \bigg) \\
    &= \sum_{t=1}^T \sum_{n,x,a} \mu_n^{\pi^*, \hat{p}^t}(x,a) \log \bigg(\frac{\pi_n^t(a|x)}{\tilde{\pi}_n^t(a|x)} \bigg) \\
    &= \sum_{t=1}^T \sum_{n,x,a} \mu_n^{\pi^*, \hat{p}^t}(x,a) \log \bigg(\frac{\pi_n^t(a|x)}{(1-\alpha_t) \pi_n^t(a|x) + \alpha/|\mathcal{A}|} \bigg) \\
    &\leq N \sum_{t=1}^T (-\log(1-\alpha_t)) \leq 2 N \sum_{t=1}^T \alpha_t,
\end{split}
\end{equation*}
where the last inequality is valid if $0 \leq \alpha_t \leq 0.5$. 

It is easy to see that the third term telescopes, therefore, as $-\Gamma(\nu^{T+1}, \mu^{T+1}) \leq 0$ as a Bregman divergence is always positive,
\begin{equation*}
    (iii) \leq \Gamma(\nu^1,\mu^1).
\end{equation*}

Before adding back the three terms, note that, for $\mu^1$ initialized such that $\nabla\psi(\mu^1) = 0$, we have $\Gamma(\nu^1, \mu^1) - \psi(\nu^1) = -\psi(\mu^1)$. Furthermore, from Lemma~\ref{upper_bound_gamma}, $-\psi(\mu^1) \leq N \log(|\mathcal{A}|)$. Therefore, 
\begin{equation}\label{bound_max_entropy}
    \Gamma(\nu^{1}, \mu^{1}) -  \psi(\nu^{1}) \leq N \log(|\mathcal{A}|).
\end{equation}

Summing over our bounds and using the Inequality~\eqref{bound_max_entropy}, we get that $B$ is upper bounded as
\begin{equation}\label{final_bound_gammas}
     \frac{1}{\tau }\sum_{t=1}^T \big[ \Gamma(\nu^{t+1}, \tilde{\mu}^t) - \Gamma(\nu^{t+1}, \mu^{t+1}) \big] \leq \frac{1}{\tau} \big[ (i) + (ii) + (iii) \big] \leq \frac{N}{\tau} \log(|\mathcal{A}|) + \frac{2 N^2}{\tau} \sum_{t=1}^T \log\bigg(\frac{|\mathcal{A}|}{\alpha_t} \bigg) \frac{1}{t} + \frac{2 N}{\tau} \sum_{t=1}^T \alpha_t.
\end{equation}

\paragraph{Step $2$: Upper bound on $A$}
It remains to delimit the first term of the bound in $R_T^{policy}$ in Equation~\eqref{R_MD_partial_bound} given by
\begin{equation}\label{adversarial_term}
    A = \frac{1}{\tau} \bigg[\sum_{t=1}^T \tau \langle l^t, \mu^{t} - \mu^{t+1} \rangle - \Gamma(\mu^{t+1}, \tilde{\mu}^t)\bigg],
\end{equation}
representing what we pay for not knowing the loss function in advance. For that we use Young's inequality \citep{MD}. 

Recall that Young's inequality states that for any $\sigma > 0$, for any dual norms,
\begin{equation*}
    \langle a, b \rangle \leq \frac{1}{2 \sigma} \|a\|^2 + \frac{\sigma}{2} \|b\|^2_*.
\end{equation*}
Therefore, for any $\sigma > 0$ to be optimized later, and for each episode $t \in [T]$,
\begin{equation}\label{inequality_after_young}
    \tau \langle l^t, \mu^{t} - \mu^{t+1} \rangle - \Gamma(\mu^{t+1}, \tilde{\mu}^t) \leq \frac{\tau^2 \|l^t\|^2_{1, \infty}}{2 \sigma} + \frac{\sigma}{2} \|\mu^{t} - \mu^{t+1} \|_{\infty,1}^2 - \Gamma(\mu^{t+1}, \tilde{\mu}^t),
\end{equation}
where recall that for $v := (v_n)_{n \in [N]}$ such that $v_n \in \mathbb{R}^{\mathcal{X} \times \mathcal{A}}$, we defined $\|v\|_{\infty,1} := \sup_{ n \in [N]} \|v_n\|_1$, and we let $\|\zeta\|_{1,\infty} := \sup_{v} \{ |\langle \zeta, v \rangle |, \|v\|_{\infty,1} \leq 1 \} = \sup_{n \in [N]} \|\zeta_n\|_1$ as the respective dual norm.

From Lemma~\ref{prop:D_decomp} and inequality~\eqref{gamma_strong_convex} stating the strong convexity of $\psi$, we have that for all $t \in [T]$ 
\begin{equation}\label{gamma_bound_norm_2}
\begin{split}
    \Gamma(\mu^{t+1}, \tilde{\mu}^{t}) &= D(\mu_{1:N}^{t+1}, \mu_{1:N}^{\tilde{\pi}^{t},\hat{p}^{t+1}}) \geq \frac{1}{2} \|\mu^{t+1} - \mu^{\tilde{\pi}^{t},\hat{p}^{t+1}} \|_{\infty,1}^2
\end{split}
\end{equation}
where recall that $\mu_{1:N}$ is the joint state-action distribution while that $\mu := (\mu_n)_{n \in[N]}$ is the sequence of state-action distributions. 

Using that for any vectors $a,b,c \in \mathbb{R}^d$, and that for any norm $\|\cdot\|$, $\|a-b\|^2 \leq 2 \big(\|a-c\|^2 + \|b-c\|^2\big)$, we then have by Equation~\eqref{gamma_bound_norm_2}
\begin{equation}\label{Gamma_Young_bound_ourcase}
\begin{split}
       \frac{1}{4} \|\mu^{t} - \mu^{t+1} \|_{\infty,1}^2  - \Gamma(\mu^{t+1}, \tilde{\mu}^{t}) &\leq \frac{1}{4} \|\mu^{t} - \mu^{t+1} \|_{\infty,1}^2 - \frac{1}{2}\| \mu^{\tilde{\pi}^t,\hat{p}^{t+1}} - \mu^{t+1}\|_{\infty,1}^2 \\
   &\leq \frac{1}{2} \big( \|\mu^{t} - \mu^{\tilde{\pi}^t,\hat{p}^{t+1}} \|_{\infty,1}^2 + \|\mu^{\tilde{\pi}^t,\hat{p}^{t+1}} - \mu^{t+1} \|_{\infty,1}^2 \big)  - \frac{1}{2}\| \mu^{\tilde{\pi}^t,\hat{p}^{t+1}} - \mu^{t+1}\|_{\infty,1}^2 \\
   &= \frac{1}{2} \|\mu^{t} - \mu^{\tilde{\pi}^t,\hat{p}^{t+1}} \|_{\infty,1}^2. 
\end{split}
\end{equation}

To bound $\|\mu^{t} - \mu^{\tilde{\pi}^t,\hat{p}^{t+1}} \|^2_{\infty,1}$ we first use Lemma~\ref{proof_mu_tilde} which gives
\begin{equation*}
    \|\mu^{t} - \mu^{\tilde{\pi}^t,\hat{p}^{t+1}} \|_{\infty,1} \leq \sum_{n=1}^N \sum_{x,a} \mu_i^t(x,a) \|\hat{p}^t_{i+1}(\cdot|x,a) - \hat{p}^{t+1}_{i+1}(\cdot|x,a) \|_1 + 2 N \alpha_t.
\end{equation*}
Then, by Equation~\eqref{difference_consecutive_p}, $\|\hat{p}^t_{i+1}(\cdot|x,a) - \hat{p}^{t+1}_{i+1}(\cdot|x,a) \|_1 \leq 2/t$ for all $t \in [T]$, therefore
\begin{equation*}
    \|\mu^{t} - \mu^{\tilde{\pi}^t,\hat{p}^{t+1}} \|_{\infty,1}^2 \leq  \bigg(\frac{2 N}{t} + 2 N \alpha_t \bigg)^2.
\end{equation*}

Therefore, plugging into Equation~\eqref{inequality_after_young} with $\sigma = 1/2$ yields,
\begin{equation*}
    \tau \langle l^t, \mu^{t} - \mu^{t+1} \rangle - \Gamma(\mu^{t+1}, \tilde{\mu}^t)\leq \tau^2 \|l^t\|^2_{1, \infty} + \frac{1}{2}  \bigg(\frac{2 N}{t} + 2 N \alpha_t \bigg)^2.
\end{equation*}

Summing over $t \in [T]$, and $\|l^t\|_{1,\infty} \leq L := l N$ as showed in Lemma~\ref{proof_conv_MD_MFC} then entails:
\begin{equation}\label{final_bound_adversarial}
    A \leq \tau L^2 T + \frac{1}{2 \tau} \sum_{t=1}^T \bigg(\frac{2 N}{t} + 2 N \alpha_t \bigg)^2.
\end{equation}

\paragraph{Conclusion}

Finally, by replacing the final bounds of Equations~\eqref{final_bound_gammas} and~\eqref{final_bound_adversarial}, we obtain 
\begin{equation*}
    R_T^{policy} \leq A + B \leq \tau T  L^2 + \frac{2 N^2}{\tau} \sum_{t=1}^T \bigg(\frac{1}{t} + \alpha_t \bigg)^2 +  \frac{N}{\tau} \log(|\mathcal{A}|) + \frac{2 N^2}{\tau}  \sum_{t=1}^T \log\bigg(\frac{|\mathcal{A}|}{\alpha_t} \bigg) \frac{1}{t} + \frac{2 N}{\tau} \sum_{t=1}^T \alpha_t.
\end{equation*}

Let
\begin{equation*}
    b := \bigg( 2 N^2 \bigg[ \sum_{t=1}^T \bigg(\frac{1}{t} + \alpha_t \bigg)^2 + \sum_{t=1}^T \log\bigg(\frac{|\mathcal{A}|}{\alpha_t}\bigg) \frac{1}{t}\bigg] + 2 N \sum_{t=1}^T \alpha_t + N \log(|\mathcal{A}|) \bigg)^{\frac{1}{2}}.
\end{equation*}
Optimising over $\tau = \frac{b}{L \sqrt{T}}$, 
\begin{equation*}
     R_T^{policy} \leq 2 L b \sqrt{T} = 2 l N b \sqrt{T},
\end{equation*}
concluding the proof.

In particular, if $\alpha_t = \frac{1}{T}$ for all $t \in [T]$, we have $R_T^{policy} \leq \sqrt{T} \log(T)$.

\end{proof}

% \subsection{Proof of Corollary~\ref{cor:md_bound_special_case}}
% \begin{proof}
% If we consider a general case in which $\hat{\beta}(t) \propto C/t$ with $C$ a constant, then $\sum_{t=1}^T \hat{\beta}(t) \propto C \log(T)$. If we choose $\alpha_t \propto \frac{1}{t+1}$ then $\sum_{t=1}^T \alpha_t \propto \log{T}$, and $\sum_{t=1}^T \hat{\beta}(t)\log\bigg(\frac{|\mathcal{A}|}{\alpha_t} \bigg) \propto \sum_{t=1}^T C \frac{\log(|\mathcal{A}| t)}{t} \leq C \frac{\log(T)^2}{2} + C \log(|\mathcal{A}| T)$. Therefore, optimising over $\tau$ in this case gives
% \begin{equation*}
%     R_T^{policy} \leq O\bigg( L N \sqrt{C T \bigg(\log(|\mathcal{A}|T) + \frac{\log^2(T)}{2}\bigg)} \bigg).
% \end{equation*}

% In the special case of categorical distribution we have $C = |I|$, where $|I|$ indicate the number of possible external noises.
% \end{proof}

%%%%%%%%%%%%%%%%%%%%%%%%%%%%%%%%%%%%%%%%%%%%%%%%%%%%%%%%%%%%%%%%%%

\section{Bounds with unknown $g$}\label{bounds_unknown_g}
Now suppose that $g_n$ and $h_n$ in the model of Equation~\eqref{dynamics} are unknown. In this case, we have no information about the probability kernel, and the exploration/exploitation dilemma arises. 

In order to learn the complete probability kernel, we need to modify the learning model slightly. Let us suppose that, at each episode $t$ the learner maintains the number of visit counts to each episode $(x,a)$ at time step $n$, denoted $N_n^t(x,a)$, and the number of times this event is followed by a transition to a state $x'$, denoted $M_n^t(x'|x,a)$, that is
\begin{equation*}
    \begin{split}
        M_n^t(x'|x,a) &= \sum_{s=1}^t \mathds{1}_{\{x_{n+1}^s = x', x_n^s = x, a_n^s = a\}} \\
        N_n^t(x,a) &= \sum_{s=1}^t \mathds{1}_{\{x_n^s = x, a_n^s = a\}}.
    \end{split}
\end{equation*}
To ease notations, we take $M=1$ in this section. We define $\hat{p}^t$, at each $(x,a)$ and time step $n+1$ by
\begin{equation}\label{proba_counts}
    \hat{p}_{n+1}^t(x'|x,a) = \frac{M_n^t(x'|x,a)}{\max\{1,N_n^t(x,a)\}}.
\end{equation}
The following lemma ensures the true probability kernel $p$ lies at a certain distance from this estimation of $\hat{p}^t$ with high probability. 

\begin{lemma}[\cite{UCRL-2, neu12}]\label{lemma:proba_difference}
    For any $0 < \delta < 1$, 
    \begin{equation*}
    \begin{split}
        \|p_n(\cdot|x,a) - \hat{p}_n^t(\cdot|x,a)\|_1 \leq 
        \sqrt{\frac{4 |\mathcal{X}|\log\left(\frac{|\mathcal{X}| |\mathcal{A}| N T}{\delta}\right)}{\max{\{1,N_n^t(x,a)}\}} }
    \end{split}
    \end{equation*}
    holds, with a probability of at least $1-\delta$, for simultaneously all $(x,a) \in \mathcal{X} \times \mathcal{A}$, all $n \in [N]$, and all episodes $t \in [T]$.
\end{lemma}

Recall that the regret $R_T$ is decomposed as $R_T := R_T^{MDP}\big((\pi^t)_{t \in [T]}) \big) + R_T^{policy} + R_T^{MDP}\big(\pi^*\big)$, and we treat each term separately. The regret bound for  $R_T^{policy}$ follows the same procedure as in Proposition~\ref{thm:bound_regret_md}. However, the bound on the terms of $R_T^{MDP}$ are different, as we must now ensure that we visit all necessary state-action pairs $(x,a)$ sufficiently often. This also means that the bound for $R_T^{MDP}\big((\pi^t)_{t \in [T]}) \big)$ is different from the bound of $R_T^{MDP}\big(\pi^*\big)$. For bounding both terms related to $R_T^{MDP}$ we use a similar approach as in UC-O-REPS \citep{UC-O-REPS}.

\begin{lemma}\label{lemma:martingale}
For $0 < \delta < 1$,
    \begin{equation*}
    \begin{split}
        &\sup_{n \in [N]} \sum_{t=1}^T \sum_{i=0}^{n-1} \sum_{x,a} \mu_i^{\pi^t, p}(x,a) \|p_{i+1}(\cdot|x,a) - \hat{p}_{i+1}^t(\cdot|x,a) \|_1 \\
        &\leq (\sqrt{2} +1) N |\mathcal{X}| \sqrt{4 |\mathcal{A}| T \log\bigg(\frac{T |\mathcal{X}||\mathcal{A}|N}{\delta}\bigg)} + 2 N |\mathcal{X}| \sqrt{2 T \log\bigg(\frac{N}{\delta}\bigg)}
    \end{split}
    \end{equation*}
with probability $1-2\delta$.
\end{lemma}
\begin{proof}
    Using Lemma $19$ from \cite{UCRL-2}, we have that
    \begin{equation*}
        \sum_{t=1}^T \frac{\mathds{1}_{\{x_n^t=x, a_n^t=a\}}}{N_n^t(x,a)} \leq (\sqrt{2} + 1) \sqrt{N_n^T(x,a)},
    \end{equation*}
    and by Jensen's inequality,
    \begin{equation}\label{bound_N}
        \sum_{x,a} \sum_{t=1}^T \frac{\mathds{1}_{\{x_n^t=x, a_n^t=a\}}}{N_n^t(x,a)} \leq (\sqrt{2} + 1) \sum_{x,a} \sqrt{|\mathcal{X}| |\mathcal{A}| T}.
    \end{equation}

    Let $(x_n^t,a_n^t)_{n\in[N]}$ be the trajectory made by policy $\pi^t$ for all $t \in [T]$. Therefore,
    \begin{align}\label{decomposition_martingale_t}
        &\sum_{i=0}^{n-1} \sum_{x,a} \mu_i^{\pi^t,p}(x,a)\|p_{i+1}(\cdot|x,a) - \hat{p}_{i+1}^t(\cdot|x,a) \|_1 \\
        &\leq \sum_{i=0}^{n-1} \|p_{i+1}(\cdot|x_i^t,a_i^t) - \hat{p}_{i+1}^t(\cdot|x_i^t,a_i^t) \|_1 + \sum_{i=0}^{n-1} \sum_{x,a} \big( \mu_i^{\pi^t, p}(x,a) - \mathds{1}_{\{x_i^t=a,a_i^t=a\}} \big)  \|p_{i+1}(\cdot|x,a) - \hat{p}_{i+1}^t(\cdot|x,a) \|_1 \nonumber
    \end{align}
    By Lemma~\ref{lemma:proba_difference}, with probability at least $1-\delta$, simultaneously for all $i \in [N]$ we have
    \begin{align}\label{proba_bound_counts}
        \sum_{t=1}^T \|p_{i+1}(\cdot|x_i^t,a_i^t) - \hat{p}_{i+1}^t(\cdot|x_i^t,a_i^t) \|_1 &\leq \sum_{t=1}^T \sqrt{\frac{4 |\mathcal{X}| \log\bigg(\frac{T |\mathcal{X}| |\mathcal{A}| N}{\delta} \bigg)}{\max{\{1, N_i^t(x_i^t,a_i^t)\}}}} \nonumber \\
        &\leq \sum_{x,a} \sum_{t=1}^T \mathds{1}_{\{x_i^t=x, a_i^t=a\}} \sqrt{\frac{4 |\mathcal{X}| \log\bigg(\frac{T |\mathcal{X}| |\mathcal{A}| N}{\delta} \bigg)}{\max{\{1, N_i^t(x,a)\}}}} \nonumber \\
        &\leq (\sqrt{2} +1) \sqrt{4 |\mathcal{X}|^2 |\mathcal{A}| T \log\bigg(\frac{T |\mathcal{X}| |\mathcal{A}| N}{\delta}\bigg)},
        \end{align}
where, for the last inequality, we use the result of Equation~\eqref{bound_N}.

As for the second term, note that for all $i \in [N]$ and $x \in \mathcal{X}$, 
\begin{equation*}
 \bigg( \sum_a \big(\mu_i^{\pi^t, p}(x,a) - \mathds{1}_{\{x_i^t=x,a_i^t=a\}} \big) \bigg)   
\end{equation*}
forms a martingale difference with respect to the trajectory $(x_0^s,a_0^s, \ldots, x_N^s, a_N^s)_{s \in [T]}$ (the expectation of the term conditional on the past trajectory is zero). Therefore, by Azuma-Hoeffding inequality,
\begin{equation*}
    \mathbb{P} \bigg[\sum_{t=1}^T \sum_a \big(\mu_i^{\pi^t, p}(x,a) - \mathds{1}_{\{x_i^t=x,a_i^t=a\}} \big) \geq \epsilon \bigg] \leq \exp\bigg(\frac{-2 \epsilon^2}{4 T} \bigg).
\end{equation*}
Taking the union bound over $i \in [N]$, we get that with probability $1 -\delta$, simultaneously for all $i \in [N]$, and considering that $\|p_{i+1}(\cdot|x,a) - p_{i+1}^t(\cdot|x,a) \|_1 \leq 2$, 
\begin{equation}\label{martingale_bound_counts}
\begin{split}
        \sum_{t=1}^T \sum_{x,a}  \big(\mu_i^{\pi^t, p}(x,a) - \mathds{1}_{x_i^t=x,a_i^t=a} \big) \|p_{i+1}(\cdot|x,a) - \hat{p}_{i+1}^t(\cdot|x,a) \|_1  \leq 2 |\mathcal{X}| \sqrt{2 T \log\bigg(\frac{N}{\delta}\bigg)}. 
\end{split}
\end{equation}
    Plugging the bounds on Equation~\eqref{proba_bound_counts} and~\eqref{martingale_bound_counts} into Equation~\eqref{decomposition_martingale_t}, we get that with probability $1-2\delta$, 
    \begin{equation*}
    \begin{split}
                &\sup_{n \in [N]} \sum_{t=1}^T \sum_{i=0}^{n-1} \sum_{x,a} \mu_i^{\pi^t, p} \|p_{i+1}(\cdot|x,a) - \hat{p}_{i+1}^t(\cdot|x,a) \|_1 \\
                &\leq \sup_{n \in [N]} \sum_{i=0}^{n-1} \bigg[ (\sqrt{2} +1)  \sqrt{4 |\mathcal{X}|^2 |\mathcal{A}| T \log\bigg(\frac{T |\mathcal{X}||\mathcal{A}|N}{\delta}\bigg)} + 2 |\mathcal{X}| \sqrt{2 T \log\bigg(\frac{N}{\delta}\bigg)} \bigg] \\
                &\leq (\sqrt{2} +1) N |\mathcal{X}| \sqrt{4 |\mathcal{A}| T \log\bigg(\frac{T |\mathcal{X}||\mathcal{A}|N}{\delta}\bigg)} + 2 N |\mathcal{X}| \sqrt{2 T \log\bigg(\frac{N}{\delta}\bigg)}.
    \end{split}
    \end{equation*}
\end{proof}

The result of Lemma~\ref{lemma:martingale} allows us to state the following proposition bounding the term $R_T^{MDP} \big( ( \pi^t)_{t \in [T]} \big)$:
\begin{proposition}\label{bound_rt_pit}
   We consider an episodic MDP with finite state space $\mathcal{X}$, finite action space $\mathcal{A}$, episodes of length $N$, and probability kernel $p:=(p_n)_{n \in [N]}$.  We let $F^t := \sum_{n=1}^N f_n^t$ convex with $f_n^t$ $\ell$-Lipschitz with respect to the norm $\|\cdot\|_1$ for all $n \in [N], t \in [T]$. We consider the probability estimation per iteration as in Equation~\eqref{proba_counts}. Then, with probability $1-\delta$, Greedy MD-CURL obtains,
    \begin{equation*}
        R_T^{MDP} \big( (\pi^t)_{t \in [T]} \big) \leq (\sqrt{2} +1) \ell N^2 |\mathcal{X}| \sqrt{4 |\mathcal{A}| T \log\bigg(\frac{T |\mathcal{X}||\mathcal{A}|N}{\delta}\bigg)} + 2 \ell N^2 |\mathcal{X}| \sqrt{2 T \log\bigg(\frac{N}{\delta}\bigg)}.
    \end{equation*}
\end{proposition}
\begin{proof}

    Recall that, given the convexity of $F^t$ and by applying Holder's inequality using that if $f_n^t$ are $\ell$-Lipschitz with respect to $\|\cdot\|_1$ then $F$ is $L$-Lipschitz with respect to $\|\cdot\|_{\infty,1}$ for $L = \ell N$ (see Lemma~\ref{proof_conv_MD_MFC}), we obtain that
    \begin{equation*}
\begin{split}
    R_T^{MDP}\big((\pi^t)_{t \in [T]}\big) &\leq \sum_{t=1}^T \langle \nabla F^t(\mu^{\pi^t,p}), \mu^{\pi^t,p} - \mu^{\pi^t,\hat{p}^t} \rangle \\
    &\leq \sum_{t=1}^T \|\nabla F^t(\mu^{\pi^t,p})\|_* \sup_{n \in [N]} \|\mu_n^{\pi^t,p} - \mu_n^{\pi^t,\hat{p}^t} \|_1 \\
    &\leq L \sup_{n \in [N]}  \sum_{t=1}^T\|\mu_n^{\pi^t,p} - \mu_n^{\pi^t,\hat{p}^t} \|_1.
\end{split}
\end{equation*}
The result then follows from the application of Lemma~\ref{lemma:bound_norm_mu} and Lemma~\ref{lemma:martingale}.
\end{proof}

To complete the bound on the regret $R_T$, we need to bound $R_T^{MDP} \big( \pi^* \big)$. For this, we need Lemma~\ref{pis_are_positive}, which states that the Greedy MD-CURL algorithm always computes policies that are lower bounded if $-\nabla f_n^t(x,a)(\mu_n) \in [0,1]$ for all $(x,a) \in \mathcal{X} \times \mathcal{A}$, all $\mu_n \in \Delta_{\mathcal{X} \times \mathcal{A}}$, $n \in [N]$ and $t \in [T]$. Proposition~\ref{bound_rt_pi*} states the bound for $R_T^{MDP} \big( \pi^* \big)$.

\begin{lemma}\label{pis_are_positive}
    Let $(\pi^t)_{t \in [T]}$ be the sequence of policies obtained after computing $T$ episodes of Greedy MD-CURL with $\alpha_t \in (0,1/2)$ and objective functions $F^t = \sum_{n=1}^N f_n^t$ such that $-\nabla f_n^t(\mu_n)(x,a) \in [0,1]$. Consequently, there is $\xi \in (0,1)$ such that for all $(x,a) \in \mathcal{X} \times \mathcal{A}$, for all $n \in [N]$, and for all episodes $t \in [T]$, $\pi_n^t(a|x) \geq \xi$. 
\end{lemma}
\begin{proof}
    At each episode $t$, we compute $\pi^{t+1} := \text{MD-CURL}(1, \tilde{\pi}^t \backslash \pi^t , F^t, \hat{p}^{t+1}, \mu_0, \tau),$ where $\tilde{\pi}^t = (1-\alpha_t) \pi^t + \alpha_t \frac{1}{|\mathcal{A}|}$, and the other parameters are defined in Algorithm~\ref{alg:greedy_MD}. 

    From its definition, we can see that $\tilde{\pi}^t_n(a|x) \geq  \frac{\alpha_t}{|\mathcal{A}|}$. The closed form solution of one iteration of MD-CURL with the given parameters gives
    \begin{equation*}
        \pi_n^{t+1}(a|x) = \frac{\tilde{\pi}^t_n(a|x) \exp\big(\tau \tilde{Q}^t_n(x,a)\big)}{\sum_{a' \in \mathcal{A}} \tilde{\pi}^t_n(a'|x) \exp\big(\tau \tilde{Q}^t_n(x,a')\big)  },
    \end{equation*}
    where $\tilde{Q}^t_n(x,a)$ is defined in Equation~\eqref{Bellman_Q_tilde} with $-\nabla f_n^t(x,a)(\mu_n^t)$ and $\tilde{\pi}^t$ at the place of $\pi^k$. As $-\nabla f_n^t(x,a)(\mu_n^t) \in [0,1]$, we have $1 \leq \exp\big(\tau \tilde{Q}^t_n(x,a')\big) \leq \exp(\tau (N-n))$. Therefore, we have $\pi_n^{t+1}(a|x) \geq \frac{\alpha_t}{|\mathcal{A}| \exp(\tau (N-n))}$ for all steps $n \in [N]$ and $(x,a)$. 

    Taking $\xi := \min_{t \in [T], n \in [N]} \frac{\alpha_t}{|\mathcal{A}| \exp(\tau (N-n))} = \frac{1}{|\mathcal{A}| \exp(\tau N)} \min_{t \in [T]} \alpha_t$, we then have for all $(x,a) \in \mathcal{X} \times \mathcal{A}$, for all $n \in [N]$, and for all episodes $t \in [T]$, $\pi_n^t(a|x) \geq \xi$.
\end{proof}

\begin{proposition}\label{bound_rt_pi*}
   We consider an episodic MDP with finite state space $\mathcal{X}$, finite action space $\mathcal{A}$, episodes of length $N$, and probability kernel $p:=(p_n)_{n \in [N]}$. We let $F^t := \sum_{n=1}^N f_n^t$ convex with $f_n^t$ $\ell$-Lipschitz with respect to the norm $\|\cdot\|_1$ for all $n \in [N], t \in [T]$. We let $\xi$ be the lower bound of $\pi_n^t(a|x)$ for all $n, t, (x,a)$ defined as in Lemma~\ref{pis_are_positive}. We consider the probability estimation per iteration as in Equation~\eqref{proba_counts}. Then, with probability $1-2 \delta$, Greedy MD-CURL obtains,
    \begin{equation*}
        R_T^{MDP} \big( \pi^* \big) \leq  \frac{1}{\xi^N}  \bigg[(\sqrt{2} +1) \ell N^2 |\mathcal{X}| \sqrt{4 |\mathcal{A}| T \log\bigg(\frac{T |\mathcal{X}||\mathcal{A}|N}{\delta}\bigg)} + 2 \ell N^2 |\mathcal{X}| \sqrt{2 T \log\bigg(\frac{N}{\delta}\bigg)} \bigg].
    \end{equation*}
\end{proposition}
\begin{proof}
     Recall that, given the convexity of $F^t$ and by applying Holder's inequality using that if $f_n^t$ are $\ell$-Lipschitz with respect to $\|\cdot\|_1$ then $F$ is $L$-Lipschitz with respect to $\|\cdot\|_{\infty,1}$ for $L = \ell N$ (see Lemma~\ref{proof_conv_MD_MFC}), we obtain that
    \begin{equation*}
\begin{split}
    R_T^{MDP}\big(\pi^* \big) &\leq \sum_{t=1}^T \langle \nabla F^t(\mu^{\pi^*,\hat{p}^t}), \mu^{\pi^*,\hat{p}^t} - \mu^{\pi^*,p} \rangle \\
    &\leq \sum_{t=1}^T \|\nabla F^t(\mu^{\pi^t,p})\|_* \sup_{n \in [N]} \|\mu_n^{\pi^*,\hat{p}^t} - \mu_n^{\pi^*,p} \|_1 \\
    &\leq L \sup_{n \in [N]}  \sum_{t=1}^T\|\mu_n^{\pi^*,\hat{p}^t} - \mu_n^{\pi^*,p} \|_1.
\end{split}
\end{equation*}
As $\pi_n^t(a|x) \geq \xi$ for all $n \in [N], t\in [T]$ and $(x,a) \in \mathcal{X} \times \mathcal{A}$, we get $\frac{\mu_n^{\pi^*,p}(x,a)}{\mu_n^{\pi^t,p}(x,a)} \leq \frac{1}{\xi^n}$. This can be demonstrated recursively: suppose it's true for $n$, then for $n+1$, by definition
\begin{equation*}
    \begin{split}
        \mu_{n+1}^{\pi^t,p}(x,a) &= \sum_{x',a'} \mu_n^{\pi^t,p}(x',a') p_{n+1}(x|x',a') \pi_{n+1}^t(a|x) \\
        &\leq \sum_{x',a'}\xi^n \mu_n^{\pi^*,p}(x',a') p_{n+1}(x|x',a') \pi_{n+1}^*(a|x) \xi \\
        &= \mu_{n+1}^{\pi^*,p}(x,a) \xi^{n+1}.
    \end{split}
\end{equation*}

Using Lemma~\ref{lemma:bound_norm_mu}, and Proposition~\ref{bound_rt_pit}, we get
\begin{equation*}
\begin{split}
        \sup_{n\in [N]} \sum_{t=1}^T\|\mu_n^{\pi^*,\hat{p}^t} - \mu_n^{\pi^*,p} \|_1 &\leq   \sup_{n\in [N]}  \sum_{t=1}^T \sum_{i=0}^{n-1} \sum_{x,a} \mu_i^{\pi^*, p}(x,a)\|p_{i+1}(\cdot|x,a) - p_{i+1}^t(\cdot|x,a) \|_1 \\
        &\leq   \sup_{n\in [N]}  \sum_{t=1}^T \sum_{i=0}^{n-1} \frac{1}{\xi^i} \sum_{x,a} \mu_i^{\pi^t, p}(x,a)\|p_{i+1}(\cdot|x,a) - p_{i+1}^t(\cdot|x,a) \|_1 \\
        &\leq \sup_{n\in [N]} \sum_{i=0}^{n-1} \frac{1}{\xi^i} \bigg[ (\sqrt{2} +1)  \sqrt{4 |\mathcal{X}|^2 |\mathcal{A}| T \log\bigg(\frac{T |\mathcal{X}||\mathcal{A}|N}{\delta}\bigg)} + 2 |\mathcal{X}| \sqrt{2 T \log\bigg(\frac{N}{\delta}\bigg)} \bigg] \\
        &\leq \frac{N}{\xi^N} \bigg[ (\sqrt{2} +1)  \sqrt{4 |\mathcal{X}|^2 |\mathcal{A}| T \log\bigg(\frac{T |\mathcal{X}||\mathcal{A}|N}{\delta}\bigg)} + 2 |\mathcal{X}| \sqrt{2 T \log\bigg(\frac{N}{\delta}\bigg)} \bigg]
\end{split}
\end{equation*}
where the third inequality is obtained by following the same steps of the proof from Lemma~\ref{lemma:martingale}.
\end{proof}

\paragraph{Conclusion: bounding $R_T$}
We join the propositions~\ref{bound_rt_pit} and~\ref{bound_rt_pi*} bounding $R_T^{MDP} \big( (\pi^t)_{t \in [T]} \big)$ and $R_T^{MDP} \big( \pi^* \big)$ when playing Greedy MD-CURL with $g_n$ and $h_n$ unknown, and Proposition~\ref{thm:bound_regret_md} bounding $R_T^{policy}$ which remains general regardless of prior knowledge of $g_n$ and $h_n$. 

Here, we show regret in terms of the number of episodes $T$ and do not worry about other constant terms. We use $\lesssim$ to denote an inequality up to constant or logarithmic terms independent of $T$. To simplify, we take $\alpha_t = \alpha$ for all $t \in [T]$. Therefore, $\alpha$ and $\tau$ are the parameters to be optimized. We hypothesize that $\alpha < 1$, $T \alpha \geq \log\big( \frac{1}{\alpha} \big) \log(T)$, and $\tau \leq \frac{1}{N}$. We will later verify when the optimized $\alpha$ and $\tau$ satisfy these conditions.

From Proposition~\ref{bound_rt_pit}, we have 
\[R_T^{MDP} \big( (\pi^t)_{t \in [T]} \big) \lesssim  \sqrt{T \log(T)}. \] 
From Proposition~\ref{thm:bound_regret_md},  
\[R_T^{policy} \lesssim \tau T + \frac{T}{\tau} \alpha^2 + \frac{1}{\tau} \log\bigg(\frac{1}{\alpha} \bigg) \log(T) + \frac{T}{\tau} \alpha \lesssim \tau T +  \frac{T}{\tau} \alpha.\]
From Proposition~\ref{bound_rt_pi*}, we have
\[R_T^{MDP} \big( \pi^* \big) \lesssim \xi^{-N} \sqrt{T \log(T)} \lesssim \alpha^{-N} \sqrt{T \log(T)},\]
where $\xi^{-1} = \bigg(\frac{\alpha}{|\mathcal{A}| \exp(\tau N)}\bigg) \lesssim \alpha^{-1}$.

Therefore,
\begin{equation*}
    R_T = R_T^{MDP} \big( (\pi^t)_{t \in [T]} \big) + R_T^{policy} + R_T^{MDP} \big( \pi^* \big) \lesssim \tau T + \frac{T}{\tau} \alpha + \alpha^{-N} \sqrt{T \log(T)}.
\end{equation*}

We first optimize over $\tau$. For $\tau = \alpha^{\frac{1}{2}}$, then 
\[
R_T \lesssim \alpha^{\frac{1}{2}} T + \alpha^{-N} \sqrt{T \log(T)}.
\]
Then, we optimize over $\alpha$. For $\alpha = T^{-\frac{1}{2N +1}}$, 
\[
R_T \lesssim T^{\frac{4N + 1}{4N +2}}.
\]

If $T > 1$, then the conditions $\alpha <1$ and $T \alpha \geq \log\big( \frac{1}{\alpha} \big) \log(T)$ are satisfied. For $T \geq N^{4N +2}$, then $\tau \leq 1/N$. 

In a classic non-episodic online learning scenario, or in an episodic MDP with stationary probability kernels, we would not incur the term on $\xi^N$ but only $\xi$. This would reduce the final regret bound to $O \big( T^\frac{5}{6} \big)$ for any $T \geq 1$. That is for example the case of the showcase experiments in Section~\ref{experiments} and Appendix~\ref{more_experiments}.

%%%%%%%%%%%%%%%%%%%%%%%%%%%%%%%%%%%%%%%%%%%%%%%
\section{Additional experiments}\label{more_experiments}

\subsection{MD-CURL known probability kernel}

We present the state distribution induced by the policies computed with MD-CURL in the offline optimization scenario when both $g_n$ and $h_n$ are known for varying steps $n$ and iterations $k$. The episode length is fixed to $N=100$ for all experiments. We illustrate the Entropy Maximization problem in Figure~\ref{MD-CURL_maxentropy_knownp} and the Multi-Objective problem in Figure~\ref{MD-CURL_multiobj_knownp}, and show that MD-CURL achieves the main goal in both cases.

\begin{figure}[H]
     \centering
     \begin{subfigure}[h]{0.3\textwidth}
         \centering
         \includegraphics[width=\textwidth]{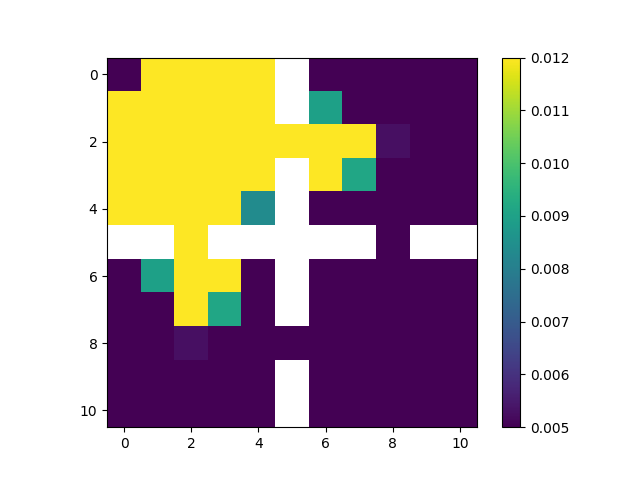}
         \caption{$n = 10, k=10$}
     \end{subfigure}
     \hfill
     \begin{subfigure}[h]{0.3\textwidth}
         \centering
         \includegraphics[width=\textwidth]{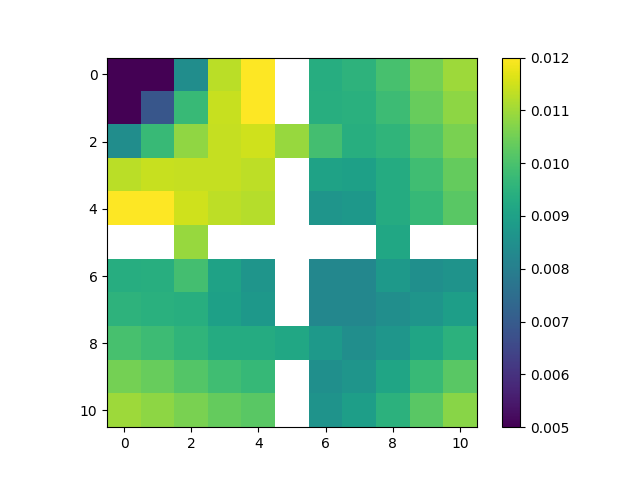}
         \caption{$n = 40, k=10$}
     \end{subfigure}
     \hfill
     \begin{subfigure}[h]{0.3\textwidth}
         \centering
         \includegraphics[width=\textwidth]{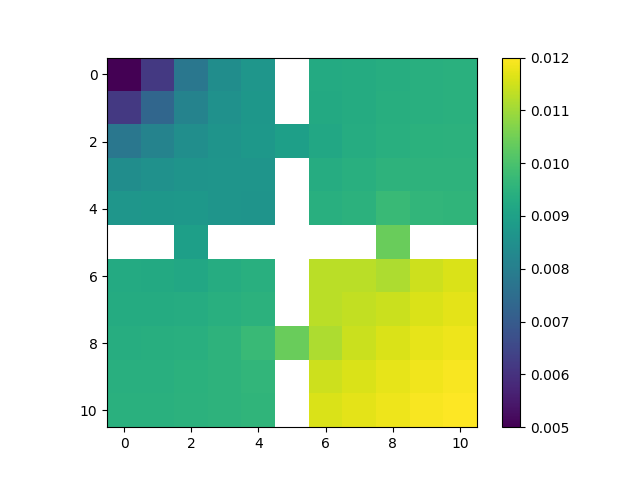}
         \caption{$n = 100, k=10$}
     \end{subfigure}
          \begin{subfigure}[h]{0.3\textwidth}
         \centering
         \includegraphics[width=\textwidth]{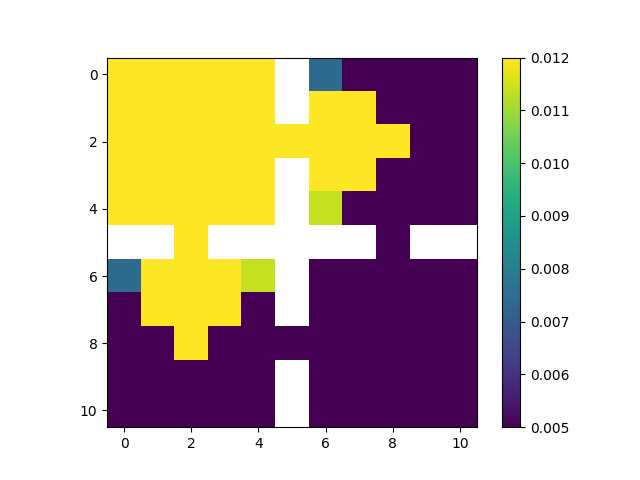}
         \caption{$n = 10, k=50$}
     \end{subfigure}
     \hfill
     \begin{subfigure}[h]{0.3\textwidth}
         \centering
         \includegraphics[width=\textwidth]{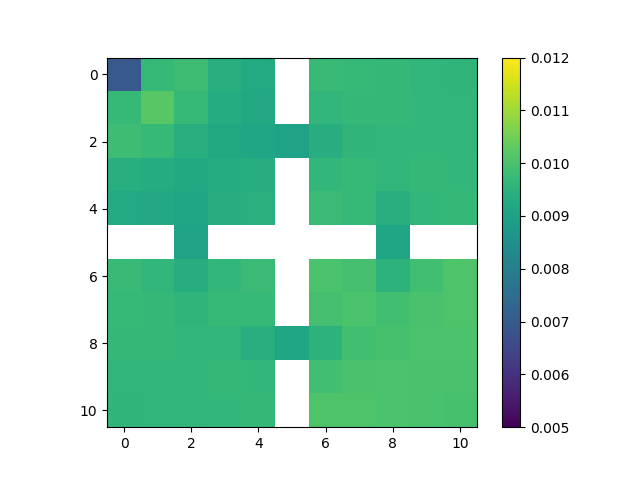}
         \caption{$n = 40, k=50$}
     \end{subfigure}
     \hfill
     \begin{subfigure}[h]{0.3\textwidth}
         \centering
         \includegraphics[width=\textwidth]{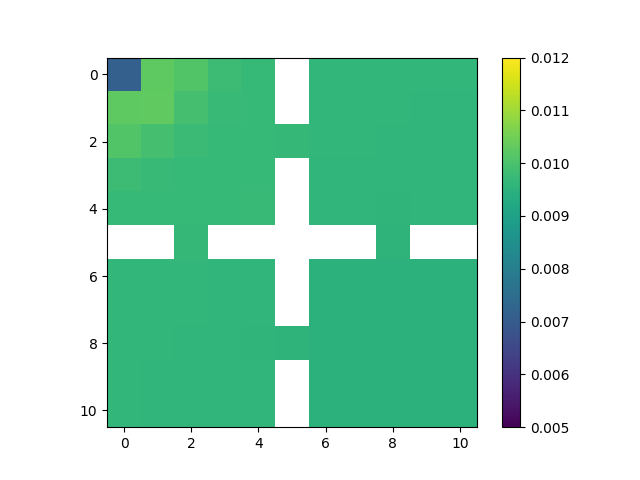}
         \caption{$n = 100, k=50$}
     \end{subfigure}
             \begin{subfigure}[h]{0.3\textwidth}
         \centering
         \includegraphics[width=\textwidth]{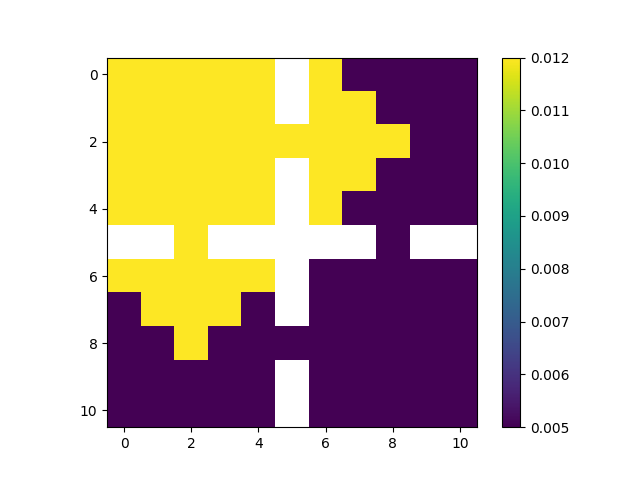}
         \caption{$n = 10, k=500$}
     \end{subfigure}
     \hfill
     \begin{subfigure}[h]{0.3\textwidth}
         \centering
         \includegraphics[width=\textwidth]{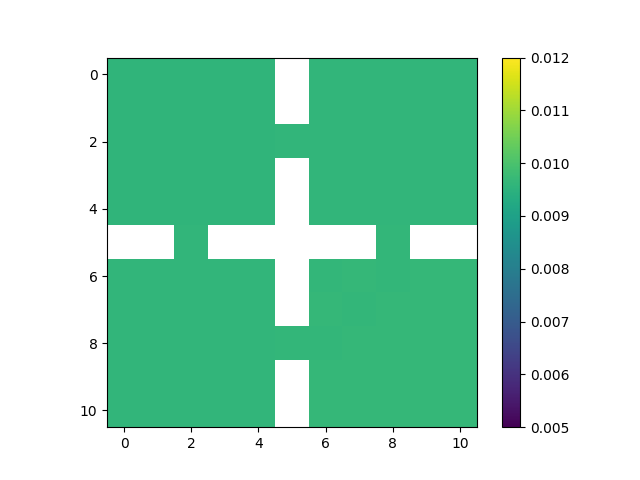}
         \caption{$n = 40, k=500$}
     \end{subfigure}
     \hfill
     \begin{subfigure}[h]{0.3\textwidth}
         \centering
         \includegraphics[width=\textwidth]{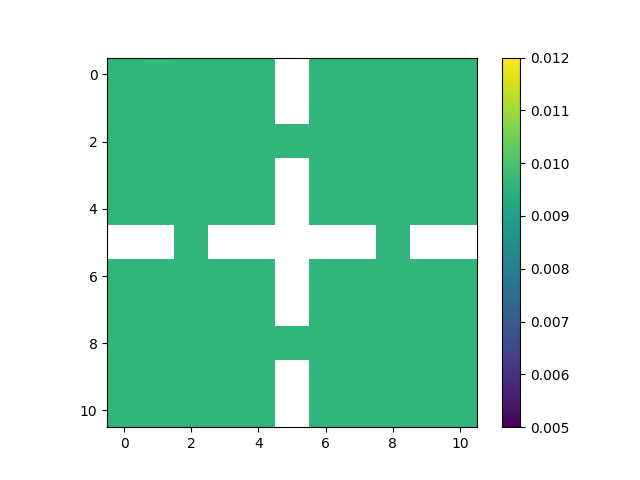}
         \caption{$n = 100, k=500$}
     \end{subfigure}
     \caption{State distribution of MD-CURL applied to Entropy Maximisation for steps $n \in \{10, 40, 100\}$ and iterations $k \in \{10, 50, 500\}$.}
     \label{MD-CURL_maxentropy_knownp}
\end{figure}

\begin{figure}[H]
     \centering
     \begin{subfigure}[h]{0.3\textwidth}
         \centering
         \includegraphics[width=\textwidth]{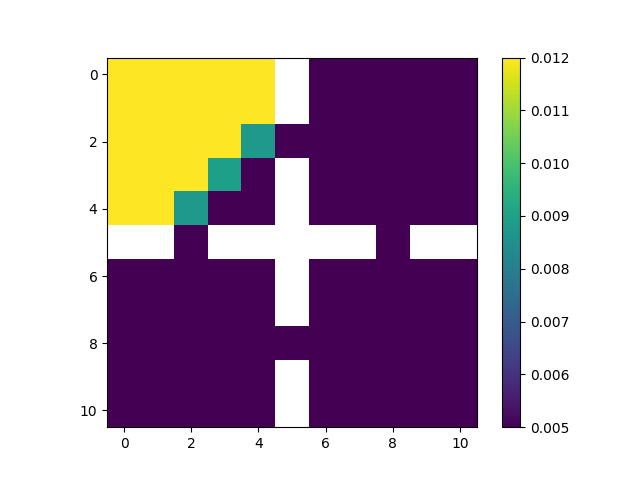}
         \caption{$n = 10, k=10$}
     \end{subfigure}
     \hfill
     \begin{subfigure}[h]{0.3\textwidth}
         \centering
         \includegraphics[width=\textwidth]{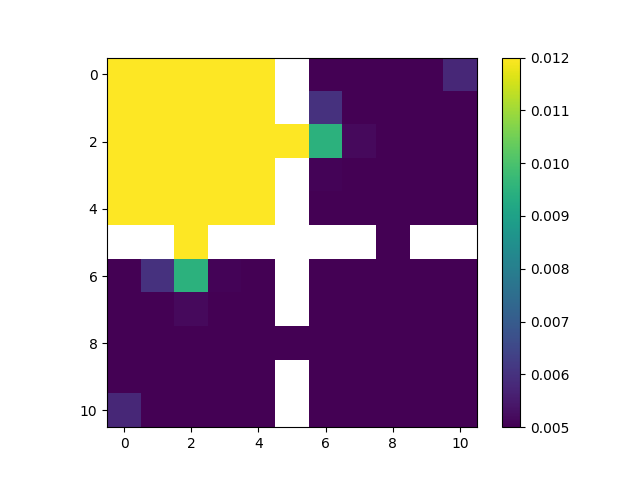}
         \caption{$n = 40, k=10$}
     \end{subfigure}
     \hfill
     \begin{subfigure}[h]{0.3\textwidth}
         \centering
         \includegraphics[width=\textwidth]{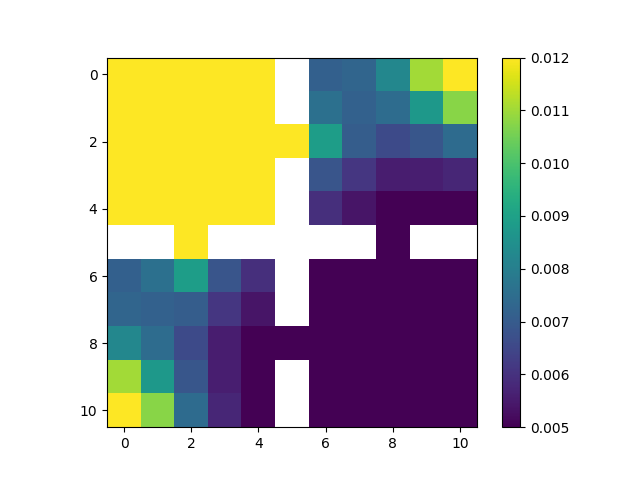}
         \caption{$n = 100, k=10$}
     \end{subfigure}
          \begin{subfigure}[h]{0.3\textwidth}
         \centering
         \includegraphics[width=\textwidth]{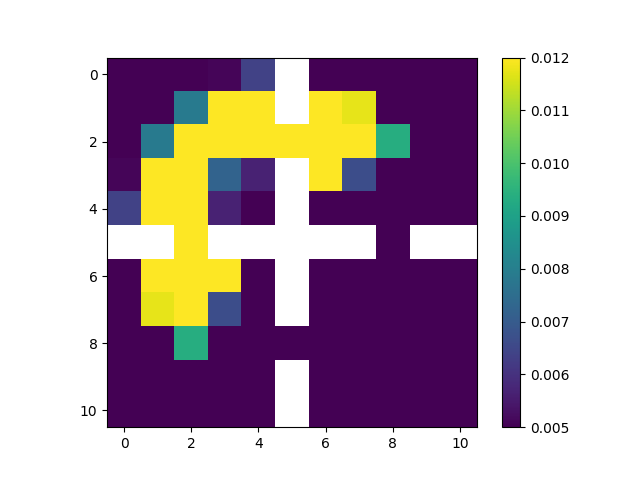}
         \caption{$n = 10, k=50$}
     \end{subfigure}
     \hfill
     \begin{subfigure}[h]{0.3\textwidth}
         \centering
         \includegraphics[width=\textwidth]{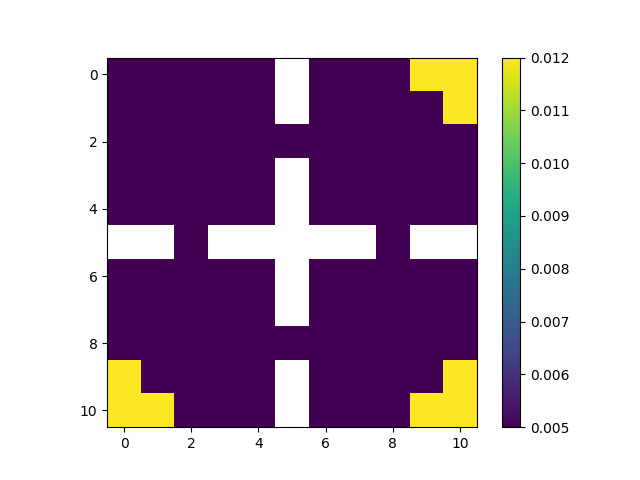}
         \caption{$n = 40, k=50$}
     \end{subfigure}
     \hfill
     \begin{subfigure}[h]{0.3\textwidth}
         \centering
         \includegraphics[width=\textwidth]{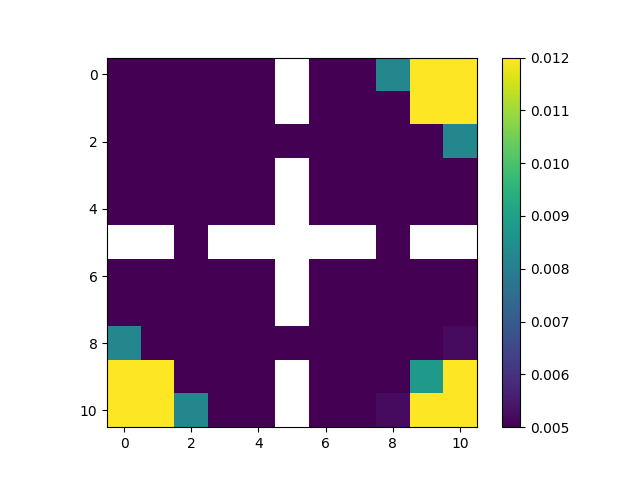}
         \caption{$n = 100, k=50$}
     \end{subfigure}
             \begin{subfigure}[h]{0.3\textwidth}
         \centering
         \includegraphics[width=\textwidth]{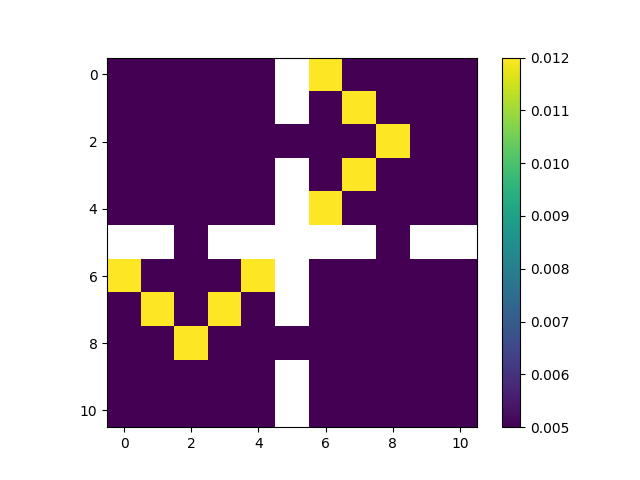}
         \caption{$n = 10, k=500$}
     \end{subfigure}
     \hfill
     \begin{subfigure}[h]{0.3\textwidth}
         \centering
         \includegraphics[width=\textwidth]{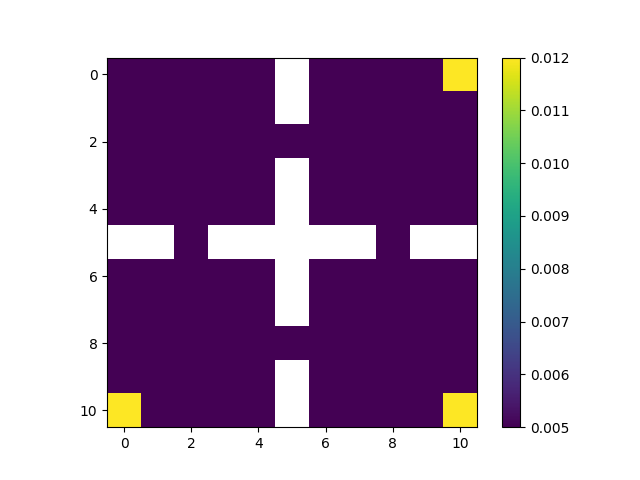}
         \caption{$n = 40, k=500$}
     \end{subfigure}
     \hfill
     \begin{subfigure}[h]{0.3\textwidth}
         \centering
         \includegraphics[width=\textwidth]{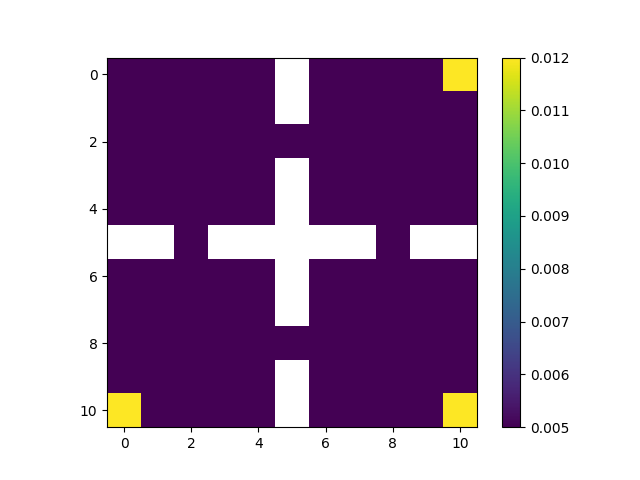}
         \caption{$n = 100, k=500$}
     \end{subfigure}
     \caption{State distribution of MD-CURL applied to Multi-Objectives for steps $n \in \{10, 40, 100\}$ and iterations $k \in \{10, 50, 500\}$.}
     \label{MD-CURL_multiobj_knownp}
\end{figure}

\subsection{Greedy MD-CURL with completely unknown probability kernel}

In this section, we present the state distribution induced by the policies computed with Greedy MD-CURL in the online learning scenario. We assume that both $g_n$ and $h_n$ are unknown, and we estimate the probability kernel $\hat{p}^t$ using Equation~\eqref{proba_counts} at each episode. We vary the steps $n$ and episodes $t$, and fix the episode length to $N=100$ for all experiments.

We illustrate the Entropy Maximization problem in Figure~\ref{fig:max_entropy_iter500_greedy} and the Multi-Objective problem in Figure~\ref{fig:multiobj_iter500_greedy} with a central noise of a probability of $0.2$. These results show that even when the full dynamics are unknown, Greedy MD-CURL can still achieve the main goal.

\begin{figure}[H]
     \centering
     \begin{subfigure}[h]{0.3\textwidth}
         \centering
         \includegraphics[width=\textwidth]{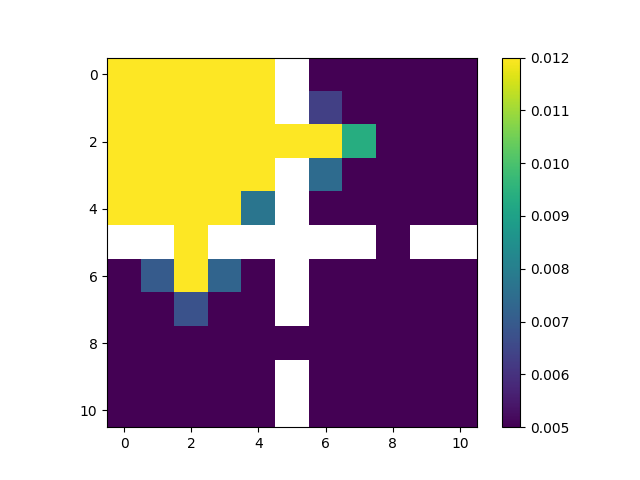}
         \caption{$n = 10, t=10$}
     \end{subfigure}
     \hfill
     \begin{subfigure}[h]{0.3\textwidth}
         \centering
         \includegraphics[width=\textwidth]{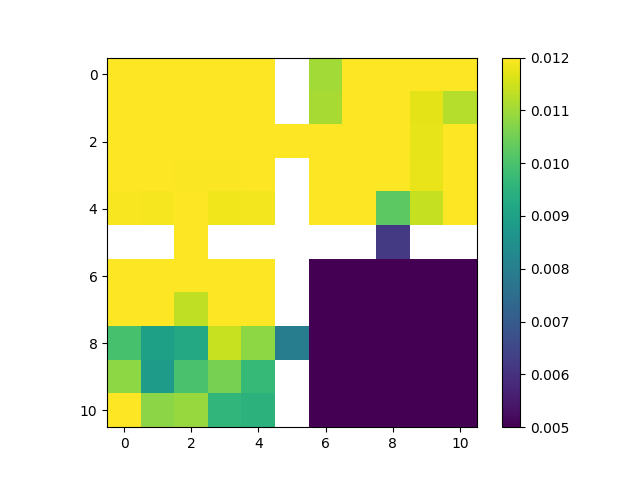}
         \caption{$n = 40, t=10$}
     \end{subfigure}
     \hfill
     \begin{subfigure}[h]{0.3\textwidth}
         \centering
         \includegraphics[width=\textwidth]{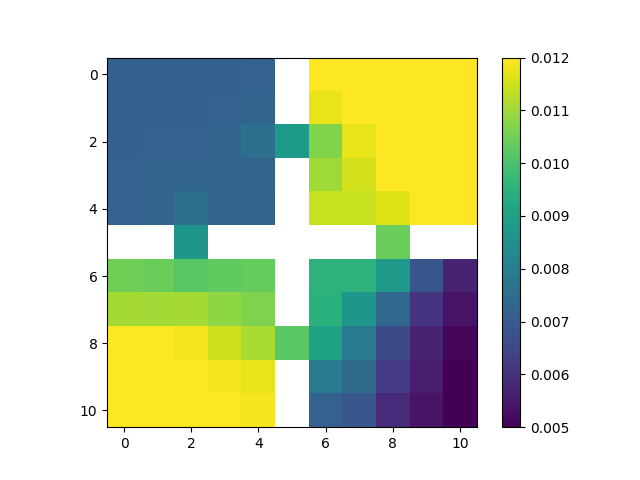}
         \caption{$n = 100, t=10$}
     \end{subfigure}
     \begin{subfigure}[h]{0.3\textwidth}
         \centering
         \includegraphics[width=\textwidth]{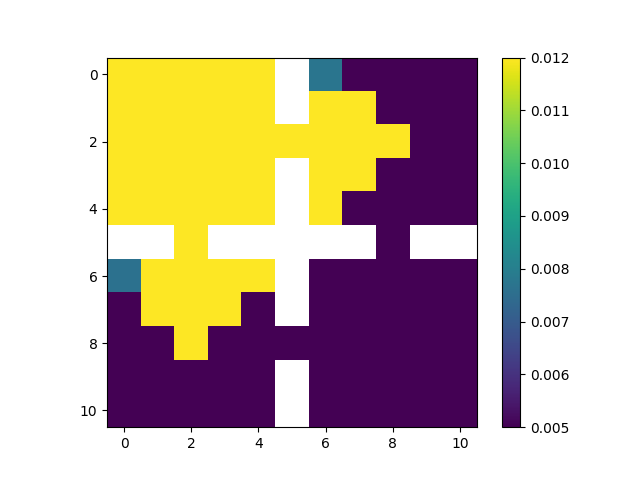}
         \caption{$n = 10, t=50$}
     \end{subfigure}
     \hfill
     \begin{subfigure}[h]{0.3\textwidth}
         \centering
         \includegraphics[width=\textwidth]{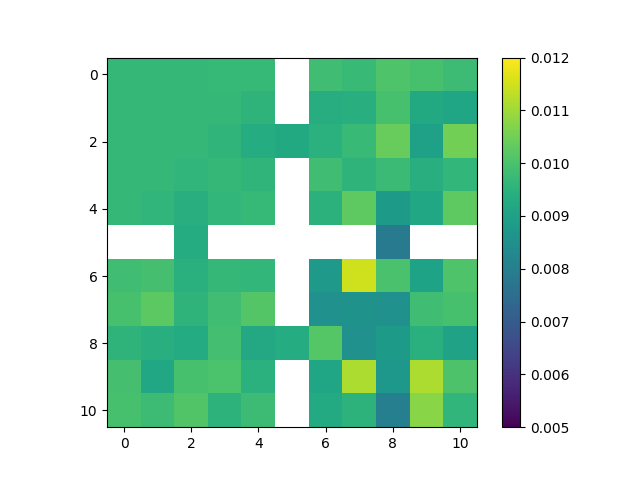}
         \caption{$n = 40, t=50$}
     \end{subfigure}
     \hfill
     \begin{subfigure}[h]{0.3\textwidth}
         \centering
         \includegraphics[width=\textwidth]{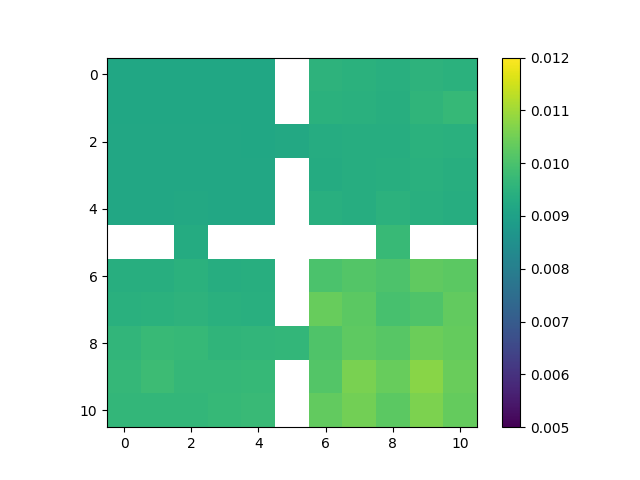}
         \caption{$n = 100, t=50$}
     \end{subfigure}
     \begin{subfigure}[h]{0.3\textwidth}
         \centering
         \includegraphics[width=\textwidth]{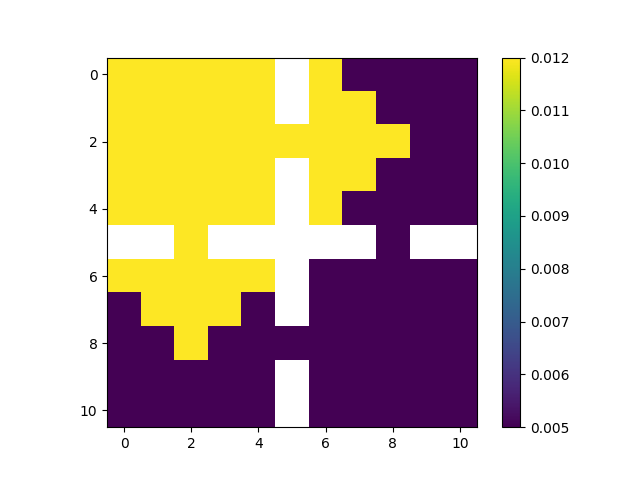}
         \caption{$n = 10, t=500$}
     \end{subfigure}
     \hfill
     \begin{subfigure}[h]{0.3\textwidth}
         \centering
         \includegraphics[width=\textwidth]{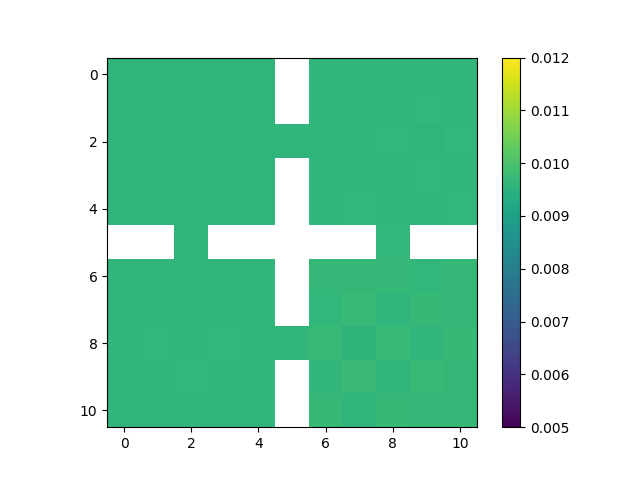}
         \caption{$n = 40, t=500$}
     \end{subfigure}
     \hfill
     \begin{subfigure}[h]{0.3\textwidth}
         \centering
         \includegraphics[width=\textwidth]{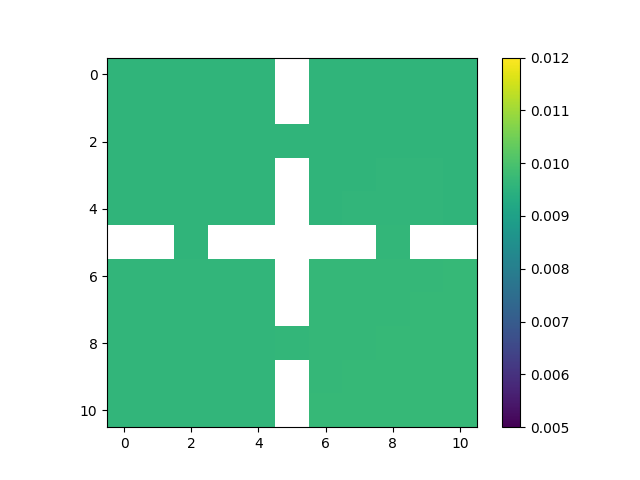}
         \caption{$n = 100, t=500$}
     \end{subfigure}
        \caption{State distribution of Greedy MD-CURL applied to Entropy Maximisation for steps $n \in \{10, 40, 100\}$ and episodes $t \in \{10, 50, 500\}$.}
        \label{fig:max_entropy_iter500_greedy}
\end{figure}

\begin{figure}[H]
     \centering
     \begin{subfigure}[h]{0.3\textwidth}
         \centering
         \includegraphics[width=\textwidth]{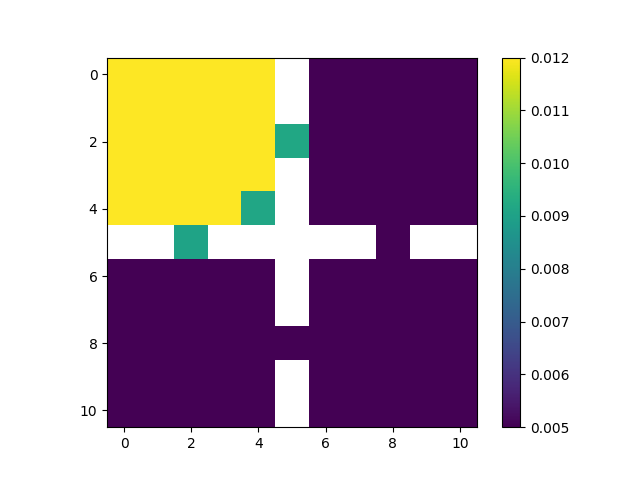}
         \caption{$n = 10, t=10$}
     \end{subfigure}
     \hfill
     \begin{subfigure}[h]{0.3\textwidth}
         \centering
         \includegraphics[width=\textwidth]{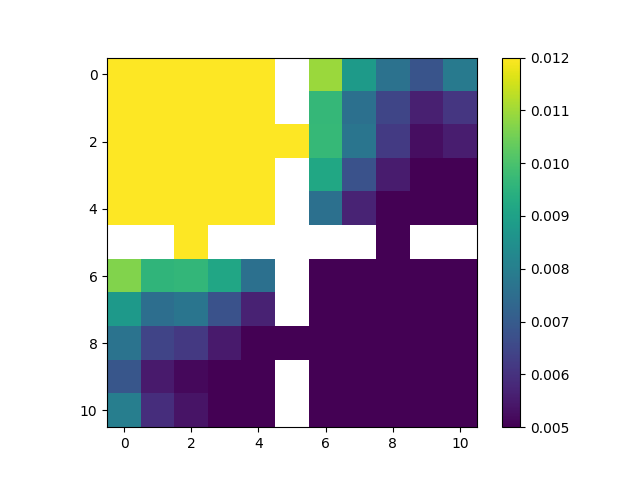}
         \caption{$n = 40, t=10$}
     \end{subfigure}
     \hfill
     \begin{subfigure}[h]{0.3\textwidth}
         \centering
         \includegraphics[width=\textwidth]{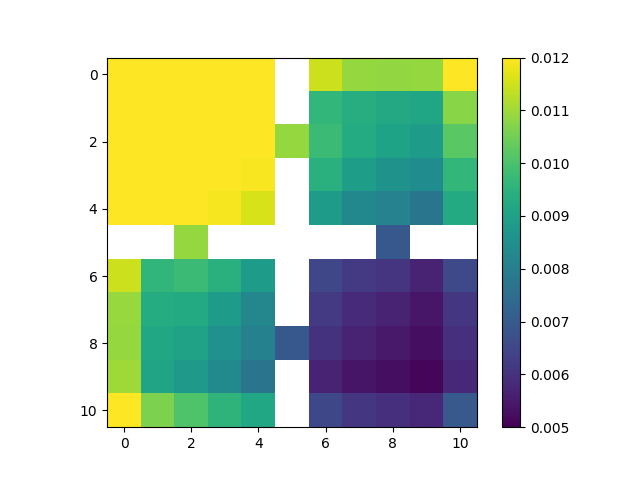}
         \caption{$n = 100, t=10$}
     \end{subfigure}
     \begin{subfigure}[h]{0.3\textwidth}
         \centering
         \includegraphics[width=\textwidth]{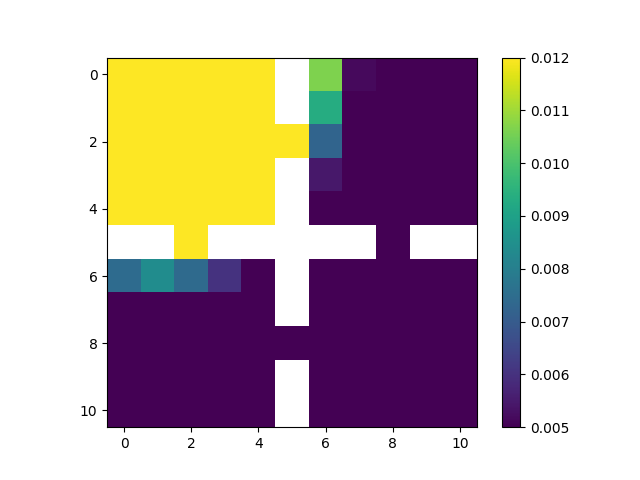}
         \caption{$n = 10, t=50$}
     \end{subfigure}
     \hfill
     \begin{subfigure}[h]{0.3\textwidth}
         \centering
         \includegraphics[width=\textwidth]{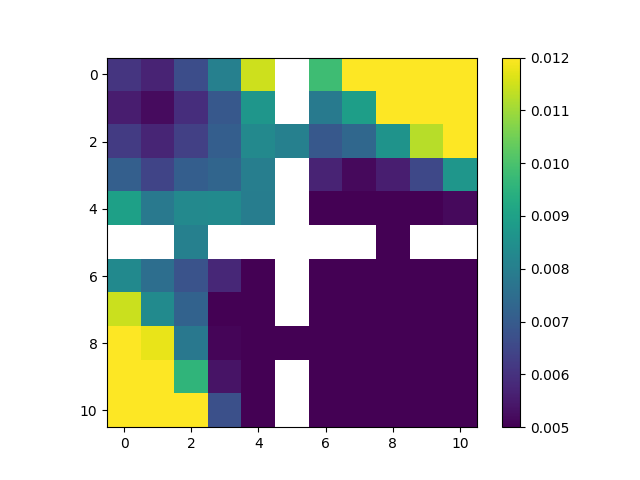}
         \caption{$n = 40, t=50$}
     \end{subfigure}
     \hfill
     \begin{subfigure}[h]{0.3\textwidth}
         \centering
         \includegraphics[width=\textwidth]{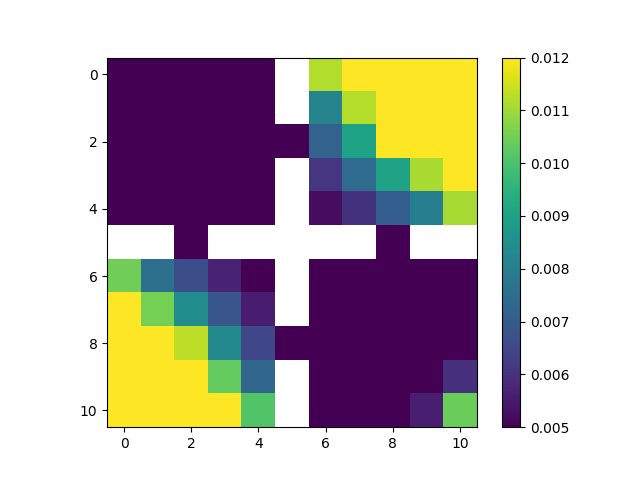}
         \caption{$n = 100, t=50$}
     \end{subfigure}
     \begin{subfigure}[h]{0.3\textwidth}
         \centering
         \includegraphics[width=\textwidth]{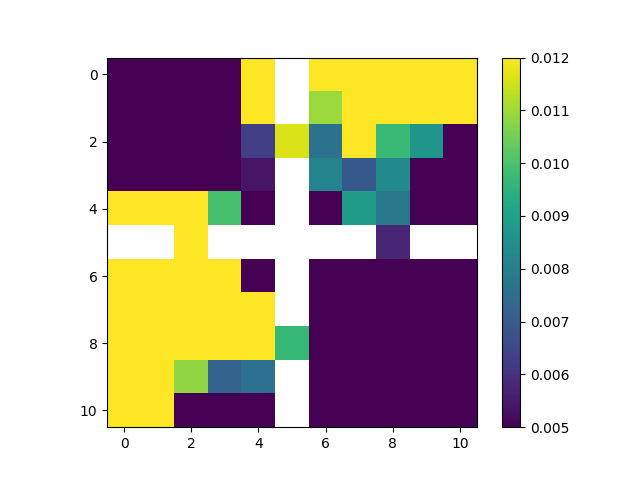}
         \caption{$n = 10, t=500$}
     \end{subfigure}
     \hfill
     \begin{subfigure}[h]{0.3\textwidth}
         \centering
         \includegraphics[width=\textwidth]{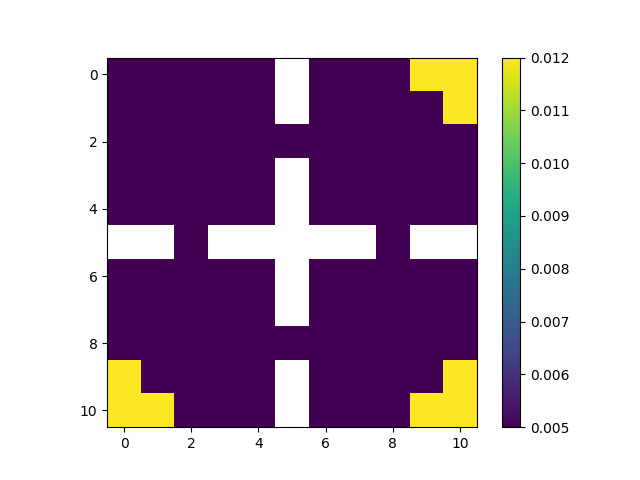}
         \caption{$n = 40, t=500$}
     \end{subfigure}
     \hfill
     \begin{subfigure}[h]{0.3\textwidth}
         \centering
         \includegraphics[width=\textwidth]{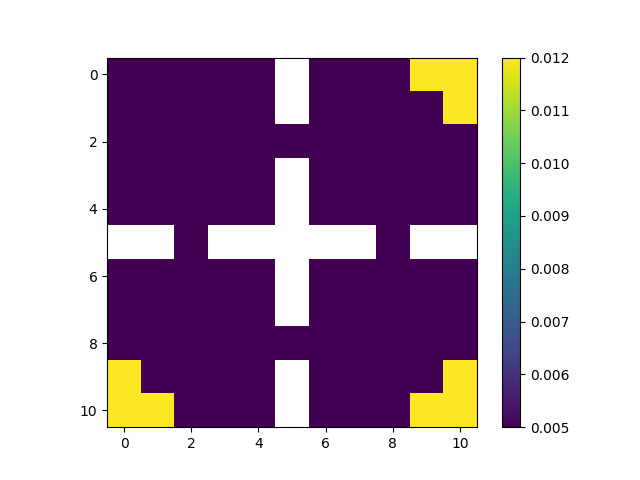}
         \caption{$n = 100, t=500$}
     \end{subfigure}
        \caption{State distribution of Greedy MD-CURL applied to Multi-Objectives for steps $n \in \{10, 40, 100\}$ and episodes $t \in \{10, 50, 500\}$.}
        \label{fig:multiobj_iter500_greedy}
\end{figure}
\vfill

\end{document}